\documentclass[12pt,reqno,a4paper]{amsart}
\usepackage{latexsym,bm}
\usepackage{amsmath,amssymb,cases}
\usepackage{amsthm}
\usepackage{xcolor}
\usepackage{enumerate}
\usepackage{caption}
\usepackage{graphicx,subfig}
\usepackage[numbers,sort&compress]{natbib}
\usepackage{accents}

\usepackage{geometry}
\usepackage{float}

\usepackage{ulem}
\usepackage[thicklines]{cancel}
\usepackage{relsize}
\usepackage{exscale}

\newtheorem{The}{Theorem}[section]
\newtheorem{Lem}[The]{Lemma}

\newtheorem{Def}[The]{Definition}
\newtheorem{Pro}[The]{Proposition}

\theoremstyle{definition}

\graphicspath{{fig/}}

\numberwithin{equation}{section}
\numberwithin{figure}{section}

\begin{document}
\captionsetup[figure]{labelfont={default},labelformat={default},labelsep=period,name={Fig.}}

\title[Relaxation limit of pressureless Euler-Poisson equations]
{Relaxation limit of pressureless Euler-Poisson equations
}

\author{Guirong Tang}
\address{Guirong Tang: School of Mathematical Sciences, Capital Normal University, Beijing 100048, China}
\email{\tt tangguirong@amss.ac.cn}

\begin{abstract}
We investigate the relaxation problem for the one-dimensional pressureless Euler--Poisson equations with the initial density being a finite Radon measure.
The entropy solution of this linearly degenerate hyperbolic system converges to the weak solution of part of drift--diffusion equations without diffusion term when the momentum relaxation time tends to zero.
Independent of compactness, our strategy is to construct the formula of solutions to both systems and directly obtain the convergent result.
Furthermore, the uniqueness of entropy solution to pressureless Euler--Poisson equations is also proved by Oleinik condition and the initially weak continuity of kinetic energy measure.
\end{abstract}

\date{\today}
\maketitle

\tableofcontents
 
\thispagestyle{empty}

\section{Introduction}
We are concerned with the following one-dimensional pressureless Euler--Poisson equations with relaxation term:
\begin{equation}\label{maineq}
\begin{cases}
\rho_t +(\rho u) _x=0,\\[1mm]
(\rho u)_t+(\rho u^2)_x=\kappa\rho \Phi _x-\frac{\rho u}{\tau},\\[1mm]  
\Phi_{xx}=\rho
\end{cases}
\end{equation}
for $ (x,t) \in \mathbb{R}\times\mathbb{R}^+:=({-}\infty,\infty)\times(0,\infty)$.   
Here $\rho\geq 0$ is the density of mass, $u$ denotes the velocity, and the parameter $\kappa<0$.
For the Poisson source term $\kappa\rho \Phi _x$, 
$\Phi$ represents the gravitational potential of the dusty stars when $\kappa<0$ 
(for the case that $\kappa>0$, $\Phi$ represents the potential of plasma electric fields).
Since the parameter $\kappa<0$ can be rescaled to be any negative number, 
without loss of generality, we take $\kappa={-}1$ throughout this paper.
$\tau\in (0,1]$ is the relaxation time which means the average time required for the state being relaxed to the equilibrium.   

\vspace{1pt}
Since system $\eqref{maineq}$ is linearly degenerate hyperbolic, combining with the attractive Poisson force, then
no matter how smooth the initial data is (even for constant initial data), 
the solutions develop singularities and form $\delta$-shocks generically in a finite time.
Thus, it is natural and necessary to understand the solutions in the sense of Radon measures.
We consider the general Cauchy initial data as follows:
\begin{equation}\label{ID}
(\rho,u)|_{t=0}=(\rho_0,u_0) \in (\mathcal{M}(\mathbb{R}),L^{\infty}_{\rho_0}(\mathbb{R})),
\qquad\,\,\,
M:=\int_{\mathbb{R}} \rho_0\, {\rm d}x<\infty,
\end{equation}
where $\mathcal{M}(\mathbb{R})$ is the space of finite Radon measures
and $L^{\infty}_{\rho_0}(\mathbb{R})$ represents the space of bounded measurable functions 
with respect to the measure $\rho_0$.

\smallskip
The relaxation limit is well-known in asymptotic analysis and singular perturbation theories for many physical models with relaxation time \cite{GBW}.
As the relaxation time tends to zero, the limit can be from first-order hyperbolic systems to hyperbolic system (see \cite{Liu1987, Na1997} and references therein), 
or to parabolic system by slow time scaling \cite{MR2000}.
To the best of our knowledge, there is no result on relaxation limit from linearly degenerate hyperbolic system. 

The slow time scaling is $t=\tau t'$. If we introduce 
\begin{align}\label{2.0}
\rho^{\tau}(x,t)=\rho(x,\frac{t}{\tau}),\quad u^{\tau}(x,t)=\frac{1}{\tau}u(x,\frac{t}{\tau}),\quad \Phi^{\tau}(x,t)=\Phi(x,\frac{t}{\tau}),
\end{align}
then the degenerate hyperbolic system \eqref{maineq}--\eqref{ID} ($\kappa=-1$) transforms into
\begin{equation}\label{2.1}
\begin{cases}
\rho^{\tau}_t +(\rho^{\tau} u^{\tau})_x=0,\\[1mm]
\tau^2\big((\rho^{\tau} u^{\tau})_t+(\rho^{\tau} (u^{\tau})^2)_x \big)=-\rho^{\tau} \Phi^{\tau}_x-\rho^{\tau} u^{\tau} ,\\[1mm]  
\Phi^{\tau}_{xx}=\rho^{\tau}.
\end{cases}
\end{equation}
with initial data
\begin{align}
\rho^{\tau}_0(x):=\rho^{\tau}(x,0)=\rho_0(x), \qquad
u^{\tau}_0(x):=u^{\tau}(x,0)=\frac{1}{\tau} u_0(x).
\end{align}

Formally, as $\tau \rightarrow 0$, if denoting the limits of $(\rho^{\tau},  u^{\tau}, \Phi^{\tau})$ as $(\bar{\rho},\bar{u}, \bar{\Phi})$, 
one gets that the formal limiting system of \eqref{2.1} is
\begin{align}\label{2.2}
\begin{cases}
\bar{\rho}_t+(\bar{\rho} \, \bar{u})_x=0,\\[1mm]
\bar{\rho}\, (\bar{\Phi}_x+\bar{u})=0,\\[1mm]
\bar{\Phi}_{xx}=\bar{\rho}.
\end{cases}
\end{align}
Substituting $\eqref{2.2}_2$ into $\eqref{2.2}_1$, as a result, \eqref{2.2} yields a new model 
\begin{align}\label{2.2a}
\begin{cases}
\bar{\rho}_t-(\bar{\rho}\, \bar{\Phi}_x)_x=0,\\[1mm]
\bar{\Phi}_{xx}=\bar{\rho}.
\end{cases}
\end{align}

To better understand this model, we introduce the  drift--diffusion equations
\begin{align}\label{2.3}
\begin{cases}
\rho_t+(P_x-\rho \Phi_x)_x=0,\\[1mm]
\Phi_{xx}=\rho,
\end{cases}
\end{align}
where $P$ is the pressure which is a function of density $\rho$. 
The classical drift--diffusion equations is to explain the dynamics of electrons with the interaction force being repulsive,
which is different with our attractive model \eqref{2.3}, so that the signs of $\rho \Phi_x$ and $P_x$ are reversed. 
In semiconductor, the current density $P_x-\rho \Phi_x$ is the sum of drift current $-\rho \Phi_x$ and the diffusion current $P_x$, which explains the name of the model \cite{JA2009}.

Since we consider the problem in pressureless backward, the pressure term $P$ is vanishing, so that the diffusion term is zero, and the drift--diffusion equations \eqref{2.3} transforms into \eqref{2.2a}, which is called ``{\it drift equations}" in this paper.

\smallskip

Now we state some related results before.
Concisely, for pressureless Euler system, it is well-known that Huang-Wang \cite{huang2001well} construct the formula of entropy solution and prove its uniqueness by Oleinik condition with the weak continuity of energy measure.
For pressureless Euler system with damping, Ha-Huang-Wang \cite{HHW2014} get the same result as above. 
Their unique result does not depend on the weak continuity of energy measure because the initial density is a Lebesgue-measurable function.
After two years, by adding the weak continuity condition, Jin \cite{Jin2016} extends their result to the case that initial density is a Radon measure.
For pressureless Euler--Poisson system without relaxation, Cao-Huang-Tang \cite{CHT2025} prove the existence, uniqueness and asymptotic behaviour of entropy solution.

There are other methods to study the pressureless system, such as generalized variational principle \cite{E1996}, scalar conservation laws \cite{BG1998, LT2023, NT2008, NT2015}, Lagrangian solution with differential inclusion \cite{NS2009, BG2013}, metric projection \cite{CSW2015}, probability measure on trajectory \cite{H2020, ZQY2024}, and so on \cite{B1994, BJ1999, Z1970}. 

For relaxation limit, the main results concentrate on the Euler--Poisson system with pressure.
In one dimensional case, Marcati--Natalini \cite{MN1995} consider the model for semiconductors and prove the limit solutions satisfy the classical drift-diffusion equation by compensated compactness.
Alì--Bini--Rionero \cite{ABR2000} consider the problem with the initial data being a perturbations of a stationary solution of the drift-diffusion equations and get the decay rate.
Hsiao--Zhang \cite{HZ2000} consider the initial-boundary value problem by Godunov scheme.
For spherically symmetric case, D. H. Wang \cite{Wang2001} show the general large weak entropy solutions converge to the solution of a generalized drift-diffusion equation. 
For $3-D$ model, Lattazio--Marcati \cite{LM1999} prove the zero relaxation limit for isentropic case 
and Li-Peng-Zhao \cite{LPZ2021} get the convergence rate in torus for smooth periodic initial data sufficiently close to constant equilibrium state.

\smallskip

We will prove (Theorem \ref{thm3}) that the unique entropy solution of pressureless Euler--Poisson system \eqref{maineq} converges to the weak solution and the velocity of drift equations \eqref{2.2a} as relaxation time $\tau \rightarrow 0$.
Since the solutions to both systems are Radon measure, which is very singular, our strategy is to construct the formulas of solutions and analyze the convergence directly.

Motivated by Huang-Wang \cite{huang2001well}, we exploit the similar idea to construct the formula of entropy solution to pressureless Euler--Poisson equations with relaxation.
We introduce the first generalized potential $F(y;x,t)$ as in $\eqref{potentialF}$ 
to define the formula for the mass $m(x,t)$, velocity $u(x,t)$
and the momentum $q(x,t)$. 
As the collision of particles is inelastic, the machanical energy is not conserved, so that
we have to track the particle trajectory to formulate the well-defined energy.
Moreover, we need to deal with the Poisson source term and the relaxation term in the context of Radon measures, 
which is done via introducing the second and third generalized potentials $G(y;x,t)$ 
and $H(y;x,t)$ as in $\eqref{potentialG}$ and $\eqref{potentialH}$, respectively. 

Compared with the pressureless Euler--Poisson system without relaxation \cite{CHT2025}, the damping term can not prevent the formation of $\delta$-shock, but can bound the velocity in the whole space. 
Thus  our uniqueness result (Theorem \ref{UniThm}) lies in $\mathbb{R}\times \mathbb{R}^+$.
For the proof of the uniqueness, we show that any entropy solution can be expressed by the initial data and thus equals to the solution constructed in the existence theorem (Theorem \ref{ExisThm}).
Our proof also depends on the weak continuity of energy measure $\rho u^2$ at $t=0$, which is effective since the Poisson force is attractive.

The drift equations \eqref{2.2a} is absolutely different with the drift--diffusion equations \eqref{2.3}.
To the best of our knowledge, there is no result about the existence of solution to drift equations.
Since the pressure is vanishing, there is no force to avoid singularities, so that  Dirac measure exists in the weak solutions.
Based on this observation, we employ the similar methods to construct the formula of weak solution. 
Besides, we also define the velocity $\bar{u}(x,t)$ \eqref{2.13b} of drift equations, which is apparent from the equation \eqref{2.7}.

In order to prove the relaxation limit (Theorem \ref{thm3}), the key point is to show that the generalized potential $F(y;x,t)$ of Euler--Poisson system converges to the potential $\bar{F}(y;x,t)$ of drift equations by slow time scaling as relaxation time $\tau\rightarrow 0$. 
Then by the formula of solutions and the Radon-Nikodym derivative for momentum with respect to mass, the weak convergence of density and the convergence of velocity are directly obtained.

\vspace{2pt}
This paper is organized as follows: 
In \S 2.1 and \S 2.2, we outline the derivation of the formula for solutions 
to pressureless Euler--Poisson system and drift equations respectively.
The theorems about existence or uniqueness are also stated here.
In \S 2.3, we present the main theorem about the relaxation limit.
In \S 3, we rigorously derive the formula for entropy solutions to pressureless Euler--Poisson system
via constructing the mass, velocity, momentum, and energy 
by using the generalized potentials, and prove the existence theorem. 
To maintain coherence in reading and emphasize the relaxation limit, the proof of uniqueness is placed in the final section \S 6.
In \S 4, we prove the formula of weak solutions to drift equations.
In \S 5, the proof of relaxation limit is presented.

\section{Formula of Solutions and Main Theorems}
In this section, we present the idea to construct the formula of solutions and show main  theorems. 
Concisely, in \S 2.1, the cumulative distribution function $m(x,t)$ is introduced to rewrite the Euler--Poisson equations, then the definitions of weak solution and entropy solution are present.
From the formula of smooth characteristic, we construct the generalized potential $F(y;x,t)$ and analyze its minimum point, which is the key to define the mass, momentum, and energy.
At last, we present the main theorems about existence and uniqueness.
In \S 2.2, we construct the weak solution of drift equations by similar methods and present the existence theorem.
In \S 2.3, we define the physical velocity of drift equations and show the relaxation limit of density as well as velocity in Theorem \ref{thm3}.

\subsection{Formula of solutions to pressureless Euler--Poisson equations}
Firstly, we determine the gravitational potential $\Phi$ 
via $m(x,t)$, which is the cumulative distribution function of $\rho$.
From $\eqref{maineq}$--$\eqref{ID}$, when $\rho$ is Lebesgue integrable, 
the mass $m(x,t)$ is given by
$$m(x,t)=\int_{{-}\infty}^{x}\rho(\eta,t) \,{\rm d}\eta\in [0,M].$$
It follows from the Poisson equation $\eqref{maineq}_3$ that, $\Phi_x$ can be realized by \begin{equation*}
\Phi_x(x,t)=\frac{1}{2}\Big(\int_{{-}\infty}^{x} \rho(\eta,t)\,{\rm d}\eta
-\int_{x}^{\infty}\rho(\eta,t)\,{\rm d}\eta\Big)
=m(x,t)-\frac{1}{2}M.
\end{equation*}

In general, for the case that $\rho \in \mathcal{M}(\mathbb{R})$, 
by the relation of $\Phi_x(x,t)$ and $m(x,t)$ as above, 
it is consistent to let 
\begin{equation}\label{mPhix}
m(x,t):= \int_{{-}\infty}^{x{-}} \rho(\eta,t)\, {\rm d}\eta, \qquad\,\,\,
\Phi_x(x,t):=\frac{1}{2}\Big(\int_{{-}\infty}^{x{-}}
-\int_{x{+}}^{\infty}\Big) \rho(\eta,t)\,{\rm d}\eta;
\end{equation}
since $\partial_x m(x,t)=\partial_{xx}\Phi=\rho$ in the sense of BV derivative ({\it cf.} \cite{vol1967spaces})
and $\Phi_x({\pm}\infty)={\pm}\frac{M}{2}$, 
it is direct to check from $\eqref{mPhix}$ that 
\begin{equation}\label{widem}
\Phi_x(x,t)=\frac{1}{2}\big(m(x{-},t)+m(x{+},t)-M\big)=:\widetilde{m}(x,t).
\end{equation}
Then, by using $\eqref{mPhix}$--$\eqref{widem}$, 
system $\eqref{maineq}$ transforms into 
\begin{equation}\label{mmaineq1}
\begin{cases}
(m_x)_t+(m_x u)_x=0,\\[1mm]
(m_x u)_t+(m_x u^2)_x=-\widetilde{m}m_x-\frac{u m_x}{\tau}.
\end{cases}
\end{equation} 
For $\eqref{mmaineq1}_1$, by multiplying a test function $\psi\in C_c^{\infty}(\mathbb{R}\times\mathbb{R}^+)$ and integration by parts for twice, 
\begin{align*}
\iint \psi_{xt} m \,{\rm d}x {\rm d}t-\iint \psi_x u m_x \,{\rm d}x {\rm d}t=0.
\end{align*}
Thus if we denote $\varphi=\psi_x$, the system \eqref{mmaineq1} can also formally be written as 
\begin{equation}\label{mmaineq}
\begin{cases}
m_t+u m_x=0,\\[1mm]
(m_x u)_t+(m_x u^2)_x=-\widetilde{m}m_x-\frac{u m_x}{\tau},
\end{cases}
\end{equation} 
and their solutions in the sense of distribution are equivalent.

\smallskip
Motivated by Huang--Wang \cite{huang2001well}, we use the Lebesgue-Stieltjes integral to define the weak solutions, 
and use the Oleinik entropy condition associating with the initial weak continuity of energy to define the entropy solutions.

\begin{Def}[Weak Solution I]\label{weakdef} 
Let $m(x,t) \in [0,M]$ be nondecreasing in $x$ and 
$u(x,t)$ be bounded and measurable to $m_x$.
Assume that the measures $m_x$ and $um_x$ are weakly continuous in $t$. 
Then, $(\rho,u)=(m_x,u)$ $(resp., (m,u))$ is called to be a weak solution of system $\eqref{maineq}$ $(resp.,\, $\eqref{mmaineq}$)$ 
if, for any $\varphi,\psi \in C_c^{\infty}(\mathbb{R}\times\mathbb{R}^+)$,  
\begin{align}\label{weakeq}
\begin{cases}
\displaystyle\iint \varphi _t m \,{\rm d}x {\rm d}t-\iint \varphi u \,{\rm d}m {\rm d}t=0,\\[4mm]
\displaystyle\iint \psi_tu+\psi_xu^2 \,{\rm d}m {\rm d}t=\iint (\widetilde{m}+\frac{u}{\tau}) \psi \,{\rm d}m {\rm d}t;
\end{cases}
\end{align}
and as $t\rightarrow 0{+}$, 
the measures $\rho=m_x$ and $\rho u=um_x$ weakly converge to $\rho_0$ and $\rho_0 u_0$ 
respectively. 
Here, the $\iint \cdots {\rm d}m{\rm d}t$ denotes the Lebesgue-Stieltjes integral 
{\rm(}{\it cf.} {\rm \cite{carter2000lebesgue}}{\rm)}.
\end{Def}

\begin{Def}[Entropy Solution]\label{entropylem} 
A weak solution $(\rho,u)$ is called to be an entropy solution of system $\eqref{maineq}$ 
if the following two hold{\rm:}
\begin{itemize}
\item [(i)] For any $x_1<x_2$ and {\it a.e.} $t>0$, 
\begin{equation}\label{ux2ux1}
\frac{u(x_2,t)-u(x_1,t)}{x_2-x_1} \leq \frac{1}{t};
\end{equation}
    
\item [(ii)] The measure $\rho u^2$ weakly converges to $\rho_0 u_0^2$ as $t \rightarrow 0{+}$.
\end{itemize}
\end{Def}

\vspace{2pt}
If we consider the smooth solutions of \eqref{mmaineq}, then the path $X(\eta,t)$ of the particle originated from $\eta$ with initial velocity $u_0(\eta)$ satisfies
\begin{align}\label{2.1c}
\begin{cases}
\frac{{\rm d}X(\eta,t)}{{\rm d}t}=u(X(\eta,t),t),  \\[1mm]
\frac{{\rm d}u(X(\eta,t),t)}{{\rm d}t}=-\widetilde{m}(X(\eta,t),t)-\frac{u(X(\eta,t),t)}{\tau}, \\[1mm]
X(\eta,0)=\eta,\quad u(\eta,0)=u_0(\eta).
\end{cases}
\end{align}
Along the path $X(\eta,t)$, it follows from $\eqref{mmaineq}_1$ that $m(x,t)$ keeps constant, then $m(X(\eta,t),t)=m_0(\eta):=\int_{-\infty}^{y}\rho_0(\xi)\, {\rm d}\xi$ and $\widetilde{m}(X(\eta,t),t)=\widetilde{m}_0(\eta):=\frac{1}{2}\big(m_0(\eta{+}){+}m_0(\eta{-}){-}M\big)$.
Substituting into \eqref{2.1c} and solving the ODE of $X(\eta,t)$, we have
\begin{align*}
X(\eta,t)=\eta+u_0(\eta)(\tau-\tau e^{-\frac{t}{\tau}})+ \widetilde{m}_0(\eta)(\tau^2-\tau^2 e^{-\frac{t}{\tau}}-\tau t).
\end{align*}

This motivates us to define the generalized potential $F(y;x,t)$ as follows:
For any fixed $(x,t)\in \mathbb{R}\times \mathbb{R}^+$, 
\begin{equation}\label{potentialF}
 F(y;x,t):=\int_{{-}\infty}^{y{-}}
 \eta +u_0(\eta)(\tau-\tau e^{-\frac{t}{\tau}})+ \widetilde{m}_0(\eta)(\tau^2-\tau^2 e^{-\frac{t}{\tau}}-\tau t)-x \,{\rm d}m_0(\eta),
\end{equation}
where $m_0(x)$ and $\widetilde{m}_0(x)$ are respectively given by
\begin{equation}\label{m0m0}
 m_0(x)=\int_{{-}\infty}^{x{-}}\rho_0(\eta) \,{\rm d}\eta,\qquad\,\,\, \widetilde{m}_0(x)=\frac{1}{2}\big(m_0(x{-})+m_0(x{+})-M\big).
\end{equation} 
By the nondecreasing property of $m_0(x)$ in $x$, if $[m_0(x_0)]>0$, then
\begin{equation}
\widetilde{m}_0(x_0{-})< \widetilde{m}_0(x_0)< \widetilde{m}_0(x_0{+}),
\end{equation}
which results in that there exists three different types of characteristics at some points with positive mass.

Denote $U_0=||u_0(x)||_{L^{\infty}_{\rho_0}(\mathbb{R})}$,
since $m_0(x)$ and $\widetilde{m}_0(x)$ are both nondecreasing 
and bounded by $[0,M]$ and $[-\frac{M}{2},\frac{M}{2}]$ respectively,
then for any fixed $(x,t)\in \mathbb{R}\times \mathbb{R}^+$,  
\begin{align*}
&F(y_1;x,t)-F(y_2;x,t)\\[1mm]
&=
\int_{y_2{-}}^{y_1{-}} \big(\eta +u_0(\eta)(\tau-\tau e^{-\frac{t}{\tau}})+ \widetilde{m}_0(\eta)(\tau^2-\tau^2 e^{-\frac{t}{\tau}}-\tau t)-x\big)\,{\rm d}m_0(\eta) \geq 0
\end{align*}
holds for any $y_1$ and $y_2$ satisfying 
\begin{align}
y_1 < y_2 \leq x-U_0 (\tau-\tau e^{-\frac{t}{\tau}})+\frac{M}{2} (\tau^2-\tau^2 e^{-\frac{t}{\tau}}-\tau t),\nonumber\\[1mm]
{\text or}\quad y_1 > y_2 \geq x+U_0 (\tau-\tau e^{-\frac{t}{\tau}})-\frac{M}{2} (\tau^2-\tau^2 e^{-\frac{t}{\tau}}-\tau t),\label{2.1b}
\end{align}
where $\tau^2-\tau^2 e^{-\frac{t}{\tau}}-\tau t\leq 0$ for $t\geq 0$.
This means that function $F(\cdot\,;x,t)$ has a finite low bound in $y\in \mathbb{R}$.
Thus, we let 
\begin{equation}\label{nuSeq}
\begin{cases}
\nu(x,t):=\mathop{\min}\limits_{y\in\mathbb{R}} F(y;x,t),\\[2mm]
S(x,t):=\{y\,:\, \exists\, y_n \rightarrow y\quad  {\rm s.t.}\ F(y_n;x,t) \rightarrow \nu(x,t)\};
\end{cases}
\end{equation}
and then define $y_*(x,t)$ and $y^*(x,t)$ by
\begin{align}\label{y*eq}
y_*(x,t):=\inf \big\{S(x,t) \cap {\rm spt}\{\rho_0\}  \big\},\qquad\,\,\,
y^*(x,t):=\lim_{\epsilon \rightarrow 0{+}} y_*(x{+}\epsilon,t).
\end{align}
As shown in Lemma $\ref{convergelem}$, $S(x,t)$ is closed so that
$$y_*(x,t),\, y^*(x,t) \in S(x,t) \cap {\rm spt}\{\rho_0\}.$$
Then, by using the backward characteristic areas from $(x,t)$ determined by $y_*(x,t)$ and $y^*(x,t)$, 
for any point $(x_0,t_0)$ with $t_0>0$, there exists a unique forward generalized characteristic $x(t)$ with $x(t_0)=x_0$.

\vspace{2pt}
{\it The formula for entropy solutions $(m,u)$, or equivalently $(m_x,u)$,} is defined by
\begin{eqnarray}
\!\!\!\!\!\!\!\!\!\!\!\!\!\!\!\!\!\!
&&m(x,t):=
\begin{cases}
\int_{{-}\infty}^{y_*(x,t){-}} \,{\rm d}m_0(\eta) 
&\quad\quad\quad\,\,\, {\rm if}\, \nu(x,t)\!=\!F(y_*(x,t),x,t),\\[2mm]
\int_{{-}\infty}^{y_*(x,t){+}} \,{\rm d}m_0(\eta)
&\quad\quad\quad\,\,\, {\rm otherwise};
\end{cases}\label{meq0}\\[2mm]
\!\!\!\!\!\!\!\!\!\!\!\!\!\!\!\!\!\!
&&\,\, u(x,t):=
\begin{cases}
U_0 e^{-\frac{t}{\tau}}{+}\widetilde{m}_0(y_*(x,t){+})(\tau e^{{-}\frac{t}{\tau}}{-}\tau)  
&{\rm if} \ c(y_*(x,t);x,t)> u_0(y_*(x,t)) \\ 
& {\rm and}\ c(y_*(x,t){+};x,t)> U_0, \\[1mm]
{-}U_0 e^{{-}\frac{t}{\tau}}{+}\widetilde{m}_0(y_*(x,t){-}) (\tau e^{{-}\frac{t}{\tau}}{-}\tau) 
&{\rm if} \ c(y_*(x,t);x,t)<u_0(y_*(x,t)) \\
& {\rm and}\ c(y_*(x,t){-};x,t)< {-}U_0,\\[1mm]
\bar{x}'(t)  &{\rm otherwise.}
\end{cases}\label{ueq0}
\end{eqnarray}
where $\bar{x}(\tau)$ for $\tau\geq t$ is the unique forward generalized characteristic 
emitting from point $(x,t)$ as shown in Lemma $\ref{velocitylem}$, 
and $c(y;x,t)$ and $c(y{\pm};x,t)$ are the corresponding initial speeds of the 
three types of characteristics connecting $(y,0)$ and $(x,t)$ and satisfy
\begin{align}\label{xytctm}
\begin{cases}
c(y;x,t)=\frac{x-y}{\tau-\tau e^{-\frac{t}{\tau}}}- \widetilde{m}_0(y)\big(\tau-\frac{t}{1-e^{-\frac{t}{\tau}}}\big),\\[1mm]
c(y{\pm};x,t)=\frac{x-y}{\tau-\tau e^{-\frac{t}{\tau}}}- \widetilde{m}_0(y{\pm})\big(\tau-\frac{t}{1-e^{-\frac{t}{\tau}}}\big).
\end{cases}
\end{align}
Moreover, we define the momentum $q(x,t)$ and energy $E(x,t)$ as follows:
\begin{align}\label{qeq}
q(x,t):=
\begin{cases}
\int_{{-}\infty}^{y_*(x,t){-}}u_0(\eta) e^{{-}\frac{t}{\tau}} {+} \widetilde{m}_0(\eta) (\tau e^{-\frac{t}{\tau}}{-}\tau)\,{\rm d}m_0(\eta)
& {\rm if}\,  \nu(x,t)\!=\!F(y_*(x,t),x,t),\\[2mm]
\int_{{-}\infty}^{y_*(x,t){+}}u_0(\eta) e^{{-}\frac{t}{\tau}} {+} \widetilde{m}_0(\eta) (\tau e^{-\frac{t}{\tau}}{-}\tau) \, {\rm d}m_0(\eta) 
& {\rm otherwise};
\end{cases}
\end{align}
\begin{align}\label{Eeq}
&E(x,t):=\nonumber\\[2mm]
&\begin{cases}
\int_{{-}\infty}^{y_*(x,t){-}}\big(u_0(\eta) e^{{-}\frac{t}{\tau}} {+} \widetilde{m}_0(\eta) (\tau e^{-\frac{t}{\tau}}{-}\tau)\big)u(x(\eta,t),t) \, {\rm d}m_0(\eta) 
& {\rm if}\,  \nu(x,t)\!=\!F(y_*(x,t),x,t),\\[2mm]
\int_{{-}\infty}^{y_*(x,t){+}}\big(u_0(\eta) e^{{-}\frac{t}{\tau}} {+} \widetilde{m}_0(\eta) (\tau e^{-\frac{t}{\tau}}{-}\tau)\big)u(x(\eta,t),t) \, {\rm d}m_0(\eta) 
& {\rm otherwise},
\end{cases}
\end{align}
where $x(\eta,t)$ for $t\in\mathbb{R}^+$ is one of the forward generalized characteristics emitting from point $(\eta,0)$ determined by $\eqref{xetaeq1}$--$\eqref{xetaeq2}$.
As shown in Lemma $\ref{RNlem}$, for almost everywhere $t>0$, 
$m(\cdot\,,t)$, $q(\cdot\,,t)$, and $E(\cdot\,,t)$ satisfy 
\begin{equation}\label{qumEum}
    q_x=um_x, \qquad\,\,\, E_x=u^2m_x 
\end{equation}
in the sense of Radon--Nikodym derivatives; 
and as shown in Lemma $\ref{mqnulem}$, 
\begin{equation}\label{numnuq}
    \nu_x(x,t)=-m(x,t),\quad \nu_t(x,t)=q(x,t)\qquad\,\,\, {\rm in}\ \mathcal{D}'.
\end{equation}

We now introduce the second and third generalized potentials 
to deal with the Poisson source term $-\widetilde{m}m_x$ and the relaxation term $-\frac{u m_x}{\tau}$
in the momentum equation.

{\it The second generalized potential} $G(y;x,t)$ is defined by
\begin{equation}\label{potentialG}
G(y;x,t)=\int_{{-}\infty}^{y{-}} \big(u_0(\eta) e^{-\frac{t}{\tau}}+\widetilde{m}_0(\eta)(\tau e^{-\frac{t}{\tau}}-\tau) +k\big)(x(\eta,t)-x) {\rm d}m_0(\eta),
\end{equation}
where $k$ is any constant satisfying $k>U_0{+}\frac{M}{2}
\tau$ and $x(\eta,t)$ is given by $\eqref{xetaeq1}$--$\eqref{xetaeq2}$.
Similar to the first generalized potential $F(y;x,t)$, 
we set $\mu(x,t):=\min_{y\in\mathbb{R}}G(y;x,t)$
and let $\omega:=\omega(x,t)$ be defined by
\begin{align}\label{wxt}
\omega(x,t)=
\begin{cases}
-\frac{1}{\tau}e^{-\frac{t}{\tau}} \int_{{-}\infty}^{y_*(x,t){-}}(u_0(\eta)+\tau \widetilde{m}_0(\eta))(x(\eta,t)-x) \,{\rm d}m_0(\eta) 
\ \ {\rm if} \, \mu(x,t)=G(y_*(x,t);x,t),\\[2mm]
-\frac{1}{\tau}e^{-\frac{t}{\tau}} \int_{{-}\infty}^{y_*(x,t){-}}(u_0(\eta)+\tau \widetilde{m}_0(\eta))(x(\eta,t)-x)\,{\rm d}m_0(\eta) \ \ {\rm otherwise};
\end{cases}
\end{align}
Then, similar to Lemma $\ref{mqnulem}$, it can be checked from $\eqref{qeq}$--$\eqref{Eeq}$ and $\eqref{wxt}$ that
\begin{equation}\label{thetaqE}
\theta_x(x,t)=-q(x,t),\quad \theta_t(x,t)=E(x,t)+\omega(x,t)\qquad\,\,\, {\rm in}\ \mathcal{D}',
\end{equation}
where $\theta=\theta(x,t)$ is defined by
\begin{align*}
&\theta(x,t)=\\
&\begin{cases}
\int_{{-}\infty}^{y_*(x,t){-}}\big(u_0(\eta) e^{-\frac{t}{\tau}}+ \widetilde{m}_0(\eta)(\tau e^{-\frac{t}{\tau}}-\tau)\big)(x(\eta,t){-}x) \,{\rm d}m_0(\eta)
\  {\rm if} \ \mu(x,t)=G(y_*(x,t);x,t),\\[2mm]
\int_{{-}\infty}^{y_*(x,t){+}}\big(u_0(\eta) e^{-\frac{t}{\tau}}+ \widetilde{m}_0(\eta)(\tau e^{-\frac{t}{\tau}}-\tau)\big)(x(\eta,t){-}x) \,{\rm d}m_0(\eta)\  {\rm otherwise}.
\end{cases}
\end{align*}

{\it The third generalized potential} $H(y;x,t)$ is defined by
\begin{equation}\label{potentialH}
H(y;x,t)=-\frac{1}{\tau}e^{-\frac{t}{\tau}} \int_{{-}\infty}^{y{-}}(u_0(\eta)+\tau \widetilde{m}_0(\eta)+k)(x(\eta,t)-x) \,{\rm d}m_0(\eta),
\end{equation}
where $k>U_0+\frac{M}{2}\tau$ is an arbitrary constant. 
By the same arguments as the second generalized potential $G(y;x,t)$ 
and using the BV chain rule ({\it cf.} \cite{vol1967spaces}), it can be checked that
\begin{equation}\label{omegammm}
\omega_{xx}(x,t)=\frac{1}{\tau}q_x(x,t)+\widetilde{m}(x,t) \,m_x(x,t).
\end{equation}

Finally, system $\eqref{mmaineq}$ can be formally checked from $\eqref{meq0}$--$\eqref{ueq0}$, $\eqref{qumEum}$--$\eqref{numnuq}$, 
$\eqref{thetaqE}$, and $\eqref{omegammm}$ as follows:
\begin{eqnarray*}
&&m_t+um_x=m_t+q_x=-\nu_{xt}+\nu_{tx}=0,\\[1mm]
&&(m_xu)_t+(m_xu^2)_x=-(\theta_x)_{xt}+(\theta_t-\omega)_{xx}=-\omega_{xx}=-\widetilde{m}m_x-\frac{u m_x}{\tau}.
\end{eqnarray*}

We now present the main theorems of Euler-Poisson system about the existence and uniqueness.

\begin{The}[Existence Theorem I]\label{ExisThm} 
Let $\rho_0\geq 0$ be a finite Radon measure and $u_0 \in L^{\infty}_{\rho_0}(\mathbb{R})$ as in $\eqref{ID}$. 
Let $(m,u)=(m(x,t),u(x,t))$ be defined by
\begin{eqnarray}
\!\!\!\!\!\!\!\!\!\!\!\!\!\!\!\!\!\!
&&m(x,t)=
\begin{cases}
\int_{{-}\infty}^{y_*(x,t){-}} \,{\rm d}m_0(\eta) 
&\quad\quad\quad\,\,\, {\rm if}\, \nu(x,t)\!=\!F(y_*(x,t),x,t),\\[2mm]
\int_{{-}\infty}^{y_*(x,t){+}} \,{\rm d}m_0(\eta)
&\quad\quad\quad\,\,\, {\rm otherwise};
\end{cases}\label{meq}\\[2mm]
\!\!\!\!\!\!\!\!\!\!\!\!\!\!\!\!\!\!
&&\,\, u(x,t)=
\begin{cases}
U_0 e^{-\frac{t}{\tau}}{+}\widetilde{m}_0(y_*(x,t){+})(\tau e^{-\frac{t}{\tau}}{-}\tau)  \quad 
&{\rm if} \ c(y_*(x,t);x,t)> u_0(y_*(x,t))\ \\ 
&\quad {\rm and}\ c(y_*(x,t){+};x,t)> U_0, \\[1mm]
{-}U_0 e^{-\frac{t}{\tau}}{+}\widetilde{m}_0(y_*(x,t){-})(\tau e^{-\frac{t}{\tau}}{-}\tau) \quad
&{\rm if} \ c(y_*(x,t);x,t)<u_0(y_*(x,t))\ \\
&\quad {\rm and}\ c(y_*(x,t){-};x,t)< {-}U_0,\\[1mm]
\bar{x}'(t)  \quad &\text{otherwise.} 
\end{cases}\label{ueq}
\end{eqnarray}
Then, $\eqref{meq}$--$\eqref{ueq}$ is the solution formula of system $\eqref{maineq}$
in the sense{\rm:} 
$(\rho,u)=(m_x,u)$ is an entropy solution of the Cauchy problem $\eqref{maineq}$--$\eqref{ID}$.

\vspace{2pt}
In the above, $U_0=||u_0(x)||_{L^{\infty}_{\rho_0}(\mathbb{R})}${\rm;} 
$m_0(x)$ and $\widetilde{m}_0(x)$ are given by $\eqref{m0m0}${\rm;}
$F(y;x,t)$, $\nu(x,t)$, and $y_*(x,t)$ are given by 
$\eqref{potentialF}$, $\eqref{nuSeq}$, and $\eqref{y*eq}$ respectively{\rm;}
$\bar{x}(\xi)$ for $\xi\geq t$ is 
the unique forward generalized characteristic emitting from point $(x,t)${\rm;} 
$c(y;x,t)$ and $c(y{\pm};x,t)$ are given by \eqref{xytctm}.
\end{The}

\begin{The}[Uniqueness Theorem]\label{UniThm} 
Let $\rho_0\geq 0$ be a finite Radon measure and $u_0 \in L^{\infty}_{\rho_0}(\mathbb{R})$ as in $\eqref{ID}$.
Assume that $(m_1,u_1)$ and $(m_2,u_2)$ are two entropy solutions of system $\eqref{maineq}$ 
with the same initial data $(\rho_0,u_0)$. 
Then, 
\begin{equation}
m_1\mathop{=}\limits^{a.e.}m_2\,;\qquad\,\,\,
u_1\mathop{=}\limits^{a.e.}u_2\quad {\it w.r.t}\ \ {\rm the\ measure}\ m_{1x}=m_{2x}.
\end{equation}
\end{The}

\subsection{Formula of solutions to drift equations}

\smallskip

To the best of our knowledge, there is no result about the model \eqref{2.2a}. In order to prove the relaxation limit from \eqref{maineq} to \eqref{2.2a}, we need to construct the formula of its solutions.
Without the diffusion term, the drift equations exhibit behavior similar to pressureless Euler-Poisson system and we can solve it by similar method.

Similar to \eqref{mPhix} and \eqref{widem}, we define the mass and the potential as follows
\begin{align}
&\bar{m}(x,t):=\int_{-\infty}^{x-} \bar{\rho}(\eta,t) \, {\rm d}\eta,\qquad M=\int_{\mathbb{R}} \bar{\rho}(\eta,t)\, {\rm d}\eta,  \\[1mm]
&\bar{\Phi}_x(x,t):=\frac{1}{2}\big(\int_{-\infty}^{x-}{-}\int_{x+}^{\infty}\big) \bar{\rho}(\eta,t)\,{\rm d}\eta=\frac{1}{2}\big(\bar{m}(x{-},t){+}\bar{m}(x{+},t){-}M\big)=:\widehat{m}(x,t).\label{2.6}
\end{align}
Since $\partial_x \bar{m}(x,t)=\partial_x \widehat{m}(x,t)=\bar{\rho}$ in the sense of BV derivative, by integrating in $(-\infty,x)$, we can transform \eqref{2.2a} into
\begin{align}\label{2.7}
\begin{cases}
\bar{m}_t-\widehat{m}\,\bar{m}_x=0,\\[1mm]
\bar{m}(x,0)=m_0(x),
\end{cases}
\end{align}
where $m_0(x)$ is given by \eqref{m0m0}.
We focus on system \eqref{2.7} and give the definition of weak solutions as follows.

\begin{Def}[Weak Solution II]
Let $\bar{m}(x,t)
\in[0,M]$ be nondecreasing in $x$ and the measure $\bar{m}_x$ be weakly continuous in $t$, 
then $\bar{\rho}=\bar{m}_x$ {\rm(}resp. $m${\rm)} is called to be a weak solution of \eqref{2.2a} {\rm(}resp. \eqref{2.7}{\rm)}, if for any $\varphi \in C_c^{\infty}(\mathbb{R}\times\mathbb{R}^+)$,
\begin{align}
\iint \varphi_t \bar{m} \,{\rm d}x{\rm d}t+\iint \varphi \widehat{m} \,{\rm d}\bar{m}{\rm d}t=0;
\end{align}
and as $t\rightarrow 0+$, the measure $\bar{\rho}=\bar{m}_x$ weakly converges to $\rho_0$. 
\end{Def}

Now we construct the weak solutions of \eqref{2.7}. First,
it follows from \eqref{m0m0}, \eqref{2.6}, and \eqref{2.7} that 
\begin{align}
\widehat{m}_0(x):=\widehat{m}(x,0)=\widetilde{m}_0(x),\nonumber
\end{align}
so that the potential $\bar{F}(y;x,t)$ can be defined as
\begin{align}\label{2.9}
\bar{F}(y;x,t)=\int_{-\infty}^{y-} \eta -t\widetilde{m}_0(\eta)-x\,{\rm d}m_0(\eta).
\end{align}
For any fixed $(x,t)\in \mathbb{R}\times \mathbb{R}^+$, 
\begin{align}
\bar{F}(y_1;x,t)-\bar{F}(y_2;x,t)=\int_{y_2-}^{y_1-} \eta-t \widetilde{m}_0(\eta)-x \,{\rm d}m_0(\eta) \geq 0, \nonumber
\end{align}
holds for any $y_1$ and $y_2$ satisfying
\begin{align}\label{2.9a}
y_1>y_2>x+\frac{t M}{2}, \qquad {\text or}\quad y_1<y_2<x-\frac{t M}{2}, 
\end{align}
then $\bar{F}(\cdot;x,t)$ has a finite low bound in $y\in \mathbb{R}$.
Similar to \eqref{nuSeq}, we define
\begin{align}\label{2.10}
\begin{cases}
\bar{\nu}(x,t):=\min_{y\in \mathbb{R}}\bar{F}(y;x,t),\\[1mm]
\bar{S}(x,t):=\{y\,:\, \exists\, y_n \rightarrow y\quad  {\rm s.t.}\ \bar{F}(y_n;x,t) \rightarrow \bar{\nu}(x,t)\};
\end{cases}
\end{align}
and define $\bar{y}_*(x,t)$ and $\bar{y}^*(x,t)$ by 
\begin{align}\label{2.11}
\bar{y}_*(x,t):=\inf \big\{\bar{S}(x,t) \cap {\rm spt}\{\rho_0\}  \big\},\qquad\,\,\,
\bar{y}^*(x,t):=\lim_{\epsilon \rightarrow 0{+}} \bar{y}_*(x{+}\epsilon,t).
\end{align}
Similar to the proof of \eqref{geoentropy}, for any $x_1<x_2$ with $t>0$, we have
\begin{align}\label{2.12}
\bar{y}_*(x_1,t)\leq \bar{y}^*(x_1,t)\leq \bar{y}_*(x_2,t).
\end{align}

The weak solution $\bar{m}(x,t)$ can be constructed as follows:
\begin{align}\label{2.12m}
\bar{m}(x,t):=
\begin{cases}
\int_{{-}\infty}^{\bar{y}_*(x,t){-}} \,{\rm d}m_0(\eta) \quad
& {\rm if}\, \bar{\nu}(x,t)=\bar{F}(y_*(x,t),x,t),\\[1mm]
\int_{{-}\infty}^{\bar{y}_*(x,t){+}} \,{\rm d}m_0(\eta) \quad
& {\rm otherwise}.
\end{cases}
\end{align}
Besides, we also define 
\begin{align}\label{2.12q}
\bar{q}(x,t):=
\begin{cases}
\int_{{-}\infty}^{\bar{y}_*(x,t){-}}-\widetilde{m}_0(\eta) \,{\rm d}m_0(\eta) \quad
& {\rm if}\, \bar{\nu}(x,t)=\bar{F}(y_*(x,t),x,t),\\[1mm]
\int_{{-}\infty}^{\bar{y}_*(x,t){+}} -\widetilde{m}_0(\eta) \,{\rm d}m_0(\eta) \quad
& {\rm otherwise}.
\end{cases}
\end{align}
It follows from \eqref{4.8} that $\partial_x \bar{q}={-}\widehat{m} \bar{m}_x$ in the sense of BV derivative.
From Lemma \ref{4.3}, we can prove  $\bar{\nu}_x(x,t)=-\bar{m}(x,t),\,\, \bar{\nu}_t(x,t)=\bar{q}(x,t)$ in the sense of distribution,
so that \eqref{2.7} can be checked formally as follows:
\begin{align*}
    \bar{m}_t-\widehat{m}\bar{m}_x=-\bar{\nu}_{xt}+\bar{\nu}_{xt}=0.
\end{align*}
Thus we have the following existence theorem of drift equations.

\begin{The}[Existence Theorem II]\label{Ex2}
Let $\rho_0\geq 0$ be a finite Radon measure and $\bar{m}=\bar{m}(x,t)$ be defined by
\begin{align}\label{2.13}
\bar{m}(x,t):=
\begin{cases}
\int_{{-}\infty}^{\bar{y}_*(x,t){-}} \,{\rm d}m_0(\eta) \quad
& {\rm if}\, \bar{\nu}(x,t)=\bar{F}(y_*(x,t),x,t),\\[1mm]
\int_{{-}\infty}^{\bar{y}_*(x,t){+}} \,{\rm d}m_0(\eta) \quad
& {\rm otherwise}.
\end{cases}
\end{align}
Then \eqref{2.13} is the solution formula of \eqref{2.2a} in the sense: $\bar{\rho}=\bar{m}_x$ is a weak solution of \eqref{2.2a}.

In the above, $\bar{y}_*(x,t)$, $\bar{\nu}(x,t)$, and $\bar{F}$ are given by \eqref{2.11}, \eqref{2.10} and \eqref{2.9}, respectively.
\end{The}

\subsection{The relaxation limit of solutions}
By slow time scaling, the unique entropy solution \eqref{meq0}--\eqref{ueq0} of pressureless Euler-Poisson system transforms into
\begin{align}\label{2.13a}
m^{\tau}(x,t):=m(x,\frac{t}{\tau}), \quad
\rho^{\tau}(x,t):=\rho(x,\frac{t}{\tau})=m^{\tau}_x,\quad
u^{\tau}(x,t):=\frac{1}{\tau}u(x,\frac{t}{\tau}). 
\end{align}

For drift equations \eqref{2.2}, if we consider the smooth solutions, then $\bar{u}=-\bar{\Phi}_x$.
Due to this motivation, for weak solutions,  we define the velocity $\bar{u}(x,t)$ of drift equations as
\begin{align}\label{2.13b}
\bar{u}(x,t)=-\widehat{m}(x,t),
\end{align}
where $\widehat{m}(x,t)$ is given by \eqref{2.6}.
Then we have the following theorem about the relaxation limit.

\begin{The}[Relaxation Limit]\label{thm3}
Let $(\rho^{\tau}(x,t),u^{\tau}(x,t))$ be the sequence of the unique entropy solution to \eqref{maineq} with the initial data $(\rho_0,u_0)$ as in \eqref{ID},  $\bar{\rho}(x,t)$ and $\bar{u}(x,t)$ be the weak solution and velocity to \eqref{2.2a} respectively with the same initial data.
Then
\begin{align}
\rho^{\tau}(x,t)\rightharpoonup  \bar{\rho}(x,t) \ {\rm as}\ \tau \rightarrow 0;\qquad
\lim_{\tau\rightarrow 0{+}}u^{\tau}(x,t)\mathop{=}\limits^{a.e.} \bar{u}(x,t) \ w.r.t. \ measure\ \bar{m}_x.
\end{align}

In the above, $\rho^{\tau}$ and $u^{\tau}$ are given in \eqref{2.13a}{\rm;} $\bar{\rho}$ and $\bar{u}$ are given by \eqref{2.13} and \eqref{2.13b}, respectively.
\end{The}

\section{Proof of the Formula for Pressureless Euler--Poisson equations} 
In \S 3.1, we use the generalized potential $F(y;x,t)$ to construct 
the mass $m(x,t)$ and momentum $q(x,t)$ 
via the analysis of the backward characteristics;
in \S 3.2, by using the forward generalized characteristics, 
velocity $u(x,t)$ and energy $E(x,t)$ are constructed,  
and prove the relations $q_x=um_x$ and $E_x=u^2m_x$ in the sense of the Radon--Nikodym derivatives;
in \S 3.3--3.4, by using the three generalized potentials $F(y;x,t)$, $G(y;x,t)$, and $H(y;x,t)$ 
and the relations $q_x=um_x$ and $E_x=u^2m_x$, 
we prove that the $(m,u)=(m(x,t),u(x,t))$, defined by $\eqref{meq1}$ and $\eqref{ueq1}$, 
solve the mass equation and momentum equation as in Definition $\ref{weakdef}$.

\subsection{Construction of the mass and momentum}
For any fixed $(x,t)\in\mathbb{R}\times\mathbb{R}^+$,
since $F(\cdot\,;x,t)$ is left continuous in $y$, 
then for any $y_0 \in S(x,t)$, 
\begin{equation}\label{nuxt}
\nu(x,t)=
\begin{cases}
F(y_0;x,t) & {\rm if}\ F(y;x,t)\ {\rm achieves \ its \ minimum \ at}\ y_0, \\[1mm]
F(y_0{+};x,t) & {\rm if\ otherwise}.
\end{cases}
\end{equation}

\begin{Lem}\label{minlem}
Assume that $y_0 \in S(x,t)$. If $[m_0(y_0)]:=m_0(y_0{+})-m_0(y_0{-})>0$, then
\begin{equation}\label{nueq}
\nu(x,t)=\mathop{\min}\limits_{y\in\mathbb{R}} F(y;x,t)=
\begin{cases}
F(y_0;x,t) &\quad {\rm if}\ c(y_0;x,t) \leq u_0(y_0),\\[1mm]
F(y_0{+};x,t) &\quad {\rm if}\  c(y_0;x,t)> u_0(y_0),
\end{cases}
\end{equation}
where $c(y_0;x,t)$ is given by
\begin{equation*}
c(y_0;x,t)=\frac{x-y_0}{\tau-\tau e^{-\frac{t}{\tau}}}- \widetilde{m}_0(y_0)\big(\tau-\frac{t}{1-e^{-\frac{t}{\tau}}}\big).
\end{equation*}
\end{Lem}

\begin{proof}
From $u_0\in L^\infty_{\rho_0}$, if $[m_0(y_0)]=m_0(y_0{+})-m_0(y_0{-})>0$, 
then $\rho_0$ has positive mass at point $y_0$, so that
$u_0(\cdot)$ is well defined at point $y_0$. 
If $c(y_0;x,t) \leq u_0(y_0)$, then
$$F(y_0{+};x,t)-F(y_0{-};x,t)=(\tau-\tau e^{-\frac{t}{\tau}})(u_0(y_0)-c(y_0;x,t))[m_0(y_0)] \geq 0,$$
which, by the left continuity of $F(\cdot\,;x,t)$, implies that $F(y;x,t)$ achieves its minimum at $y_0$. 
This, by $\eqref{nuxt}$, yields $\eqref{nueq}$ for the case that $c(y_0;x,t)\leq u_0(y_0)$. 
By the same arguments, it is direct to check $\eqref{nueq}$ for the case that $c(y_0;x,t)>u_0(y_0)$.
\end{proof}

\begin{Lem}\label{convergelem}
For $\nu(x,t)$ and $S(x,t)$ defined by $\eqref{nuSeq}$, the following statements hold{\rm:}
\begin{itemize}
\item [(i)] $\nu(x,t)$ is continuous in $(x,t)\in\mathbb{R}\times\mathbb{R}^+$.

\vspace{1pt}
\item [(ii)] Let $(x_n,t_n)$ and $y_n \in S(x_n,t_n)$ converge to $(x,t)$ and $y_0$, respectively, then 
\begin{equation}\label{y0ynS}
y_0=\lim_{n\rightarrow \infty}y_n \in S(x,t).
\end{equation}
Thus $S(x,t)$ is closed.

\end{itemize}

\end{Lem}

\begin{proof}
The proof is divided into two steps accordingly.

\smallskip
\noindent
{\bf 1.}  
By the definition of $F(y;x,t)$ in $\eqref{potentialF}$, 
$F(y;x,t)$ is continuous in $(x,t)$.
From $\eqref{nuSeq}$, we have
$$\nu(x_n,t_n)=F(y_n;x_n,t_n)\,\, {\rm or}\,\, F(y_n{+};x_n,t_n)$$ 
so that 
$$\liminf_{n \rightarrow \infty} \nu(x_n,t_n)
=\liminf_{n \rightarrow \infty} F(y_n;x_n,t_n)=F(y_0;x,t) \,\, {\rm or}\,\, F(y_0{+};x,t), \,\,{\rm respectively.}
$$ 
This implies 
$$\nu(x,t) \leq \liminf_{n \rightarrow \infty}\nu(x_n,t_n).$$

On the other hand, without loss of generality, 
assume that $\nu(x,t)=F(\bar{y};x,t)$. 
Since $F(\bar{y};x_n,t_n) \geq \nu(x_n,t_n)$, 
by letting $n \rightarrow \infty$, 
$$\nu(x,t) \geq \limsup_{n \rightarrow \infty}\nu(x_n,t_n).$$  

Therefore, $\nu(x,t)=\lim_{n \rightarrow \infty} \nu(x_n,t_n)$ so that 
$\nu(x,t)$ is continuous in $(x,t)$.

\smallskip
\noindent
{\bf 2.} By choosing a subsequence of $y_n$ (still denoted by $y_n$), 
at least one of the following two holds for all $n$:
$$\nu(x_n,t_n)=F(y_n;x_n,t_n),\qquad\,\,\,\nu(x_n,t_n)=F(y_n{+};x_n,t_n).$$
Without loss of generality, assume that $\nu(x_n,t_n)=F(y_n;x_n,t_n)$. 
By the continuity of $F(y;x,t)$ in $(x,t)$, we have $|F(y_n;x,t)-F(y_n;x_n,t_n)|=o(1)$ so that 
$$F(y_n;x,t) \longrightarrow \nu(x,t)\qquad\,\,\, {\rm as}\ n \rightarrow \infty.$$
This yields $\eqref{y0ynS}$.
\end{proof}

We now define $y_*(x,t)$ and $y^*(x,t)$ by
\begin{align}\label{y*eq1}
y_*(x,t):=\inf \big\{S(x,t) \cap {\rm spt}\{\rho_0\} \big\},\qquad\,\,\,
y^*(x,t):=\lim_{\epsilon \rightarrow 0{+}} y_*(x{+}\epsilon,t).
\end{align}
Here $y^*(x,t)$ is well-defined due to the fact that 
$y_*(x,t)$ is nondecreasing in $x$, {\it i.e.}, 
\begin{equation}\label{y*increas}
y_*(x_1,t)\le y_*(x_2,t)\qquad {\rm for\ any}\ x_1<x_2\ {\rm and}\ t>0.
\end{equation}
This can be seen as follows:
Without loss of generality, assume that 
$$\nu(x_1,t)=F(y_*(x_1,t);x_1,t),\qquad \nu(x_2,t)=F(y_*(x_2,t);x_2,t).$$
Then, it follows from $\eqref{potentialF}$ that
\begin{align}
0&\ge F(y_*(x_2,t);x_2,t)-F(y_*(x_1,t);x_2,t)\nonumber\\
&=F(y_*(x_2,t);x_1,t)-F(y_*(x_1,t);x_1,t)
-(x_2-x_1)\int_{y_*(x_1,t){-}}^{y_*(x_2,t){-}}\,{\rm d}m_0(\eta). \label{Fy*epsilon}
\end{align}
If $F(y_*(x_2,t);x_1,t)-F(y_*(x_1,t);x_1,t)=0$, then $y_*(x_2,t)\in S(x_1,t) \cap {\rm spt}\{\rho_0\}$, 
which, by the definition of $y_*$ in $\eqref{y*eq1}$, implies that $y_*(x_2,t)\ge y_*(x_1,t)$.
If $F(y_*(x_2,t);x_1,t)-F(y_*(x_1,t);x_1,t)>0$, 
then $\eqref{Fy*epsilon}$ implies that $m_0(y_*(x_2,t))>m_0(y_*(x_1,t))$. 
This, by the nondecreasing property of $m_0(x)$, yields that $y_*(x_2,t)> y_*(x_1,t)$.

\vspace{2pt}
By Lemma $\ref{convergelem}$, $S(x,t)$ is closed, 
then from $\eqref{y*eq1}$--$\eqref{y*increas}$,
for any $(x,t)\in\mathbb{R}\times\mathbb{R}^+$,
\begin{equation}\label{y*y*Y}
y_*(x,t),\, y^*(x,t) \in S(x,t) \cap {\rm spt}\{\rho_0\},\qquad\,\,\, y_*(x,t) \leq y^*(x,t).
\end{equation}

We now present the definition of the backward characteristics 
from the points on the upper half-plane: 
For any $(x_0,t_0)\in\mathbb{R}\times\mathbb{R}^+$, suppose that $y_0\in S(x_0,t_0)$.
If $[m_0(y_0)]=0$, or $[m_0(y_0)]>0$ with $c(y_0;x_0,t_0)-u_0(y_0)=0$, we define the curve $l_{y_0}(x_0,t_0)$ for $t<t_0$ by 
\begin{equation}\label{lyxtm0}
x=y_0{+}c(y_0;x_0,t_0) (\tau-\tau e^{-\frac{t}{\tau}})+\widetilde{m}_0(y_0)(\tau^2-\tau^2 e^{-\frac{t}{\tau}}-\tau t);
\end{equation}
and if $[m_0(y_0)]>0$ with ${\pm}(c(y_0;x_0,t_0)-u_0(y_0))>0$, 
we define the curve $l_{y_0}(x_0,t_0)$ for $t<t_0$ by 
\begin{equation}\label{lyxt000}
x=
y_0{+} c(y_0{\pm};x_0,t_0) (\tau-\tau e^{-\frac{t}{\tau}})+\widetilde{m}_0(y_0{\pm})(\tau^2-\tau^2 e^{-\frac{t}{\tau}}-\tau t).
\end{equation}
In the above, $c(y;x,t)$ and $c(y{\pm};x,t)$ are given by $\eqref{xytctm}$. 
See also Fig. $\ref{char12}$.

\begin{figure}[H]
    \centering
    \includegraphics[width=0.7\linewidth]{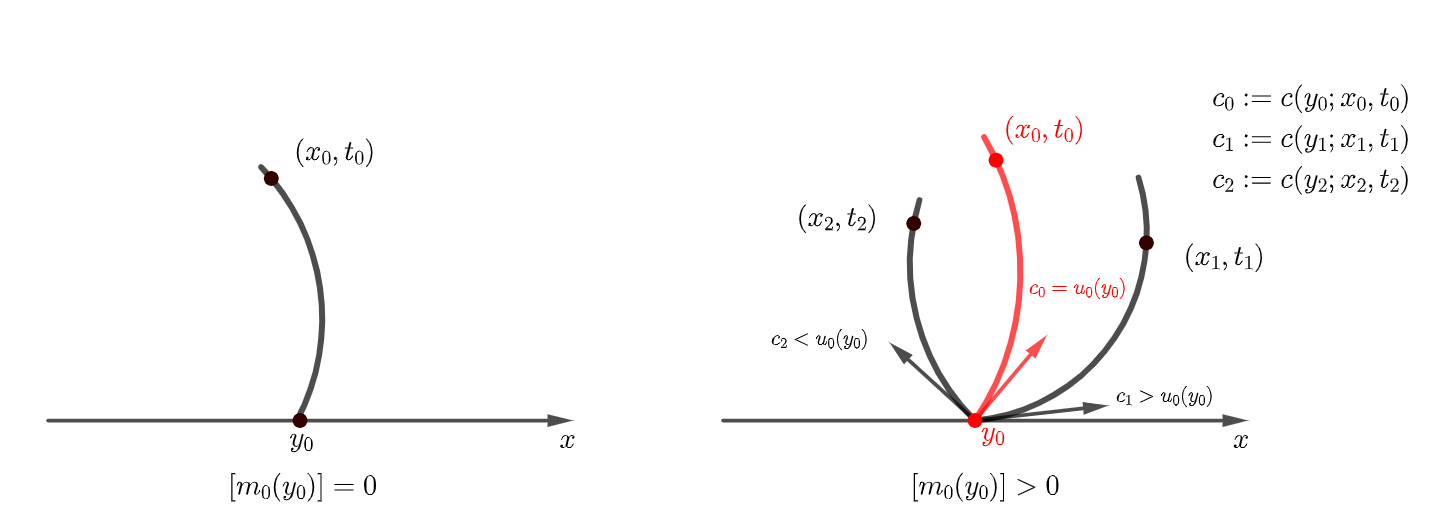}
    \caption{Different cases of backward characteristics. 
    When $[m_0(y_0)]>0$, there exist three different types of backward characteristics for some points.}
    \label{char12}
\end{figure}

\smallskip
We now present the following lemma on the characteristics.
\begin{Lem}\label{curvelem}
Let $y_*(x,t)$ and $y^*(x,t)$ be defined by $\eqref{y*eq1}$ and let 
the curve $l_y(x,t)$ be given by $\eqref{lyxtm0}$--$\eqref{lyxt000}$. 
Then, the following statements hold{\rm:} 

\noindent
\ {\rm(i)} For any $(x_0,t_0)\in\mathbb{R}\times\mathbb{R}^+$, assume that $y_0\in S(x_0,t_0)$. Then along $l_{y_0}(x_0,t_0)$,
\begin{equation}\label{y0Sx0t0}
y_0\in S(x,t) \subset S(x_0,t_0).
\end{equation}
Furthermore, if $y_0=y_*(x_0,t_0)$, then for any $(x,t)\in l_{y_*(x_0,t_0)}(x_0,t_0)$,
\begin{equation}\label{yxtx0t0}
y_*(x,t)=y_*(x_0,t_0).
\end{equation}




\noindent
\ {\rm(ii)} For any $x_1 <x_2$ with $t>0$, 
\begin{equation}\label{geoentropy}
y_*(x_1,t)\le y^*(x_1,t) \leq y_*(x_2,t)\le y^*(x_2,t). 
\end{equation}

\noindent
\ {\rm(iii)} For any $(x,t)\in\mathbb{R}\times\mathbb{R}^+$,
\begin{equation}\label{y*left}
\lim_{\epsilon\rightarrow 0{+}}y_*(x{-}\epsilon,t)=\lim_{\epsilon \rightarrow 0{+}}y^*(x{-}\epsilon,t)=y_*(x,t),
\end{equation}
\begin{equation}\label{y*right}
\lim_{\epsilon\rightarrow 0{+}}y_*(x{+}\epsilon,t)=\lim_{\epsilon \rightarrow 0{+}}y^*(x{+}\epsilon,t)=y^*(x,t).
\end{equation}
\end{Lem}

\begin{proof}
The proof is divided into three steps accordingly.

\smallskip
\noindent
{\bf 1.} Suppose that $y_0 \in S(x_0,t_0)$, then by \eqref{nuxt}, 
 $$\nu(x_0,t_0)=F(y_0;x_0,t_0)\qquad \text{or} \quad  \nu(x_0,t_0)=F(y_0{+};x_0,t_0).$$
Along curve $l_{y_0}(x_0,t_0)$, we have that, for any $y \in \mathbb{R}$,  
\begin{align}
&\ F(y;x,t)-F(y_0{\pm};x,t)\nonumber\\[1mm]
&=\int_{y_0{\pm}}^{y{-}} \eta+u_0(\eta) (\tau -\tau e^{-\frac{t}{\tau}})+\widetilde{m}_0(\eta)(\tau^2-\tau^2 e^{-\frac{t}{\tau}}-\tau t)-x \,{\rm d}m_0(\eta)\nonumber\\[1mm]
&=\frac{1-e^{-\frac{t}{\tau}}}{1 - e^{-\frac{t_0}{\tau}}}
\int_{y_0{\pm}}^{y{-}} \eta+u_0(\eta) (\tau -\tau e^{-\frac{t_0}{\tau}})+\widetilde{m}_0(\eta)(\tau^2-\tau^2 e^{-\frac{t_0}{\tau}}-\tau t_0)-x_0 \,{\rm d}m_0(\eta) \nonumber\\[1mm]
&+\frac{e^{-\frac{t}{\tau}}-e^{-\frac{t_0}{\tau}}}{1 - e^{-\frac{t_0}{\tau}}}\int_{y_0{\pm}}^{y{-}} \eta-y_0 \,{\rm d}m_0(\eta)  \nonumber\\[1mm] &+\frac{\tau}{1-e^{-\frac{t_0}{\tau}}}(t_0-t+t e^{-\frac{t_0}{\tau}}-t_0 e^{-\frac{t}{\tau}})\int_{y_0{\pm}}^{y{-}} \big(\widetilde{m}_0(\eta)-\widetilde{m}_0(y_0{\pm})\big) \,{\rm d}m_0(\eta) 
\nonumber\\[1mm]
&
\geq 0,\label{FFyy0}
\end{align}
where we exploit $F(y;x_0,t_0)-F(y_0{\pm};x_0,t_0)\geq 0$ and the fact that
\begin{equation}\label{3.1}
f(t):=t_0-t+t e^{-\frac{t_0}{\tau}}-t_0 e^{-\frac{t}{\tau}}>0.
\end{equation}
In fact, $f''(t)=-\frac{t_0}{\tau^2}e^{-\frac{t}{\tau}}<0$, then $f'(t)=-1+e^{-\frac{t_0}{\tau}}+\frac{t_0}{\tau}e^{-\frac{t}{\tau}}$ is strictly decreasing.
Since $f'(0)=-1+e^{-\frac{t_0}{\tau}}+\frac{t_0}{\tau}$, letting $g(x):=-1+e^{-\frac{x}{\tau}}+\frac{x}{\tau}$, then $g'(x)=\frac{1}{\tau}(1-e^{-\frac{x}{\tau}})>0$. 
Thus $f'(0)=g(t_0)>g(0)=0$.
It follows from $f(0)=f(t_0)=0$ that $f(t)>0$ in $t\in (0,t_0)$.

\eqref{FFyy0} infers that $y_0 \in S(x,t)$. 
If there exists $y'\in S(x,t)$ with $y'\neq y_0$, then 
$$
F(y'{\pm};x_0,t_0)-F(y_0{\pm};x_0,t_0)=0,
$$
which imply
$$
y'\in S(x_0,t_0).
$$
This means that $\eqref{y0Sx0t0}$ holds.

For the case that $(x,t)\in l_{y_*(x_0,t_0)}(x_0,t_0)$, 
from \eqref{y0Sx0t0}, $y_*(x_0,t_0)\in S(x,t)\subset S(x_0,t_0)$.
Since $y_*(x,t)=\inf\{S(x,t)\cap{\rm spt}\{\rho_0\}\}$ and $y_*(x_0,t_0)=\inf\{S(x_0,t_0)\cap{\rm spt}\{\rho_0\}\}$, then
$$
y_*(x,t)=y_*(x_0,t_0).
$$

\smallskip
\noindent
{\bf 2.}
Let $x_1<x_2$ with $t>0$.
By $\eqref{y*increas}$, $y_*(x,t)$ is nondecreasing in $x$ so that, 
for sufficiently small $\epsilon>0$,
$$
y_*(x_1,t)\le y_*(x_1{+}\epsilon,t) \le y_*(x_2,t)\le y_*(x_2{+}\epsilon,t),
$$
which, by \eqref{y*eq1}, implies that, as $\epsilon \rightarrow 0{+}$,
$$
y_*(x_1,t)\le y^*(x_1,t)\le y_*(x_2,t)\le y^*(x_2,t).
$$

\smallskip
\noindent
{\bf 3.} By the definition of $y^*(x,t)$ in \eqref{y*eq1}, 
it is direct to see: $\lim_{\epsilon\rightarrow 0{+}}y_*(x{+}\epsilon,t)=y^*(x,t)$.
By Lemma \ref{convergelem}, it follows from \eqref{y*eq1} and \eqref{y*increas} that $\lim_{\epsilon\rightarrow 0{+}}y_*(x{-}\epsilon,t) \in S(x,t)\cap {\rm spt}\{\rho_0\}$ and $y_*(x{-}\epsilon,t)\le y_*(x,t)$.
Similarly, $\lim_{\epsilon\rightarrow 0{+}}y_*(x{-}\epsilon,t)=y_*(x,t)$. 

\vspace{1pt}
From \eqref{y*y*Y} and \eqref{geoentropy}, for sufficiently small $\epsilon>0$,
$$y_*(x{-}\epsilon,t)\le y^*(x{-}\epsilon,t)\le y_*(x,t),$$
which, by $\lim_{\epsilon\rightarrow 0{+}}y_*(x{-}\epsilon,t)=y_*(x,t)$, implies that 
$\lim_{\epsilon\rightarrow 0{+}}y^*(x{-}\epsilon,t)=y_*(x,t)$.

\vspace{1pt}
By the definition of $y^*(x,t)$ in \eqref{y*eq1}, it is direct to see that
$$\lim_{\epsilon \rightarrow 0{+}}y^*(x{+}\epsilon,t)=\lim_{\epsilon \rightarrow 0{+}}\lim_{\delta \rightarrow 0{+}}y_*(x{+}\epsilon{+}\delta,t)=\lim_{\xi\rightarrow0{+}}y_*(x{+}\xi,t)=y^*(x,t).$$
This yields $\eqref{y*right}$.
\end{proof}

\noindent
{\it Definition of the mass $m(x,t)$ and momentum $q(x,t)$.} 
We now use $y_*(x,t)$ to define the mass $m(x,t)$ and momentum $q(x,t)$ as
\begin{eqnarray}
\!\!\!\!\!\!\!\!\!\!\!\!
&&m(x,t):=
\begin{cases}
\int_{{-}\infty}^{y_*(x,t){-}} \,{\rm d}m_0(\eta) \quad {\rm if}\, \nu(x,t)=F(y_*(x,t),x,t),\\[2mm]
\int_{{-}\infty}^{y_*(x,t){+}} \,{\rm d}m_0(\eta)\quad {\rm otherwise}.
\end{cases}\label{meq1}\\[2mm]
\!\!\!\!\!\!\!\!\!\!\!\!
&&q(x,t):=
\begin{cases}
\int_{{-}\infty}^{y_*(x,t){-}}u_0(\eta) e^{-\frac{t}{\tau}}{+} \widetilde{m}_0(\eta)(\tau e^{-\frac{t}{\tau}}-\tau)\,{\rm d}m_0(\eta)  
\quad {\rm if}\,  \nu(x,t)\!=\!F(y_*(x,t),x,t),\\[2mm]
\int_{{-}\infty}^{y_*(x,t){+}}u_0(\eta) e^{-\frac{t}{\tau}}{+} \widetilde{m}_0(\eta)(\tau e^{-\frac{t}{\tau}}-\tau) \, {\rm d}m_0(\eta) 
\quad {\rm otherwise}.
\end{cases}\label{qeq1}
\end{eqnarray}
By view of \eqref{yxtx0t0}, the mass $m(x,t)$ is transported along characteristics.

\subsection{Construction of the velocity and energy}
The key point to define the velocity $u(x,t)$ is to determine the speeds of $\delta$--shocks.
To do this, we need to use the forward generalized characteristics.

For any $(x,t)\in\mathbb{R}\times\mathbb{R}^+$, 
the backward characteristic region of $(x,t)$, denoted by $\Delta(x,t)$, 
is the region contoured by 
the characteristics $l_{y_*(x,t)}(x,t)$ and $l_{y^*(x,t)}(x,t)$ with the $x$--axis
and defined by 
\begin{align*}
\Delta(x,t):=\big\{(\xi,\tau)\in\mathbb{R}\times\mathbb{R}^+\,:\, 
(\xi,\tau)\ {\rm lies\ between}\ l_{y_*(x,t)}(x,t)\ {\rm and}\ l_{y^*(x,t)}(x,t) \big\}.
\end{align*}

First we show that $l_{y^*(x_1,t_0)}(x_1,t_0)$ never interacts with $l_{y_*(x_2,t_0)}(x_2,t_0)$ in $t\in (0,t_0)$ if $x_1<x_2$.
Denoting $y_1:=y^*(x_1,t_0)$ and $y_2:=y_*(x_2,t_0)$, it follows from \eqref{geoentropy} that $y_1\leq y_2$.
Without loss of generality, assume 
$$\nu(x_1,t_0)=F(y_1;x_1,t_0) \quad \text{and} \quad 
\nu(x_2,t_0)=F(y_2;x_2,t_0).
$$
From \eqref{lyxtm0} and \eqref{lyxt000}, without loss of generality, for $t\in (0,t_0)$, assume
\begin{align*}
l_{y_1}(x_1,t_0): \bar{x}_1(t)=y_1+c(y_1;x_1,t_0)(\tau-\tau e^{-\frac{t}{\tau}})+\widetilde{m}_0(y_1)(\tau^2-\tau^2 e^{-\frac{t}{\tau}}-\tau t),\\[1mm]
l_{y_2}(x_2,t_0): \bar{x}_2(t)=y_2+c(y_2;x_2,t_0)(\tau-\tau e^{-\frac{t}{\tau}})+\widetilde{m}_0(y_2)(\tau^2-\tau^2 e^{-\frac{t}{\tau}}-\tau t).
\end{align*}
Then
\begin{align}
\bar{x}_2(t)-\bar{x}_1(t)
&=\frac{e^{-\frac{t}{\tau}}-e^{-\frac{t_0}{\tau}}}{1-e^{-\frac{t_0}{\tau}}}(y_2-y_1)+\frac{1-e^{-\frac{t}{\tau}}}{1-e^{-\frac{t_0}{\tau}}}(x_2-x_1)\nonumber\\[1mm]
&\quad+\frac{\tau}{1-e^{-\frac{t_0}{\tau}}} f(t)(\widetilde{m}_0(y_2)-\widetilde{m}_0(y_1))>0,\label{3.2}
\end{align}
where $f(t)$ is given by \eqref{3.1}.

Thus for any fixed $t$ with $x_1\neq x_2$,
$\Delta(x_1,t)$ never interacts with $\Delta(x_2,t)$ on the upper half-plane.

\begin{Lem}\label{velocitylem}
For any fixed $(x_0,t_0)\in\mathbb{R}\times\mathbb{R}^+$, 
there exists a unique forward generalized characteristic of $(x_0,t_0)${\rm:}
$x=x(t)$ for $t\geq t_0$ with $x(t_0)=x_0$, 
and the following statements on the derivative $x'(t):=\lim_{t^{\prime \prime},t'\rightarrow t{+}0}
\frac{x(t'')-x(t')}{t''-t'}$
hold{\rm:}

\vspace{2pt}
\noindent
\ {\rm(i)} For the case that $y_*(x,t)<y^*(x,t)$, there exist three subcases{\rm:}
\begin{itemize}
\item [(a)] If $[m(x,t)]>0$, then  
\begin{equation}\label{xprime2}
x'(t)= 
\frac{[q(x,t)]}{[m(x,t)]}
=\frac{\int_{y_*(x,t){\pm}}^{y^*(x,t){\pm}} u_0(\eta) e^{-\frac{t}{\tau}}+ \widetilde{m}_0(\eta)(\tau e^{-\frac{t}{\tau}}-\tau)\, {\rm d}m_0(\eta)}
{\int_{y_*(x,t){\pm}}^{y^*(x,t){\pm}}\, {\rm d}m_0(\eta)},
\end{equation}
where the signs $\pm$ are determined by{\rm:} 
In term $y_*(x,t){\pm}$, the sign is $y_*(x,t){-}$ $(${\it resp.}, $y_*(x,t)+$$)$ if $c(y_*(x,t);x,t)-u_0(y_*(x,t))\leq 0$ $(${\it resp.}, $>0$$);$ 
in term $y^*(x,t){\pm}$, the sign is $y^*(x,t){-}$ $(${\it resp.}, $y^*(x,t)+$$)$ if $c(y^*(x,t);x,t)-u_0(y^*(x,t))$ $<0$ $(${\it resp.}, $\geq0$$)$.

\item [(b)] If $[m(x,t)]=0$, 
and $y_*(x{-}\epsilon,t)<y_*(x,t)$ for any $\epsilon>0$, then
\begin{align}
\quad\
x'(t)
=\lim_{\epsilon \rightarrow 0{+}} 
\frac{\int_{y_*(x{-}\epsilon,t){-}}^{y_*(x{+}\epsilon,t){-}} 
u_0(\eta) e^{-\frac{t}{\tau}}+ \widetilde{m}_0(\eta)(\tau e^{-\frac{t}{\tau}}-\tau)\,{\rm d}m_0(\eta)}
{\int_{y_*(x{-}\epsilon,t){-}}^{y_*(x{+}\epsilon,t){-}}\,{\rm d}m_0(\eta)}.\label{xprime4}
\end{align}

\item [(c)] If $[m(x,t)]=0$, 
and $y_*(x{-}\epsilon,t)=y_*(x,t)$ for sufficiently small $\epsilon>0$, then
\begin{equation}\label{xprime3}
x'(t)=\frac{x{-}y^*(x,t)}{\tau e^{\frac{t}{\tau}}-\tau }{+} \widetilde{m}_0(y^*(x,t))\big(\frac{t}{e^{\frac{t}{\tau}}{-}1}{-}\tau \big).
\end{equation}
\end{itemize}

\noindent
\ {\rm(ii)} For the case that $y_*(x,t)=y^*(x,t)$, there exist two subcases:
\vspace{1pt}
\begin{itemize}
\item [(a)] If $[m_0(y_*(x,t))]>0$, then  
\begin{equation}\label{xprime1}
\quad\quad
x'(t)= 
\begin{cases}
\!\frac{x{-}y_*(x,t)}{\tau e^{\frac{t}{\tau}}-\tau }{+} \widetilde{m}_0(y_*(x,t){\pm})\big(\frac{t}{e^{\frac{t}{\tau}}{-}1}{-}\tau \big)
& {\rm if} \, {\pm}\big(c(y_*(x,t);x,t){-}u_0(y_*(x,t))\big){>}0,\\[2mm]
\!\frac{x{-}y_*(x,t)}{\tau e^{\frac{t}{\tau}}-\tau }{+} \widetilde{m}_0(y_*(x,t))\big(\frac{t}{e^{\frac{t}{\tau}}{-}1}{-}\tau \big)
& {\rm if} \, c(y_*(x,t);x,t){-}u_0(y_*(x,t))=0.
\end{cases}
\end{equation}

\item [(b)] If $[m_0(y_*(x,t))]=0$, then
\begin{equation}\label{xprime5}
x'(t)=\frac{x{-}y_*(x,t)}{\tau e^{\frac{t}{\tau}}-\tau }{+} \widetilde{m}_0(y_*(x,t))\big(\frac{t}{e^{\frac{t}{\tau}}{-}1}{-}\tau \big).
\end{equation}
\end{itemize}
See also {\rm Fig.} $\ref{char5678}$.
\end{Lem}

\begin{figure}[H]
    \centering
    \includegraphics[width=0.8\linewidth]{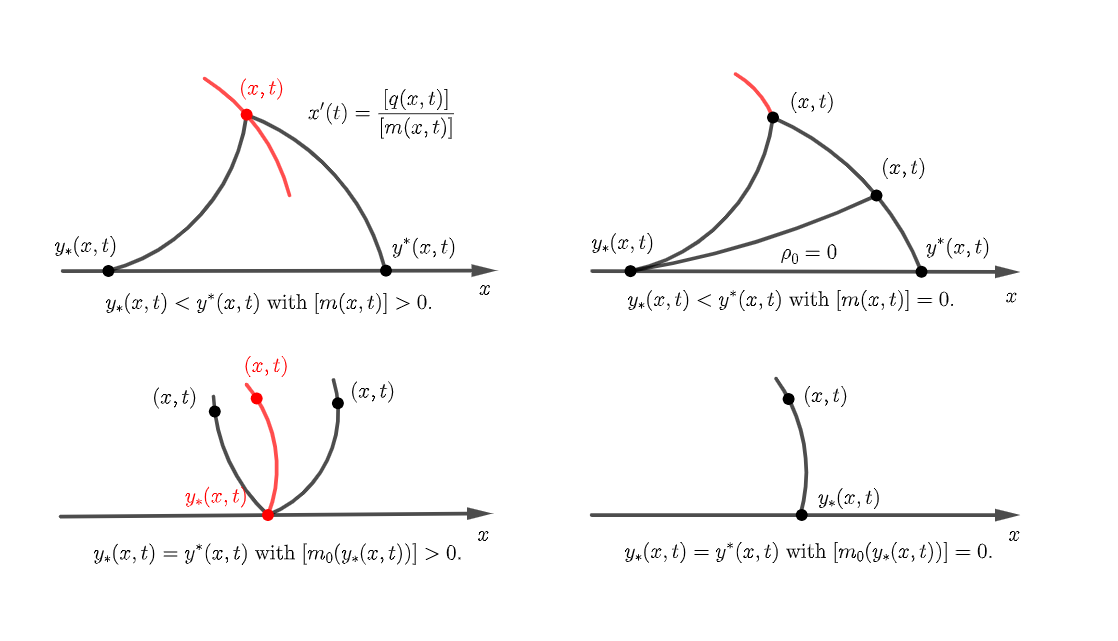}
    \caption{Different cases of points on forward generalized characteristics.}
    \label{char5678}
\end{figure}

\begin{proof}
Lemma $\ref{velocitylem}$ is proved case by case in the following five steps.

\smallskip
\noindent
{\bf 1.} For any line $t=t_1>t_0$, 
denote the $l_{y_*(x,t_1)}(x,t_1)$ as in 
$\eqref{lyxtm0}$--$\eqref{lyxt000}$ by $\xi:=X(t;x,t_1)$ with $x=X(t_1;x,t_1)$.
From \eqref{3.2}, for fixed $t$ and $t_1$, $X(t;x,t_1)$ is strictly increasing with respect to $x$.

Letting
\begin{equation}\label{at1bt1}
a(t_1)=\sup \{x\,:\, X(t_0;x,t_1)<x_0\},\qquad \,\,\,
b(t_1)=\inf \{x\,:\, X(t_0;x,t_1)>x_0\}.
\end{equation}
Then, $a(t_1) \leq b(t_1)$ holds. We claim:  $a(t_1)=b(t_1)$.
In fact, if $a(t_1)<b(t_1)$, then for any $x' \in (a(t_1),b(t_1))$,  
$$
X(t_0;x',t_1)=x_0,
$$ 
which contradicts the fact that 
any two different characteristic regions $\Delta(x,t_1)$ and $\Delta(x',t_1)$ with $x\neq x'$
never interact with each other on the upper half-plane. 

Therefore, for any $t_1>t_0$, there always exists a unique point $x=x(t_1)=a(t_1)$ on the line $t=t_1$. 
When time $t_1$ changes, the points $(x(t_1),t_1)$ form a curve $X(x_0,t_0)\,:\, x=x(t)$ with $x_0=x(t_0)$ for $t>t_0$.
In particular, for any $t_2>t_1>t_0$, 
the characteristic $\Delta(x,t_2)$ does not contain the points $(x(t_0),t_0)$ and $(x(t_1),t_1)$ if $x \neq x(t_2)$. 
This means that the curve $X(x_0,t_0)$ coincides with the curve $X(x(t_1),t_1)$ emitting from $(x(t_1),t_1)$ for $t\geq t_1$.
Therefore, there exists a unique forward generalized characteristic $X(x_0,t_0)$ of $(x_0,t_0)$.

\smallskip
\noindent
{\bf 2.} We now prove $\eqref{xprime2}$. 
Denote $y_*:=y_*(x,t)$ and $y^*:=y^*(x,t)$. 
By \eqref{meq1}, it follows from $[m(x,t)]>0$ that
\begin{equation}\label{[mxt]}
\int_{y_*{-}}^{y^*{+}}\,{\rm d}m_0(\eta)>0.
\end{equation}
Choosing a sequence of $(x_n,t')$ with $x_n>x(t')=a(t')$ such that 
$x_n \rightarrow x(t')$ as $n \rightarrow \infty$. 
From the definition of $b(t')$ in $\eqref{at1bt1}$, 
$y_*(x_n,t')\geq y^*$ so that 
$$y^*(x(t'),t')=\lim_{n \rightarrow \infty}y_*(x_n,t')\geq y^*.$$ 
Similarly, it is direct to check that $y_*(x(t'),t')\leq y_*$. 
This yields
\begin{equation}
 \label{Deltaeq}
 \Delta(x(t),t) \subset \Delta(x(t'),t')\qquad {\rm for\ any}\ t'>t.
\end{equation}

\vspace{2pt}
{\bf (a).} 
Let $x':=x(t')$ and $x^{\prime \prime}:=x(t'')$ with $t''>t'>t$, 
and denote
\begin{equation}\label{yyyy1}
y_{\prime}:=y_*(x',t'),\qquad y':=y^*(x',t'),\qquad
y_{\prime\prime}:=y_*(x'',t''),\qquad y'':=y^*(x'',t'').
\end{equation}
From $\eqref{Deltaeq}$, 
$\Delta(x,t) \subset \Delta(x',t') \subset \Delta(x'',t'')$ so that
\begin{equation}\label{yyyy2}
y_{\prime\prime} \leq y_{\prime} \leq y_* \leq y^* \leq y' \leq y''.
\end{equation}

Then, we claim: 
\begin{equation}\label{limyyyy}
\lim_{t''\rightarrow t{+}0}y_{\prime}=\lim_{t''\rightarrow t{+}0}y_{\prime\prime}=y_*,\qquad
\lim_{t''\rightarrow t{+}0}y'=\lim_{t''\rightarrow t{+}0}y''=y^*.   
\end{equation}
By Lemma $\ref{convergelem}$, 
all the upper- and lower-limits of $y_{\prime},y_{\prime\prime}$ and $y',y''$
are belonging to $S(x,t)\cap {\rm spt}\{\rho_0\}$, which, by $\eqref{y*eq1}$ and $\eqref{yyyy2}$, imply that
$\lim_{t'\rightarrow t{+}0} y_{\prime}=\lim_{t''\rightarrow t{+}0} y_{\prime\prime}=y_*$ and 
\begin{eqnarray}
&& y^*\leq \liminf\limits_{t'\rightarrow t{+}0} y'\leq \limsup\limits_{t'\rightarrow t{+}0} y'\leq \sup\big(S(x,t)\cap {\rm spt}\{\rho_0\}\big),\label{y*yY}\\
&& y^*\leq \liminf\limits_{t''\rightarrow t{+}0} y''\leq \limsup\limits_{t''\rightarrow t{+}0} y''\leq \sup\big(S(x,t)\cap {\rm spt}\{\rho_0\}\big).\label{y*yYY}
\end{eqnarray}
Denote $Y:=\sup\big(S(x,t)\cap {\rm spt}\{\rho_0\}\big)$. 
Without loss of generality, assume that $\nu(x,t)=F(Y;x,t)$ and 
$\nu(x{+}\varepsilon,t)=F(y_*(x{+}\varepsilon,t);x{+}\varepsilon,t)$ for $\varepsilon>0$.
Then, 
\begin{align*}
0&\leq F(Y;x{+}\varepsilon,t)-F(y_*(x{+}\varepsilon,t);x{+}\varepsilon,t)\\
&=F(Y;x,t)-F(y_*(x{+}\varepsilon,t);x,t)-\varepsilon\int_{y_*(x{+}\varepsilon,t){-}}^{Y{-}}\,{\rm d}m_0(\eta)
\leq -\varepsilon\int_{y_*(x{+}\varepsilon,t){-}}^{Y{-}}\,{\rm d}m_0(\eta),
\end{align*}
which, by $\eqref{y*eq1}$, implies
\begin{equation}\label{Yy*xt}
\int_{y^*(x,t){+}}^{Y{-}}\,{\rm d}m_0(\eta)\leq 0.
\end{equation}
According to $\eqref{y*yY}$--$\eqref{y*yYY}$ and $\eqref{Yy*xt}$, if $\eqref{limyyyy}$ does not hold for $y^*$, 
by choosing a subsequence, we can assume that 
$y^*< \lim_{t'\rightarrow t{+}0} y'=Y.$
Then, by Lemma $\ref{curvelem}$(i), for $(x{+}\varepsilon',t)\in l_{y'}(x',t')$ with $\varepsilon'\geq 0$, 
we have that $y^*(x{+}\varepsilon',t)\geq Y$ so that 
$y^*=\lim_{\varepsilon'\rightarrow 0{+}} y^*(x{+}\varepsilon',t)\geq Y$,
which contradicts to that $y^*<Y$. So $\eqref{limyyyy}$ also holds for $y^*$.

\vspace{2pt}
{\bf (b).} 
Without loss of generality of the chosen of the two cases in $\eqref{nueq}$, assume that
\begin{eqnarray}
&&\ \nu(x',t')\, =F(y';x',t')=F(y_{\prime};x',t'),\label{nuxp}\\[1mm]
&&\nu(x'',t'')=F(y'';x'',t'')
=F(y_{\prime\prime};x'',t'').\label{nuxp2}
\end{eqnarray}
This, by the definition of $\nu(x,t)$, implies 
\begin{equation}
F(y'';x'',t'')-F(y_{\prime};x'',t'') 
\leq F(y'';x',t')-F(y_{\prime};x',t')
\end{equation}
so that, by $\eqref{potentialF}$, 
\begin{align*}
&\int_{y_{\prime}{-}}^{y''{-}}
\eta{+}u_0(\eta)(\tau-\tau e^{-\frac{t''}{\tau}}){+}\widetilde{m}_0(\eta)(\tau^2-\tau^2e^{-\frac{\tau}{t''}}-\tau t''){-}x'' \,{\rm d}m_0(\eta)\\[1mm]
\leq &\int_{y_{\prime}{-}}^{y''{-}}
\eta{+} u_0(\eta)(\tau-\tau e^{-\frac{t'}{\tau}}){+}\widetilde{m}_0(\eta)(\tau^2-\tau^2e^{-\frac{\tau}{t'}}-\tau t'){-}x' \,{\rm d}m_0(\eta), 
\end{align*}
which, by \eqref{[mxt]} and a simple calculation, implies
\begin{equation}
 \label{loweq}
 \frac{x''-x'}{t''-t'} \geq \frac{\int_{y_{\prime}{-}}^{y''{-}}\tau u_0(\eta)\frac{e^{-\frac{t'}{\tau}}-e^{-\frac{t''}{\tau}}}{t''-t'}+\widetilde{m}_0(\eta) \frac{e^{-\frac{t'}{\tau}}-e^{-\frac{t''}{\tau}}}{t''-t'}\,{\rm d}m_0(\eta)}
 {\int_{y_{\prime}{-}}^{y''{-}} \,{\rm d}m_0(\eta)}.
\end{equation}
On the other hand, we have
\begin{equation*}
F(y_{\prime\prime};x'',t'')-F(y',x'',t'')
\leq F(y_{\prime\prime};x',t')-F(y';x',t'),
\end{equation*}
which, by \eqref{[mxt]} and a simple calculation, implies 
\begin{equation}\label{upeq}
\frac{x''-x'}{t''-t'} 
\leq \frac{\int_{y_{\prime\prime}{-}}^{y'{-}}
\tau u_0(\eta)\frac{e^{-\frac{t'}{\tau}}-e^{-\frac{t''}{\tau}}}{t''-t'}+\widetilde{m}_0(\eta) \frac{e^{-\frac{t'}{\tau}}-e^{-\frac{t''}{\tau}}}{t''-t'} \,{\rm d}m_0(\eta)}
{\int_{y_{\prime\prime}{-}}^{y'{-}} \,{\rm d}m_0(\eta)}.
\end{equation}
Then, from $\eqref{meq1}$--$\eqref{qeq1}$ and $\eqref{limyyyy}$, 
by letting $t'' \rightarrow t{+}0$ in $\eqref{loweq}$-$\eqref{upeq}$, we obtain
\begin{equation}\label{xprimetm0}
x'(t)=\frac{\int_{y_*(x,t){-}}^{y^*(x,t){+}}u_0(\eta) e^{-\frac{t}{\tau}}+ \widetilde{m}_0(\eta)(\tau e^{-\frac{t}{\tau}}-\tau) \,{\rm d}m_0(\eta)}
{\int_{y_*(x,t){-}}^{y^*(x,t){+}}\,{\rm d}m_0(\eta)}
=\frac{[q(x,t)]}{[m(x,t)]},
\end{equation}
where the chosen of $y_*(x,t){-}$ follows from the chosen of $\eqref{nuxp}$
by using $\eqref{nueq}$ and the fact, induced by Lemma $\ref{convergelem}$, that
$$
\nu(x,t)=\lim_{t' \rightarrow t{+}0}\nu(x',t')
=\lim_{t' \rightarrow t{+}0}F(y_\prime;x',t')=F(y_*(x,t);x,t);
$$
and similarly, the chosen of $y^*(x,t){+}$ follows from the chosen of $\eqref{nuxp2}$.

\vspace{2pt}
By the arbitrariness of the chosen of $\eqref{nuxp}$ and $\eqref{nuxp2}$ based on $\eqref{nueq}$, $\eqref{xprime2}$ follows.

\smallskip
\noindent
{\bf 3.} We now prove $\eqref{xprime4}$--$\eqref{xprime3}$. 
Suppose that $y_*(x,t)<y^*(x,t)$ and $[m(x,t)]=0$.

\smallskip
{\bf (a).} For $\eqref{xprime4}$, since $y_*(x{-}\epsilon,t)<y_*(x,t)$ for any $\epsilon>0$, 
similar to Step {\bf 2}, according to $\eqref{y*left}$--$\eqref{y*right}$, 
$\eqref{xprime4}$ follows from $\eqref{loweq}$--$\eqref{upeq}$ by letting $t''\rightarrow t{+}0$.

\smallskip
{\bf (b).} We now show \eqref{xprime3}. 
It follows from $\eqref{meq1}$ that
$$
\rho_0((y_*(x,t),y^*(x,t)])=0,
$$
which means that $(x,t)$ lies in $l_{y^*(x,t)}(x,t)$:
\begin{equation*}
x(\xi)=y^*(x,t)+ c(y^*(x,t);x,t)(\tau-\tau e^{-\frac{\xi}{\tau}})+ \widetilde{m}_0(y^*(x,t))(\tau^2-\tau^2 e^{-\frac{\xi}{\tau}}-\tau \xi).
\end{equation*}
Since $y_*(x{-}\epsilon,t)=y_*(x,t)$ for sufficiently small $\epsilon>0$, 
there exists $t_a>t$ such that $(x(t_a),t_a)$ lies in $l_{y^*(x,t)}(x,t)$, 
which implies
$$
x'(t)=\frac{x{-}y_*(x,t)}{\tau e^{\frac{t}{\tau}}-\tau }{+} \widetilde{m}_0(y_*(x,t))\big(\frac{t}{e^{\frac{t}{\tau}}{-}1}{-}\tau \big).
$$

\smallskip
\noindent
{\bf 4.} We now prove $\eqref{xprime1}$. There exist two cases:

\vspace{2pt}
{\bf (a).} If there exists $t_a>t$ such that $y_*(x(t_a),t_a)=y^*(x(t_a),t_a)$, from Lemma $\ref{convergelem}$, 
$$y_*(x(t_a),t_a)=y_*(x,t),$$
so that, 
if $\pm(c(y_*(x,t);x,t)-u_0(y_*(x,t)))>0$, 
by the definition of $a(t_1)$ in \eqref{at1bt1}, for any $\xi\in[t,t_a]$,
\begin{equation*}
x(\xi)=y_*(x,t)+ c(y^*(x,t){\pm};x,t)(\tau-\tau e^{-\frac{\xi}{\tau}})+ \widetilde{m}_0(y^*(x,t){\pm})(\tau^2-\tau^2 e^{-\frac{\xi}{\tau}}-\tau \xi),
\end{equation*}
which implies
\begin{equation*}\label{xprimet}
x'(t)=\frac{x{-}y_*(x,t)}{\tau e^{\frac{t}{\tau}}-\tau }{+} \widetilde{m}_0(y_*(x,t){\pm})\big(\frac{t}{e^{\frac{t}{\tau}}{-}1}{-}\tau \big).
\end{equation*}

Similarly, $\eqref{xprime1}$ holds
for the case that $c(y_*(x,t);x,t)-u_0(y_*(x,t))=0$.

\vspace{2pt}
{\bf (b).} If $y_*(x(t_a),t_a)<y^*(x(t_a),t_a)$ for any $t_a>t$, from Lemma \ref{convergelem}, 
\begin{equation}\label{y*tat}
\lim_{t_a \rightarrow t{+}0}y_*(x(t_a),t_a)=\lim_{t_a \rightarrow t{+}0}y^*(x(t_a),t_a)=y_*(x,t).
\end{equation}
This means that $l_{y_*(x,t)}(x,t)$ is given by
\begin{equation}\label{ly*tau}
x(\xi)=y^*(x,t)+ u_0(y^*(x,t))(\tau-\tau e^{-\frac{\xi}{\tau}})+ \widetilde{m}_0(y^*(x,t))(\tau^2-\tau^2 e^{-\frac{\xi}{\tau}}-\tau \xi)
\end{equation}
Let $x':=x(t')$ and $x^{\prime \prime}:=x(t'')$ with $t''>t'>t$, denote as \eqref{yyyy1}, by the same arguments as \eqref{nuxp}-\eqref{upeq}, \eqref{xprimetm0} holds so that 
$x'(t)=u_0(y_*(x,t)) e^{-\frac{t}{\tau}}+\widetilde{m}_0(y_*(x,t))(\tau e^{-\frac{t}{\tau}}-\tau)$,
which, by $c(y_*(x,t);x,t)=u_0(y_*(x,t))$, implies
$$
x'(t)=\frac{x{-}y_*(x,t)}{\tau e^{\frac{t}{\tau}}-\tau }{+} \widetilde{m}_0(y_*(x,t))\big(\frac{t}{e^{\frac{t}{\tau}}{-}1}{-}\tau \big)
$$

\smallskip
\noindent
{\bf 5.} We now prove $\eqref{xprime5}$. 
Suppose that $y_*(x,t)=y^*(x,t)$ and $[m_0(y_*(x,t))]=0$.
There exist two cases:

\smallskip
{\bf (a).} If there exists $t_a>t$ such that $y_*(x(t_a),t_a)=y^*(x(t_a),t_a)$, 
by the same arguments as {\bf(a)} in Step {\bf4},
it is direct to see that $\eqref{xprime5}$ holds.

\smallskip
{\bf (b).} If $y_*(x(t_a),t_a)<y^*(x(t_a),t_a)$ for any $t_a>t$,
let the backward characteristics 
$l_{y_*(x(t_a),t_a)}(x(t_a),t_a)$ and $l_{y^*(x(t_a),t_a)}(x(t_a),t_a)$ 
be expressed by $\bar{x}(\xi)$ and $\bar{\bar{x}}(\xi)$ for $\xi\in(0,t_a]$, respectively.
Then, it follows from
$\widetilde{m}_0(y_*(x,t){-})=\widetilde{m}_0(y_*(x,t))=\widetilde{m}_0(y_*(x,t){+})$ that, 
as $t_a \rightarrow t{+}0$,
\begin{equation}\label{lll}
l_{y_*(x(t_a),t_a)}(x(t_a),t_a),\quad l_{y^*(x(t_a),t_a)}(x(t_a),t_a)
\longrightarrow l_{y_*(x,t)(x,t)}. 
\end{equation}
Since $\Delta(x(t),t)\subset \Delta(x(t_a),t_a)$, then
$\bar{x}(t')<x(t')<\bar{\bar{x}}(t')$ for $t'\in [t,t_a)$ so that
$$
\frac{\bar{x}(t_a)-\bar{x}(t')}{t_a-t'}
\le \frac{x(t_a)-x(t')}{t_a-t'}
\le \frac{\bar{\bar{x}}(t_a)-\bar{\bar{x}}(t')}{t_a-t'},
$$
which, by $\eqref{lll}$ and letting $t_a \rightarrow t{+}0$, implies 
$$
x'(t)=\frac{x{-}y_*(x,t)}{\tau e^{\frac{t}{\tau}}-\tau }{+} \widetilde{m}_0(y_*(x,t))\big(\frac{t}{e^{\frac{t}{\tau}}{-}1}{-}\tau \big).
$$

\smallskip
Up to now, we have completed the proof of Lemma $\ref{velocitylem}$.
\end{proof}

\noindent
{\it Definition of the velocity $u(x,t)$.}
By using the forward generalized characteristics in Lemma $\ref{velocitylem}$ and adjusting the  velocity in vacuum region caused by initial vacuum,
we define the velocity $u(x,t)$ as
\begin{equation}\label{ueq1}
u(x,t)=
\begin{cases}
U_0 e^{-\frac{t}{\tau}}{+}\widetilde{m}_0(y_*(x,t){+})(\tau e^{-\frac{t}{\tau}}{-}\tau)  \quad 
&{\rm if} \ c(y_*(x,t);x,t)> u_0(y_*(x,t))\ \\ 
&\quad {\rm and}\ c(y_*(x,t){+};x,t)> U_0, \\[1mm]
{-}U_0 e^{-\frac{t}{\tau}}{+}\widetilde{m}_0(y_*(x,t){-})(\tau e^{-\frac{t}{\tau}}{-}\tau) \quad
&{\rm if} \ c(y_*(x,t);x,t)<u_0(y_*(x,t))\ \\
&\quad {\rm and}\ c(y_*(x,t){-};x,t)< {-}U_0,\\[1mm]
\bar{x}'(t)  \quad &\text{otherwise.} 
\end{cases}
\end{equation}
where the curve $\bar{x}(\xi)$ is the forward generalized characteristic emitting from $(x,t)$.

\smallskip
By Lemma $\ref{velocitylem}$, 
$u(x,t)$ is well defined almost everywhere in $\mathbb{R}\times\mathbb{R}^+$. 
The properties of $u(x,t)$ are presented in the following lemma.

\begin{Lem} \label{lem:uuu}
Let $u(x,t)$ and $y_*(x,t)$ be given by $\eqref{ueq1}$ and $\eqref{y*eq1}$, respectively.

\vspace{2pt}
\noindent
\ {\rm(i)} For $(x,t)\in\mathbb{R}\times\mathbb{R}^+$ with $[m_0(y_*(x{\pm},t))]>0$,
the following statements hold{\rm:}
\vspace{1pt}
\begin{itemize}
\item [(a)] If $c(y_*(x{\pm},t);x,t)=u_0(y_*(x{\pm},t))$,
\begin{equation}\label{u1}
u(x{\pm},t)=\frac{x-y_*(x{\pm},t)}{\tau e^{\frac{t}{\tau}}{-}\tau}{+} \widetilde{m}_0(y_*(x{\pm},t){\pm})(\frac{t}{e^{\frac{t}{\tau}}{-}1}{-}\tau);
\end{equation}

\item [(b)] If $c(y_*(x{\pm},t);x,t)>u_0(y_*(x{\pm},t))$, 
\begin{equation}\label{u2}
\quad u(x{\pm},t)\!=\!
\begin{cases}
\frac{x-y_*(x{\pm},t)}{\tau e^{\frac{t}{\tau}}-\tau}{+} \widetilde{m}_0(y_*(x{\pm},t){+})(\frac{t}{e^{\frac{t}{\tau}}{-}1}{-}\tau)
 &{\rm if}\, c(y_*(x{\pm},t){+};x,t)\leq U_0,\\[2mm]
U_0 e^{-\frac{t}{\tau}}{+}\widetilde{m}_0(y_*(x{\pm},t){+})(\tau e^{-\frac{t}{\tau}}{-}\tau), 
 &{\rm if}\, c(y_*(x{\pm},t){+};x,t)> U_0.
\end{cases}
\end{equation}

\item[(c)]  If $c(y_*(x{\pm},t);x,t)<u_0(y_*(x{\pm},t))$, 
\begin{equation}\label{u3}
\quad\ u(x{\pm},t)\!=\!
\begin{cases}
\frac{x-y_*(x{\pm}0,t)}{\tau e^{\frac{t}{\tau}}{-}\tau}{+} \widetilde{m}_0(y_*(x{\pm},t){-})(\frac{t}{e^{\frac{t}{\tau}}{-}1}{-}\tau)
&{\rm if}\, c(y_*(x{\pm},t){-};x,t){\geq} {-}U_0,\\[2mm]
{-}U_0 e^{-\frac{t}{\tau}}{+} \widetilde{m}_0(y_*(x{\pm},t){-})(\tau e^{-\frac{t}{\tau}}{-}\tau), 
&{\rm if}\, c(y_*(x{\pm},t){-};x,t){<}{-} U_0.
\end{cases}
\end{equation}
\end{itemize}

\noindent
\ {\rm(ii)} For $(x,t)\in\mathbb{R}\times\mathbb{R}^+$ with $[m_0(y_*(x{\pm},t))]=0$,
\begin{equation}\label{u4}
\quad u(x{\pm},t)\!=\!
\begin{cases}
\frac{x-y_*(x{\pm},t)}{\tau e^{\frac{t}{\tau}}{-}\tau}{+} \widetilde{m}_0(y_*(x{\pm},t))(\frac{t}{e^{\frac{t}{\tau}}{-}1}{-}\tau)
\quad &{\rm if}\, |c(y_*(x{\pm},t),x,t)|\leq U_0, \\[2mm]
U_0 e^{-\frac{t}{\tau}}{+}\widetilde{m}_0(y_*(x{\pm},t)) (\tau e^{-\frac{t}{\tau}}{-}\tau)
\quad &{\rm if}\, c(y_*(x{\pm},t);x,t)>U_0,\\[2mm]
{-}U_0 e^{-\frac{t}{\tau}}{+}\widetilde{m}_0(y_*(x{\pm},t))(\tau e^{-\frac{t}{\tau}}{-}\tau)
\quad &{\rm if}\, c(y_*(x{\pm},t);x,t){<}{-}U_0.
\end{cases}
\end{equation}

\noindent
\ Furthermore, the velocity $u(x,t)$ satisfies the following two properties{\rm:}
\begin{itemize}
\item[(i)]  For {\it a.e.} $(x,t)\in\mathbb{R}\times\mathbb{R}^+$,
\begin{equation}\label{uuupm}
u(x{+},t)\leq u(x,t) \leq u(x{-},t).
\end{equation}

\item[(ii)] For any $x_1<x_2$ and {\it a.e.} $t>0$
the Oleinik entropy condition $\eqref{ux2ux1}$ holds, {\it i.e.,}
$$\frac{u(x_2,t)-u(x_1,t)}{x_2-x_1}\leq \frac{1}{t}.
$$
\end{itemize}
\end{Lem}

\begin{proof}
The proof is divided into three steps.

\smallskip
\noindent
{\bf 1.} For (i), we only prove \eqref{u1}, 
since it is similar for \eqref{u2}, \eqref{u3}, and \eqref{u4}. 

It follows from \eqref{geoentropy} that there exists at most countable points $x$ satisfying $y_*(x,t)<y^*(x,t)$ for fixed $t$, so that we can choose a sequence $x_n<x$ with $y_*(x_n,t)=y^*(x_n,t)$.
Since $c(y_*(x{-},t);x,t)=u_0(y_*(x{-},t))$, then $c(y_*(x_n,t),x_n,t)<u_0(y_*(x{-},t))$ and we have $\widetilde{m}_0(y_*(x_n,t){\pm})\rightarrow \widetilde{m}_0(y_*(x{-},t){-})$ as $x_n\rightarrow x$.
From \eqref{xprime1} and \eqref{xprime5}, 
\begin{align*}
u(x{-},t)=\lim_{x_n\rightarrow x}u(x_n,t)=\frac{x-y_*(x{-},t)}{\tau e^{\frac{t}{\tau}}{-}\tau}{+} \widetilde{m}_0(y_*(x{-},t){-})(\frac{t}{e^{\frac{t}{\tau}}{-}1}{-}\tau).
\end{align*}
The formula of $u(x{+},t)$ can be proved similarly, then \eqref{u1} holds.

\smallskip
\noindent
{\bf 2.} Now we show that \eqref{uuupm} holds, and we only prove $u(x{-},t)\geq u(x,t)$ since it is similar for $u(x{+},t)\leq u(x,t)$. If $y_*(x,t)=y^*(x,t)$, for the case \eqref{u1} we have
\begin{align*}
u(x,t)=\frac{x-y_*(x{-},t)}{\tau e^{\frac{t}{\tau}}{-}\tau}{+} \widetilde{m}_0(y_*(x{-},t))(\frac{t}{e^{\frac{t}{\tau}}{-}1}{-}\tau)
\leq u(x{-},t),
\end{align*}
where we exploit the fact that $\widetilde{m}_0(\eta)$ is increasing and $y_*(x{-},t)=y_*(x,t)$.
The case \eqref{u2}, \eqref{u3}, and \eqref{u4} are similar.

If $y_*(x,t)<y^*(x,t)$, assume that 
the forward generalized characteristic emitting from $(x,t)$ is $\gamma=\bar{x}(\xi)$ with $\xi\geq t$.
For $t''>t'>t$, let curve $l_{y_*(\bar{x}(t''),t'')}(\bar{x}(t''),t'')$ 
be given by $\gamma=\bar{\bar{x}}(\xi)$ with $\xi\in[0,t'']$, then from \eqref{Deltaeq} we know $\bar{\bar{x}}(t')\leq \bar{x}(t')$.
It follows from \eqref{limyyyy} that 
\begin{align*}
\lim_{t'' \rightarrow t{+}0}\frac{\bar{\bar{x}}(t'')-\bar{\bar{x}}(t')}{t''-t'}=u(x{-},t),
\end{align*}
which implies
\begin{align*}
u(x,t)=\lim_{t'' \rightarrow t{+}0}\frac{\bar{x}(t'')-\bar{x}(t')}{t''-t'}\le 
\lim_{t'' \rightarrow t{+}0}\frac{\bar{\bar{x}}(t'')-\bar{\bar{x}}(t')}{t''-t'}=u(x{-},t).
\end{align*}

\smallskip
\noindent

{\bf 3.} We now prove the Oleinik entropy condition. For fixed $t$ and $x_1<x_2$, without loss of generality,
there exist $y_1\leq y_2$ and $c_1, c_2$ with $\widetilde{m}_0(y_1)\leq \widetilde{m}_0(y_2)$ and $y_i\in S(x_i,t)$, such that
$$
x_i=y_i+c_i(\tau-\tau e^{-\frac{t}{\tau}})+\widetilde{m}_0(y_i)(\tau^2-\tau^2 e^{-\frac{t}{\tau}}-\tau t), \qquad {\rm for}\ i=1, 2;
$$
and 
$$
u(x_2{-},t)=c_2 e^{-\frac{t}{\tau}}+\widetilde{m}_0(y_2)(\tau e^{-\frac{t}{\tau}}-\tau),
\quad u(x_1{+},t)=c_1 e^{-\frac{t}{\tau}}+\widetilde{m}_0(y_1)(\tau e^{-\frac{t}{\tau}}-\tau).
$$
From \eqref{uuupm},
\begin{align*}
&\frac{u(x_2,t)-u(x_1,t)}{x_2-x_1} 
\leq \frac{u(x_2{-},t)-u(x_1{+},t)}{x_2-x_1}\nonumber\\[1mm]
&=\frac{(c_2-c_1)e^{\frac{t}{\tau}}+(\widetilde{m}_0(y_2)-\widetilde{m}_0(y_1))(\tau e^{-\frac{t}{\tau}}-\tau)}
{(y_2-y_1)+(c_2-c_1)(\tau-\tau e^{-\frac{t}{\tau}})+(\widetilde{m}_0(y_2)-\widetilde{m}_0(y_1))(\tau^2-\tau^2 e^{-\frac{t}{\tau}}-\tau t)}\nonumber\\[1mm]
&\leq \frac{e^{-\frac{t}{\tau}}}{\tau(1-e^{-\frac{t}{\tau}})}
\frac{(c_2-c_1)+(\widetilde{m}_0(y_2)-\widetilde{m}_0(y_1))(\tau-\tau e^{\frac{t}{\tau}})}
{(c_2-c_1)+(\widetilde{m}_0(y_2)-\widetilde{m}_0(y_1))(\tau-\frac{t}{1-e^{-\frac{t}{\tau}}})}\nonumber\\[1mm]
&\leq \frac{e^{-\frac{t}{\tau}}}{\tau(1-e^{-\frac{t}{\tau}})} \leq \frac{1}{t}.
\end{align*}

\end{proof}

By using mass $m(x,t)$ and velocity $u(x,t)$, 
we now construct the energy.
To do this, 
we need to introduce the forward generalized characteristic emitting from the $x$--axis.

For each point $(\eta,0)$, there exists at least one forward characteristic $x(t)$ with $x(0)=\eta$. 
There exist three cases:
\begin{itemize}
\item [(i)] If $\eta \in {\rm spt}\{\rho_0\}$ with $[m_0(\eta)]=0$, or $[m_0(\eta)]>0$ and 
there exist only one forward generalized characteristic emitting from $(\eta,0)$, 
we can arbitrarily choose one of them and denote it by $x(\eta,t)$.

\vspace{2pt}
\item [(ii)] If $\eta \in {\rm spt}\{\rho_0\}$ with $[m_0(\eta)]>0$ and 
there exist more one forward generalized characteristic emitting from $(\eta,0)$,
denote
$$
X(\eta,t)=\eta+u_0(\eta)(\tau-\tau e^{-\frac{t}{\tau}})+\widetilde{m}_0(\eta)(\tau^2-\tau^2 e^{-\frac{t}{\tau}}-\tau t),
$$
we let
\begin{equation}\label{xetaeq1}
x(\eta,t):=
\begin{cases}
a(\eta,t)  &{\rm if} \ X(\eta,t)  \leq a(\eta,t),\\[1mm]
X(\eta,t)   &{\rm if} \ a(\eta,t)\leq X(\eta,t) \leq b(\eta,t),\\[1mm]
b(\eta,t) &{\rm if} \ X(\eta,t) \geq b(\eta,t).
\end{cases}
\end{equation}
Here, $a(\eta,t)$ and $b(\eta,t)$ are given by
\begin{equation}\label{aetat}
 a(\eta,t):=
\begin{cases}
X(\eta,t) &\quad {\rm if} \ N_1 \cup N_2=\varnothing,\\[1mm]
\ \inf\{x\,:\, x \in N_2\} &\quad {\rm if} \ N_1=\varnothing,N_2 \neq \varnothing,\\[1mm]
\sup \{x\,:\, x \in N_1\} &\quad {\rm if} \ N_1\neq \varnothing,
\end{cases}
\end{equation}

\begin{equation}\label{betat}
\ \ b(\eta,t):=
\begin{cases}
X(\eta,t) &\quad {\rm if} \ N_1 \cup N_2=\varnothing,\\[1mm]
\sup\{x\,:\, x \in N_1\}  &\quad {\rm if}\ N_1\neq \varnothing,N_2 = \varnothing,\\[1mm]
\ \inf \{x\,:\, x \in N_2\} &\quad {\rm if} \ N_2\neq \varnothing,
\end{cases}
\end{equation}
where $N_1=\{x\,:\, y_*(x,t)<\eta \}$ and $N_2=\{x\,:\, y_*(x,t)>\eta \}$. 

\vspace{2pt}
\item [(iii)] If $\eta \in \mathbb{R}/{\rm spt}\{\rho_0\}=\cup_n(a_n,b_n)$, denote $b_-=\inf{\rm spt}\{\rho_0\}$. 
Since $ \widetilde{m}_0({-}\infty)=-\frac{M}{2}$, 
we let
\begin{align*}
\bar{x}(\eta,t):=
\begin{cases}
\eta{+}U_0(\tau{-}\tau e^{-\frac{t}{\tau}}) 
{+}\widetilde{m}_0(\eta)(\tau^2{-}\tau^2 e^{-\frac{t}{\tau}}{-}\tau t)  &\text{if $\eta \in (a_n,b_n)$ for $a_n>{-}\infty$},\\[2mm]
\eta{-}U_0(\tau{-}\tau e^{-\frac{t}{\tau}}) 
{+}\widetilde{m}_0(\eta)(\tau^2{-}\tau^2 e^{-\frac{t}{\tau}}{-}\tau t)  &\text{if $\eta \in ({-}\infty,b_-)$ }.
\end{cases}
\end{align*}
Then, there exists $t_{\eta}>0$ such that 
$\bar{x}(\eta,t)$ does not interact with 
other curve $x(\eta',t)$ on $(0,t_{\eta})$ for any $\eta' \in {\rm spt}\{\rho_0\}$. 
Denote $\bar{t}_{\eta}$ be the first time when $\bar{x}(\eta,t)$ interacts with 
curve $x(\bar{\eta},t)$ for some $\bar{\eta}\in {\rm spt}\{\rho_0\}$, and let
\begin{equation}\label{xetaeq2}
x(\eta,t):=
\begin{cases}
\bar{x}(\eta,t)&\quad {\rm if}\ 0<t< \bar{t}_{\eta},\\[1mm]
x(\bar{\eta},t) &\quad {\rm if}\ t\geq\bar{t}_{\eta}.
\end{cases}
\end{equation}
\end{itemize}

Then, it is easy to see that $x(\eta,t)$ is increasing in $\eta$ and continuous in $t>0$. 
Thus, $x=x(\eta,t)$ is well defined with respect to the measure $\rho_0$,  
and it follows from Lemma $\ref{velocitylem}$ and $\eqref{ueq1}$ that,
for any $\eta\in\mathbb{R}$ and {\it a.e.} $t>0$, 
\begin{equation}\label{xetat}
\frac{\partial x(\eta,t)}{\partial t}
:=\lim_{t^{\prime \prime},t'\rightarrow t{+}0}
\frac{x(\eta,t'')-x(\eta,t')}{t''-t'}=u(x(\eta,t),t).
\end{equation}

\noindent
{\bf Definition of the energy $E(x,t)$.}
By using the forward generalized characteristics $x(\eta,t)$, 
we define the energy $E(x,t)$ as
\begin{align}\label{Eeq1}
&E(x,t)\!:=\! \nonumber\\
&
\begin{cases}
\int_{{-}\infty}^{y_*(x,t){-}}
\big(u_0(\eta) e^{-\frac{t}{\tau}}+ \widetilde{m}_0(\eta)(\tau e^{-\frac{t}{\tau}}-\tau)\big)
u(x(\eta,t),t) \, {\rm d}m_0(\eta) 
&\! {\rm if}\,  \nu(x,t)\!=\!F(y_*(x,t),x,t),\\[2mm]
\int_{{-}\infty}^{y_*(x,t){+}}
\big(u_0(\eta) e^{-\frac{t}{\tau}}+ \widetilde{m}_0(\eta)(\tau e^{-\frac{t}{\tau}}-\tau)\big)
u(x(\eta,t),t) \, {\rm d}m_0(\eta) 
&\! {\rm otherwise}.
\end{cases}
\end{align}

Noticing that $u(x(\eta,t),t)$ is measurable to the measure $\rho_0$ and $m(x,t)$ is of bounded variation
in $x$, we have the following lemma on the relations of these functions.

\begin{Lem}\label{RNlem}
Let $m(x,t)$ and $u(x,t)$ be given by $\eqref{meq1}$ and $\eqref{ueq1}$, respectively. 
Then, $q(x,t)$ and $E(x,t)$ defined by $\eqref{qeq1}$ and $\eqref{Eeq1}$ satisfy that, for {\it a.e.} $t>0$,
\begin{equation}\label{qeum}
q_x=um_x,\qquad E_x=u^2m_x
\end{equation}
in the sense of Radon--Nikodym derivatives. 
\end{Lem}

\begin{proof}
For {\it a.e.} $(x_0,t_0)\in \mathbb{R}\times\mathbb{R}^+$, denote
$U_{\delta}(x_0,t_0):=(x_0-\delta,x_0{+}\delta) \times \{t_0\}$ with sufficiently small $\delta>0$.
Due to whether $y_*(x,t_0)$ can keep constant in $U_{\delta}(x_0,t_0)$ or not for $\delta>0$,
there exist two cases.

\smallskip
\noindent
{\bf 1.} Suppose that $y_*(x,t_0)$ keeps constant in $U_{\delta}(x_0,t_0)$ for  $\delta>0$. 
Then, for any $(x,t_0) \in U_{\delta}(x_0,t_0)$, $y_*(x,t_0)=y_*(x_0,t_0)$,
which, by Lemma $\ref{convergelem}$, implies  
$$y^*(x,t_0)=y_*(x_0,t_0)\qquad\,\,\, {\rm for}\ x\in(x_0-\delta,x_0{+}\delta).$$
Then there are two subcases: 

\vspace{2pt}
\noindent
{\it Subcase} 1. $\nu(x,t)\equiv F(y_*(x_0,t_0);x,t_0)$ or $\nu(x,t)\equiv F(y_*(x_0,t_0){+};x,t_0)$ for $x\in(x_0-\delta,x_0{+}\delta)$ with $\delta$ sufficiently small.

For such case, $m(x,t_0)$, $q(x,t_0)$, and $E(x,t_0)$ keep constant in $U_{\delta}(x_0,t_0)$ so that $\eqref{qeum}$ holds.
In particular, each $(x,t)$ that satisfies $c(y_*(x,t);x,t)>u_0(y_*(x,t))$ with $c(y_*(x,t){+};x,t)>U_0$ or $c(y_*(x,t);x,t)<u_0(y_*(x,t))$ with $c(y_*(x,t){-};x,t)<-U_0$  belongs to this {\it Subcase} 1.

\vspace{2pt}
\noindent
{\it Subcase} 2. $\nu(x,t)\equiv F(y_*(x_0,t_0);x,t_0)$ for $x\in (x_0-\delta,x_0)$,
and $\nu(x,t)\equiv F(y_*(x_0,t_0)+;x,t_0)$ for $x\in (x_0,x_0+\delta)$ with $\delta$ sufficiently small.

In this case, $m(x_0{-},t_0)\neq m(x_0{+},t_0)$,
which implies that $\rho_0$ possesses a positive mass at $y_0:=y_*(x_0,t_0)$. 
This, by Lemmas $\ref{minlem}$ and $\ref{velocitylem}$, implies
\begin{equation}\label{uxty0}
u(x_0,t_0)=\frac{x{-}y_0}{\tau e^{\frac{t_0}{\tau}}-\tau }{+} \widetilde{m}_0(y_0)\big(\frac{t_0}{e^{\frac{t_0}{\tau}}{-}1}{-}\tau \big).
\end{equation}
For any $x_1,x_2 \in U_{\delta}(x_0,t_0)$ and $x_1<x_0<x_2$, we have 
$$ F(y_0;x_1,t_0) \leq F(y_0{+};x_1,t_0),
\qquad\,\,\, F(y_0{+};x_2,t_0) \leq F(y_0;x_2,t_0),$$
so that
$$\int _{y_0{-}}^{y_0{+}}u_0(\eta) e^{-\frac{t}{\tau}}{+}\widetilde{m}_0(\eta)(\tau e^{-\frac{t}{\tau}}-\tau) \, {\rm d}m_0(\eta) 
\geq \int_{y_0{-}}^{y_0{+}}\frac{x_1-\eta}{\tau e^{\frac{t_0}{\tau}}-\tau}{+}\widetilde{m}_0(\eta)\big(\frac{t_0}{e^{\frac{t_0}{\tau}}{-}1}{-}\tau \big) \, {\rm d}m_0(\eta),$$
$$\int _{y_0{-}}^{y_0{+}}u_0(\eta) e^{-\frac{t}{\tau}}{+}\widetilde{m}_0(\eta)(\tau e^{-\frac{t}{\tau}}-\tau)\, {\rm d}m_0(\eta) 
\leq \int_{y_0{-}}^{y_0{+}}\frac{x_2-\eta}{\tau e^{\frac{t_0}{\tau}}-\tau}{+}\widetilde{m}_0(\eta)\big(\frac{t_0}{e^{\frac{t_0}{\tau}}{-}1}{-}\tau \big) \, {\rm d}m_0(\eta),$$
Let $x_2,x_1\rightarrow x_0{}\pm$,
\begin{equation*}
u_0(y_0) e^{-\frac{t}{\tau}}{+}\widetilde{m}_0(y_0)(\tau e^{-\frac{t}{\tau}}-\tau)=
\frac{x{-}y_0}{\tau e^{\frac{t_0}{\tau}}-\tau }{+} \widetilde{m}_0(y_0)\big(\frac{t_0}{e^{\frac{t_0}{\tau}}{-}1}{-}\tau \big)
\end{equation*}
This, by $\eqref{uxty0}$, yields
\begin{align}
\lim_{x_2,x_1 \rightarrow x_0{\pm}0}\,
\frac{q(x_2,t_0)-q(x_1,t_0)}{m(x_2,t_0)-m(x_1,t_0)}
&=u_0(y_0) e^{-\frac{t_0}{\tau}}{+}\widetilde{m}_0(y_0)(\tau e^{-\frac{t_0}{\tau}}-\tau)\nonumber\\[2mm]
&=u(x_0,t_0).\label{qqmmu}
\end{align}
Furthermore, by \eqref{Eeq1}, it follows from \eqref{uxty0} and \eqref{qqmmu} that
\begin{align}
\!\!\!\! 
&\lim_{x_2,x_1 \rightarrow x_0{\pm}0}\,\frac{E(x_2,t_0)-E(x_1,t_0)}{m(x_2,t_0)-m(x_1,t_0)}\nonumber\\[1mm]
&=\lim_{x_2,x_1 \rightarrow x_0{\pm}0}\frac{\int_{y_0{-}}^{y_0{+}} \big(u_0(\eta) e^{-\frac{t}{\tau}}+ \widetilde{m}_0(\eta)(\tau e^{-\frac{t}{\tau}}-\tau)\big)u(x(\eta,t_0),t_0){\rm d}m_0(\eta)}{\int_{y_0{-}0}^{y_0{+}0}\,{\rm d}m_0(\eta)}\nonumber\\[1mm]
&=\big(u_0(y_0) e^{-\frac{t}{\tau}}{+}\widetilde{m}_0(y_0)(\tau e^{-\frac{t}{\tau}}-\tau)\big)
u(x(y_0,t_0),t_0)\nonumber\\[1mm]
&=u^2(x_0,t_0).\label{EEmmu2}
\end{align}
This means that $\eqref{qeum}$ holds.

\smallskip
\noindent
{\bf 2.}  Suppose that $y_*(x,t_0)$ can not keep constant in $U_{\delta}(x_0,t_0)$ for any $\delta>0$. 
There exist two subcases.
When $y_*(x_0,t_0)=y^*(x_0,t_0)$ or $y_*(x_0,t_0)<y^*(x_0,t_0)$ with $[m(x_0,t_0)]=0$, it is similar to {\it Subcases} 2 above.

When $y_*(x_0,t_0)<y^*(x_0,t_0)$ with $[m(x_0,t_0)]>0$, without loss of generality,
assume that 
\begin{equation}\label{nuFy*}
    \nu(x_0,t_0)=F(y_*(x_0,t_0);x_0,t_0)=F(y^*(x_0,t_0){+};x_0,t_0).
\end{equation}
Similar to $\eqref{qqmmu}$, by using $\eqref{xprime2}$ and $\eqref{ueq1}$, it can be check that 
\begin{align}
\lim_{x_2,x_1 \rightarrow x_0{\pm}0}\,
\frac{q(x_2,t_0)-q(x_1,t_0)}{m(x_2,t_0)-m(x_1,t_0)}
&=
\frac{\int_{y_*(x_0,t_0){-}}^{y^*(x_0,t_0){+}}u_0(\eta) e^{-\frac{t}{\tau}}+ \widetilde{m}_0(\eta)(\tau e^{-\frac{t}{\tau}}-\tau) \, {\rm d}m_0(\eta)}
{\int_{y_*(x_0,t_0){-}}^{y^*(x_0,t_0){+}} \, {\rm d}m_0(\eta)}\nonumber\\[1mm]
& =u(x_0,t_0).\label{3.3a}
\end{align}
To prove $\eqref{EEmmu2}$ for this case, we first point out that,
\begin{equation}\label{xetaeq}
x(\eta,t_0)=x_0 \qquad\,\,\, {\rm for}\ y_*(x_0,t_0)<\eta<y^*(x_0,t_0),
\end{equation}
and hence
\begin{equation}\label{uxetaeq1}
u(x(\eta,t_0),t_0)=u(x_0,t_0).
\end{equation}

If $\rho_0$ has a positive mass at $y_*(x_0,t_0)$ ({\it resp.,} $y^*(x_0,t_0)$), then  
\begin{equation}\label{xeq1}
\begin{cases}
x(y_*(x_0,t_0),t_0)<x_0 & \quad {\rm if} \  u_0(y_*(x_0,t_0))<c(y_*(x_0,t_0);x_0,t_0),\\[1mm]
x(y_*(x_0,t_0),t_0)=x_0 & \quad {\rm otherwise},
\end{cases} 
\end{equation}
and respectively,
\begin{equation} \label{xeq2}
\begin{cases}
x(y^*(x_0,t_0),t_0)>x_0 &\quad {\rm if} \  u_0(y^*(x_0,t_0))>c(y^*(x_0,t_0);x_0,t_0),\\[1mm]
x(y^*(x_0,t_0),t_0)=x_0 &\quad {\rm otherwise}.
 \end{cases} 
\end{equation}
Furthermore, if $[m_0(y_*(x_0,t_0))]>0$ and $[m_0(y^*(x_0,t_0))]>0$, 
by \eqref{nuFy*}, it follows from Lemma \ref{minlem} that 
$$u_0(y_*(x_0,t_0))\ge c(y_*(x_0,t_0);x_0,t_0)
\qquad\,\,\, u_0(y^*(x_0,t_0))\le c(y^*(x_0,t_0);x_0,t_0).$$
Then, by combining $\eqref{xetaeq}$--$\eqref{uxetaeq1}$ with $\eqref{xeq1}$--$\eqref{xeq2}$, 
it follows from Lemma $\ref{velocitylem}$ that
\begin{align*}
&\lim_{x_2,x_1 \rightarrow x_0{\pm}0}\,
\frac{E(x_2,t_0)-E(x_1,t_0)}{m(x_2,t_0)-m(x_1,t_0)}\\[1mm]
&=\lim_{x_2,x_1 \rightarrow x_0{\pm}0}\frac{\int_{y_*(x_1,t_0){\pm}}^{y_*(x_2,t_0){\pm}} \big(u_0(\eta) e^{-\frac{t}{\tau}}+ \widetilde{m}_0(\eta)(\tau e^{-\frac{t}{\tau}}-\tau)\big)u(x(\eta,t_0),t_0){\rm d}m_0(\eta)}{\int_{y_*(x_1,t_0){\pm}}^{y_*(x_2,t){\pm}}\,{\rm d}m_0(\eta)}\\[1mm]
&=\frac{\int_{y_*(x_0,t_0){-}}^{y^*(x_0,t_0){+}}\big(u_0(\eta) e^{-\frac{t}{\tau}}+ \widetilde{m}_0(\eta)(\tau e^{-\frac{t}{\tau}}-\tau)\big) u(x_0,t_0) \, {\rm d}m_0(\eta)}
{\int_{y_*(x_0,t_0){-}}^{y^*(x_0,t_0){+}} \, {\rm d}m_0(\eta)}\\[1mm]
&=u^2(x_0,t_0).
\end{align*}

To sum up, Lemma $\ref{RNlem}$ has been proved.
\end{proof}

\subsection{Proof of the mass equation}
We now use the first generalized potential $F(y;x,t)$ in $\eqref{potentialF}$ 
and the $\nu(x,t)$ in $\eqref{nuSeq}$ to prove the mass equation.
The key point is the following lemma
on the relations of $\nu(x,t)$ with $m(x,t)$ and $q(x,t)$, respectively.

\begin{Lem}\label{mqnulem}
Let the mass $m(x,t)$ and momentum $q(x,t)$ be defined by $\eqref{meq1}$--$\eqref{qeq1}$.
Then, $\nu(x,t)$ as in $\eqref{nuSeq}$ is Lipschitz continuous and satisfies the following properties{\rm:}
\begin{itemize}
\item [(i)] For any $t>0$ and $x_1,x_2\in\mathbb{R}$,
\begin{equation}\label{mnueq}
\int_{x_1}^{x_2}m(x,t)\,{\rm d}x=\nu(x_1,t)-\nu(x_2,t).
\end{equation}

\item [(ii)] For any $x\in\mathbb{R}$ and $t_2>t_1>0$,
\begin{equation}
\label{qnueq}
\int_{t_1}^{t_2}q(x,t)\,{\rm d}t=\nu(x,t_2)-\nu(x,t_1).
\end{equation}
\end{itemize}
\end{Lem}

\begin{proof}
The proof is divided into two steps accordingly.

\smallskip
 \noindent
{\bf 1.} We first prove $\eqref{mnueq}$. 
Denote $y_*:=y_*(x,t)$ and $y_*':=y_*(x',t)$.
Without loss of generality, assume that $\nu(x,t)=F(y_*;x,t)$
and $\nu(x',t)=F(y_*'{+}0;x',t)$. Then,
\begin{align}\label{nuxeq}
\nu(x,t)-\nu(x',t)
&=[\nu(x,t)-F(y_*'{+};x,t)]+[F(y_*'{+};x,t)-\nu(x',t)]\nonumber\\[1mm] 
&=[\nu(x,t)-F(y_*;x',t)]+[F(y_*;x',t)-\nu(x',t)]. 
\end{align}
Since $\nu(x,t)-F(y_*'{+};x,t)\leq 0$ and $F(y_*;x',t)-\nu(x',t)\geq 0$, 
then it follows from $\eqref{nuxeq}$ that
\begin{align*}
\nu(x,t)-F(y_*;x',t)\leq
\nu(x,t)-\nu(x',t)
\leq F(y_*'{+};x,t)-\nu(x',t). 
\end{align*}
which, by a simple calculation, implies 
\begin{align}\label{nuxeqneq}
(x'-x)m(x,t)\leq
\nu(x,t)-\nu(x',t)
\leq (x'-x)m(x',t). 
\end{align}

By the arbitrariness of $x$ and $x'$, 
$\eqref{mnueq}$ follows by integrating $\eqref{nuxeqneq}$ over $(x_1,x_2)$.

\smallskip
 \noindent
{\bf 2.} We now prove $\eqref{qnueq}$.
Similar to $\eqref{nuxeq}$, $\nu(x,t)-\nu(x,t')$ lies between 
\begin{align}
&\quad \, F(y_*;x,t)-F(y_*;x,t') \nonumber\\[1mm]
&=(t-t') \int _{{-}\infty}^{y_*{-}0} u_0(\eta) \tau \frac{ e^{-\frac{t'}{\tau}}- e^{-\frac{t}{\tau}}}{t-t'}
+\widetilde{m}_0(\eta) 
(\tau^2 \frac{ e^{-\frac{t'}{\tau}}- e^{-\frac{t}{\tau}}}{t-t'}-\tau) \,{\rm d}m_0(\eta),\nonumber\\[1mm]
&=(t-t')q(x,t)+O(1)(t-t')^2,  \label{Fyxtpri}
\end{align}
and
\begin{align}
& \quad\, F(y_*'{+};x,t)-F(y_*'{+};x,t')\nonumber\\[1mm]
&=(t-t') \int _{{-}\infty}^{y_*'{+}0} u_0(\eta) \tau \frac{ e^{-\frac{t'}{\tau}}- e^{-\frac{t}{\tau}}}{t-t'}
+\widetilde{m}_0(\eta) 
(\tau^2 \frac{ e^{-\frac{t'}{\tau}}- e^{-\frac{t}{\tau}}}{t-t'}-\tau) \,{\rm d}m_0(\eta) \,{\rm d}m_0(\eta).\nonumber\\[1mm]
&=(t-t')q(x,t')+O(1)(t-t')^2, \label{Fyxxtpri1}
\end{align}
Integrating over $(t_1,t_2)$ yields $\eqref{qnueq}$.
\end{proof}

We now prove the mass equation in $\eqref{weakeq}$.
From Lemma $\ref{mqnulem}$,  $\nu_x=-m(x,t)$ and $\nu_t=q(x,t)$ in the sense of distribution. 
Then, it follows from $\eqref{qeum}$ that, for any $\varphi \in C_c^{\infty}(\mathbb{R}\times\mathbb{R}^+)$,
\begin{align*} 
\iint m \varphi_t \,{\rm d}x {\rm d}t -\iint u\varphi \,{\rm d}m {\rm d}t
&=\iint(m \varphi_t+q \varphi_x) \,{\rm d}x {\rm d}t \\[2mm]
&=\iint \nu_t\varphi_x-\nu_x \varphi_t \,{\rm d}x {\rm d}t=0.
\end{align*}

\subsection{Proof of the momentum equation}
We now prove the momentum equation.
As in $\eqref{potentialG}$, the second generalized potential $G(y;x,t)$ is given by
\begin{equation*}
G(y;x,t)=\int_{{-}\infty}^{y{-}} \big(u_0(\eta) e^{-\frac{t}{\tau}}+\widetilde{m}_0(\eta)(\tau e^{-\frac{t}{\tau}}-\tau) +k\big)(x(\eta,t)-x) {\rm d}m_0(\eta),
\end{equation*}
where $k$ is any constant satisfying $k>U_0{+}\frac{M}{2} \tau$ 
and $x(\eta,t)$ is defined by $\eqref{xetaeq1}$--$\eqref{xetaeq2}$.

\smallskip
Since $x(\eta,t)$ is increasing in $\eta$, 
it follows from $x(y_*(x,t){-},t) \leq x$ and $x(y^*(x,t){+},t) \geq x$ that 
$G(y;x,t)$ has a finite low bound for any $(x,t)$ and 
$$y_*(x,t)=y^G_*(x,t) \in S_G(x,t),$$ 
where $y^G_*(x,t)$ and $S_G(x,t)$ are defined similar to $\eqref{nuSeq}$--$\eqref{y*eq}$. 

\smallskip
Let $\mu(x,t)=\min\limits_y G(y;x,t)$. By the same arguments, 
Lemma $\ref{mqnulem}$ still holds for $G(y;x,t)$. 
Then, it can be checked that
\begin{equation}\label{muqmE}
\mu _x =-(\bar{q}+k\bar{m}),\qquad
\mu_t=\bar{E}+\omega(x,t)+kh(x,t), 
\end{equation}
where $\bar{m}:=\bar{m}(x,t)$, $\bar{q}:=\bar{q}(x,t)$, and $\bar{E}:=\bar{E}(x,t)$
with $\omega(x,t)$ and $h(x,t)$ are given by
\begin{eqnarray*}
&&\bar{m}(x,t)=
\begin{cases}
\int_{{-}\infty}^{y_*(x,t){-}} {\rm d}m_0(\eta)\ \
{\rm if} \, \mu(x,t)=G(y_*(x,t);x,t),\\[2mm]
\int_{{-}\infty}^{y_*(x,t){+}} {\rm d}m_0(\eta) \ \ {\rm otherwise};
\end{cases}\label{barmy}\\[2mm]
&&\ \bar{q}(x,t)=
\begin{cases}
\int_{{-}\infty}^{y_*(x,t){-}}u_0(\eta) e^{-\frac{t}{\tau}}+ \widetilde{m}_0(\eta)(\tau e^{-\frac{t}{\tau}}-\tau) \,{\rm d}m_0(\eta) 
\ \ {\rm if} \, \mu(x,t)=G(y_*(x,t);x,t),\\[2mm]
\int_{{-}\infty}^{y_*(x,t){+}}u_0(\eta) e^{-\frac{t}{\tau}}+ \widetilde{m}_0(\eta)(\tau e^{-\frac{t}{\tau}}-\tau) \,{\rm d}m_0(\eta) \ \ {\rm otherwise};
\end{cases}\label{barqy}\\[2mm]
&&\bar{E}(x,t)=\\
&&\begin{cases}
\int_{{-}\infty}^{y_*(x,t){-}}\big(u_0(\eta) e^{-\frac{t}{\tau}}+ \widetilde{m}_0(\eta)(\tau e^{-\frac{t}{\tau}}-\tau)\big)u(x(\eta,t),t) \,{\rm d}m_0(\eta) 
\  {\rm if} \, \mu(x,t)=G(y_*(x,t);x,t),\\[2mm]
\int_{{-}\infty}^{y_*(x,t){+}}\big(u_0(\eta) e^{-\frac{t}{\tau}}+ \widetilde{m}_0(\eta)(\tau e^{-\frac{t}{\tau}}-\tau)\big)u(x(\eta,t),t)\,{\rm d}m_0(\eta) \ {\rm otherwise};
\end{cases}\label{barEy}\\[2mm]
&&\ \omega(x,t)=
\begin{cases}
-\frac{1}{\tau}e^{-\frac{t}{\tau}} \int_{{-}\infty}^{y_*(x,t){-}}(u_0(\eta)+\tau \widetilde{m}_0(\eta))(x(\eta,t)-x) \,{\rm d}m_0(\eta) 
\ \ {\rm if} \, \mu(x,t)=G(y_*(x,t);x,t),\\[2mm]
-\frac{1}{\tau}e^{-\frac{t}{\tau}} \int_{{-}\infty}^{y_*(x,t){-}}(u_0(\eta)+\tau \widetilde{m}_0(\eta))(x(\eta,t)-x)\,{\rm d}m_0(\eta) \ \ {\rm otherwise};
\end{cases}\label{baromegay}\\[2mm]
&&\ h(x,t)=
\begin{cases}
\int_{{-}\infty}^{y_*(x,t){-}} u(x(\eta,t),t)\, {\rm d}m_0(\eta) 
\ \ {\rm if} \, \mu(x,t)=G(y_*(x,t);x,t),\\[2mm]
\int_{{-}\infty}^{y_*(x,t){+}} u(x(\eta,t),t)\,{\rm d}m_0(\eta) \ \ {\rm otherwise}.
\end{cases}\label{barvy}
\end{eqnarray*}

We claim: {\it For any $t>0$, it holds that}
$$ m=\bar{m},\qquad q=\bar{q},\qquad E=\bar{E}\qquad \quad a.e..$$

In fact, due to $y_*(x,t)=y^G_*(x,t)$, $m(x,t) \neq \bar{m}(x,t)$ is possible only if $[m_0(y_*(x,t))]>0$. 
Noticing that the points $(y,0)$ such that $[m_0(y)]>0$ are at most countable, 
since the map from $(x,t)$ to $(y_*(x,t),0)$ is not one to one, 
we need consider the case that an interval maps to one point, {\it i.e.,} 
there exist more than one forward generalized characteristics emitting from $(y_*(x,t),0)$. 
For such case, we let $x \in (a(\eta,t),b(\eta,t))$ for $a(\eta,t)$ and $b(\eta,t)$ given by $\eqref{aetat}$--$\eqref{betat}$. 
There are two subcases:

If $c(y_*(x,t);x,t)>u_0(y_*(x,t))$, by Lemma $\ref{minlem}$, 
both $F(y;x,t)$ and $G(y;x,t)$ do not arrive the minimum at $y_*(x,t)$ so that $m(x,t)=\bar{m}(x,t)$; 
and if $c(y_*(x,t);x,t)\leq u_0(y_*(x,t))$, 
both $F(y;x,t)$ and $G(y;x,t)$ arrive the minimum at $y_*(x,t)$. 

Thus, $F(y;x,t)$ and $G(y;x,t)$ always achieve or do not achieve the minimum at $y_*(x,t)$ at the same time. 
This infers $m(x,t)=\bar{m}(x,t)$. 
By the same arguments, it is direct to check that $q(x,t)=\bar{q}(x,t)$ and $E(x,t)=\bar{E}(x,t)$.

\vspace{2pt}
Since $k>U_0{+}\frac{M}{2}\tau$ is arbitrary, it follows from $\eqref{muqmE}$ that
\begin{equation}\label{thetaqEw}
\theta_x=-q,\qquad \theta_t=E+\omega(x,t),
\end{equation}
where $\theta:=\theta(x,t)$ is given by
\begin{align*}
\ \
&\theta(x,t)=\\
&\begin{cases}
\int_{{-}\infty}^{y_*(x,t){-}}\big(u_0(\eta) e^{-\frac{t}{\tau}}+ \widetilde{m}_0(\eta)(\tau e^{-\frac{t}{\tau}}-\tau)\big)(x(\eta,t){-}x) \,{\rm d}m_0(\eta)
\  {\rm if} \ \mu(x,t)=G(y_*(x,t);x,t),\\[2mm]
\int_{{-}\infty}^{y_*(x,t){+}}\big(u_0(\eta) e^{-\frac{t}{\tau}}+ \widetilde{m}_0(\eta)(\tau e^{-\frac{t}{\tau}}-\tau)\big)(x(\eta,t){-}x) \,{\rm d}m_0(\eta)\  {\rm otherwise}.
\end{cases}
\end{align*}

As in $\eqref{potentialH}$, the third generalized potential $H(y;x,t)$ is given by
\begin{equation*}
H(y;x,t)=-\frac{1}{\tau}e^{-\frac{t}{\tau}} \int_{{-}\infty}^{y{-}}(u_0(\eta)+\tau \widetilde{m}_0(\eta)+k)(x(\eta,t)-x) \,{\rm d}m_0(\eta).
\end{equation*}
where $k>U_0+\frac{M}{2} \tau$.
By the same arguments as $G(y;x,t)$, for any $(x,t)$,
$H(y;x,t)$ has a finite low bound  
and the minimum points of $H(y;x,t)$ are the same as that of $G(y;,x,t)$. 
Similar to the proof of Lemma \ref{mqnulem}, in the sense of distribution, we have
\begin{align*}
&\frac{\partial}{\partial x}\big(\min_y H(y;x,t)\big)=\frac{1}{\tau}e^{-\frac{t}{\tau}} \int_{{-}\infty}^{y_*(x,t){\pm}}u_0(\eta)+\tau \widetilde{m}_0(\eta)+k \,{\rm d}m_0(\eta)\\[1mm]
&=\frac{1}{\tau}\int_{{-}\infty}^{y_*(x,t){\pm}} u_0(\eta)e^{-\frac{t}{\tau}} +\widetilde{m}_0(\eta)(\tau e^{-\frac{t}{\tau}}-\tau) \,{\rm d}m_0(\eta)+\int_{{-}\infty}^{y_*(x,t){\pm}} \widetilde{m}_0(\eta) \,{\rm d}m_0(\eta)\\[1mm]
&\quad +k \frac{1}{\tau}e^{-\frac{t}{\tau}}\int_{{-}\infty}^{y_*(x,t){\pm}} \,{\rm d}m_0(\eta) \\[1mm]
&=\frac{1}{\tau} q(x,t)+\int_{{-}\infty}^{y_*(x,t){\pm}} \widetilde{m}_0(\eta) \,{\rm d}m_0(\eta)+k \frac{1}{\tau}e^{-\frac{t}{\tau}}\int_{{-}\infty}^{y_*(x,t){\pm}} \,{\rm d}m_0(\eta),
\end{align*}
which, by the arbitrariness of $k$, implies
\begin{align}\label{3.3}
\omega_x(x,t)=\frac{1}{\tau} q(x,t)+\int_{{-}\infty}^{y_*(x,t){\pm}} \widetilde{m}_0(\eta) \,{\rm d}m_0(\eta).
\end{align}
By the BV chain rule,  
$\widetilde{m}_0(y){\rm d}m_0(y)={\rm d}(\frac{1}{2}m_0^2(y)-\frac{M}{2}m_0(y))$, combining with the definition of $m(x,t)$ in \eqref{meq1}, yields
\begin{align}
\int_{{-}\infty}^{y_*(x,t){\pm}} \widetilde{m}_0(\eta) \,{\rm d}m_0(\eta)
&=\frac{1}{2}m_0^2(y_*(x,t){\pm})-\frac{M}{2}m_0(y_*(x,t){\pm})\nonumber\\[1mm]
&=\frac{1}{2}m^2(x,t)-\frac{M}{2}m(x,t).\label{3.4}
\end{align}
Then
\begin{equation}\label{omegaxxt}
\omega_x(x,t)=\frac{1}{\tau} q(x,t)+\frac{1}{2}m^2(x,t)-\frac{M}{2}m(x,t).
\end{equation}

Moreover, by the BV chain rule, it is direct to check that 
\begin{equation}\label{BVm2}
{\rm d}\Big(\frac{1}{2}m^2(x,t)-\frac{M}{2}m(x,t)\Big)=\widetilde{m}(x,t)\,{\rm d}m(x,t),
\end{equation}
which, by combining $\eqref{qeum}$ and $\eqref{thetaqEw}$ with $\eqref{omegaxxt}$, implies
\begin{align*}
&\ \iint \psi_tu+\psi_xu^2 \,{\rm d}m {\rm d}t-\iint (\widetilde{m}+\frac{u}{\tau})\psi \,{\rm d}m {\rm d}t\\[1mm]
&=\iint \psi_t \,{\rm d}q {\rm d}t+\psi_x \,{\rm d}E {\rm d}t
-\iint \psi \,{\rm d}\Big(\frac{1}{\tau}q(x,t)+\frac{1}{2}m^2(x,t)-\frac{M}{2}m(x,t)\Big) {\rm d}t\\[1mm]
&={-}\iint \psi_{xt}q+\psi_{xx}(E+\omega) \,{\rm d}x {\rm d}t\\[1mm]
&={-}\iint -\psi_{xt} \theta_x+\psi_{xx}\theta_t \,{\rm d}x {\rm d}t
=0. 
\end{align*}

To prove $(m,u)$ is an entropy solution, 
it remains to show that $m_x$, $um_x$, and $u^2 m_x$ are weakly continuous at $t=0$.
The remaining proof is divided into three steps accordingly.

\smallskip
\noindent
{\bf 1.} We first prove the weak continuity of $m_x$ at $t=0$. 
Since $m_0(x)$ is continuous except at most countable points,
we only need to study the $x$ at which $m_0(x)$ is continuous. 
If there exists $t_0$ such that 
$c(y_*(x,t);x,t)>u_0(y_*(x,t))$ with $c(y_*(x,t){+};x,t)>U_0$ for any $t \in (0,t_0]$, then  
$x\in \mathbb{R}\backslash {\rm spt}\{\rho_0\}$.
For $t<t_0$, we have
$$m(x,t)=m_0(y_*(x,t){+})=m_0(x).$$
The case that $c(y_*(x,t);x,t)<u_0(y_*(x,t))$ with $c(y_*(x,t){+};x,t)<{-}U_0$ for any $t \in (0,t_0]$ is similar to above.

If there exists $t_0$ such that for any $t\in (0,t_0]$, $c(y_*(x,t);x,t)> u_0(y_*(x,t))$ but $c(y_*(x,t){+};x,t)\le U_0$, or $c(y_*(x,t);x,t)< u_0(y_*(x,t))$ but $c(y_*(x,t){+};x,t)\ge- U_0$, or $c(y_*(x,t);x,t)= u_0(y_*(x,t))$, we have 
$y_*(x,t) \rightarrow x$ as $t \rightarrow 0{+}$. 
This, by $\eqref{meq}$, implies  
$$m(x,t) \rightarrow m_0(x)\qquad {\rm as}\ t \rightarrow 0{+}.$$  
Thus, we have proved that 
$$m(x,t) \rightarrow m_0(x)\qquad {\it a.e.}\quad {\rm as}\ t\rightarrow 0{+}.$$ 
This means the weak continuity of $m_x$ at $t=0$. 

\smallskip
\noindent
{\bf 2.} By the same arguments as in Step 1,
it is direct to check that $q_x=u m_x$ is weakly continuous at $t=0$.

\smallskip
\noindent
{\bf 3.} We now prove the weak continuity of $E_x=u^2 m_x$ at $t=0$.
By $\eqref{Eeq}$, it suffices to show that, as $t \rightarrow 0{+}$, 
$u(x(\eta,t),t) \rightarrow u_0(\eta)$ almost everywhere with respect to measure $\rho_0$.
We only consider the Lebesgue points of $u_0(\eta)$ with respect to $\rho_0$.
In fact, for $\eta_0$ satisfies that, for $\eta_1<\eta_0<\eta_2$,  
\begin{equation}\label{u0etam}
u_0(\eta_0)=\lim_{\eta_2,\eta_1 \rightarrow \eta_0 {\pm}}\frac{\int_{\eta_1{-}}^{\eta_2{+}}u_0(\eta) \,{\rm d}m_0(\eta)}{\int_{\eta_1{-}}^{\eta_2{+}} \,{\rm d}m_0(\eta)},
\end{equation}
we need to show 
\begin{equation}\label{uxetatt}
u(x(\eta_0,t),t) \rightarrow u_0(\eta_0)\qquad {\rm as}\ t \rightarrow 0{+}.
\end{equation}

Since $x(\eta,t)$ is increasing in $\eta$, then $x(\eta{-},t)\leq x(\eta{+},t)$.
If $x(\eta{-},t)=x(\eta{+},t)$, then for $\epsilon>0$, $x(\eta{-}\epsilon,t)<x(\eta,t)<x(\eta{+}\epsilon,t)$ holds for $t$ sufficiently small.
In \eqref{qqmmu}, let $x_0=x(\eta,t)$, $x_1=x(\eta{-}\epsilon,t)$ and $x_2=x(\eta{+}\epsilon,t)$, then
\begin{align}
\lim_{t\rightarrow0{+}}u(x(\eta,t),t)
&=\lim_{t\rightarrow0{+}}\lim_{\epsilon\rightarrow0}\frac{q(x(\eta{+}\epsilon,t),t)-q(x(\eta{-}\epsilon,t),t)}{m(x(\eta{+}\epsilon,t),t)-m(x(\eta{-}\epsilon,t),t)}\nonumber\\[1mm]
&=\lim_{\epsilon\rightarrow0}\lim_{t\rightarrow0{+}}\frac{q(x(\eta{+}\epsilon,t),t)-q(x(\eta{-}\epsilon,t),t)}{m(x(\eta{+}\epsilon,t),t)-m(x(\eta{-}\epsilon,t),t)}\nonumber\\[1mm]
&=\lim_{\epsilon\rightarrow0}\frac{q_0(\eta+\epsilon)-q_0(\eta{-}\epsilon)}{m_0(\eta{+}\epsilon)-m_0(\eta{-}\epsilon)}\nonumber\\[1mm]
&=u_0(\eta),\label{3.5a}
\end{align}
where the second and third equals are due to the weak continuity of $m_x$ and $u m_x$, the last equal is due to \eqref{u0etam}.

If $x(\eta{-},t)<x(\eta{+},t)$ for $t$ sufficiently small, assume $[m_0(\eta)]>0$, then $\nu(x,t)=F(\eta;x,t)$ for $x(\eta{-},t)\leq x \leq x(\eta,t)$ and $\nu(x,t)=F(\eta{+};x,t)$ for $x(\eta,t)< x \leq x(\eta{+},t)$, so that the support of $\rho$ is concentrated on $x(\eta,t)$. Similar to \eqref{3.5a}, \eqref{uxetatt} holds.

\smallskip
Up to now, we have completed the proof of Theorem $\ref{ExisThm}$.

\section{Proof of the Formula for drift equations}

In this section, we construct the mass $\bar{m}(x,t)$ and momentum $\bar{q}(x,t)$ for drift equations, and show their relations with the minimum value $\bar{\nu}(x,t)$, by which the drift equations is proved.

\vspace{10pt}

For any fixed $(x,t)\in \mathbb{R}\times\mathbb{R}^+$, it follows from the left continuity of $\bar{F}(\cdot;x,t)$ with respect to $y$ that, for any $y_0\in \bar{S}(x,t)$,
\begin{align}\label{4.1}
\bar{\nu}(x,t)=
\begin{cases}
\bar{F}(y_0;x,t) & {\rm if}\ \bar{F}(y;x,t)\ {\rm achieves \ its \ minimum \ at}\ y_0, \\[1mm]
\bar{F}(y_0{+};x,t) & {\rm otherwise}.
\end{cases}
\end{align}
If $y_0\in \bar{S}(x,t)$ and $[m_0(y_0)]=m_0(y_0{+})-m_0(y_0{-})>0$, then similar to the proof of Lemma \ref{nueq}, 
\begin{align}\label{4.2}
\bar{\nu}(x,t)=
\begin{cases}
\bar{F}(y_0;x,t)  &{\rm if}\ \frac{x-y_0}{t}\leq -\widetilde{m}_0(y_0),\\[1mm]
\bar{F}(y_0{+};x,t)  &{\rm if}\ \frac{x-y_0}{t}> -\widetilde{m}_0(y_0).
\end{cases}
\end{align}
We define the cumulative mass and momentum respectively as follows:
\begin{align}\label{2.12m2}
\bar{m}(x,t):=
\begin{cases}
\int_{{-}\infty}^{\bar{y}_*(x,t){-}} \,{\rm d}m_0(\eta) \quad
& {\rm if}\, \bar{\nu}(x,t)=\bar{F}(y_*(x,t),x,t),\\[1mm]
\int_{{-}\infty}^{\bar{y}_*(x,t){+}} \,{\rm d}m_0(\eta) \quad
& {\rm otherwise},
\end{cases}
\end{align}
\begin{align}\label{2.12q2}
\bar{q}(x,t):=
\begin{cases}
\int_{{-}\infty}^{\bar{y}_*(x,t){-}}-\widetilde{m}_0(\eta) \,{\rm d}m_0(\eta) \quad
& {\rm if}\, \bar{\nu}(x,t)=\bar{F}(y_*(x,t),x,t),\\[1mm]
\int_{{-}\infty}^{\bar{y}_*(x,t){+}} -\widetilde{m}_0(\eta) \,{\rm d}m_0(\eta) \quad
& {\rm otherwise}.
\end{cases}
\end{align}
To prove the drift equations, the key point is the following relations among $\bar{\nu}(x,t)$, $\bar{m}(x,t)$, and $\Bar{q}(x,t)$. 

\begin{Lem}\label{4.3}
Let the mass $\bar{m}(x,t)$ and momentum $\bar{q}(x,t)$ be defined by $\eqref{2.12m2}$--$\eqref{2.12q2}$.
Then, $\bar{\nu}(x,t)$ as in $\eqref{2.10}$ is Lipschitz continuous and satisfies the following properties{\rm:}
\begin{itemize}
\item [(i)] For any $t>0$ and $x_1,x_2\in\mathbb{R}$,
\begin{equation}\label{4.4}
\int_{x_1}^{x_2}\bar{m}(x,t)\,{\rm d}x=\bar{\nu}(x_1,t)-\bar{\nu}(x_2,t).
\end{equation}

\item [(ii)] For any $x\in\mathbb{R}$ and $t_2>t_1>0$,
\begin{equation}\label{4.5}
\int_{t_1}^{t_2}\bar{q}(x,t)\,{\rm d}t=\bar{\nu}(x,t_2)-\bar{\nu}(x,t_1).
\end{equation}
\end{itemize}
\end{Lem}

\begin{proof}
We only prove \eqref{4.5}, since the proof of \eqref{4.4} is the same as \eqref{mnueq}.

Denote $\bar{y}_*:=\bar{y}_*(x,t)$ and $\bar{y}_*':=\bar{y}_*(x,t')$.
Without loss of generality, assume that $\bar{\nu}(x,t)=\bar{F}(y_*;x,t)$
and $\bar{\nu}(x,t')=\bar{F}(y_*'{+};x,t')$.
Similar to \eqref{nuxeqneq}, $\bar{\nu}(x,t)-\nu(x,t')$ lies between
\begin{align}\label{4.6}
\bar{F}(\bar{y}_*;x,t)-\bar{F}(\bar{y}_*;x,t')=(t-t')\int_{-\infty}^{\bar{y}_*{-}} -\widetilde{m}_0(\eta)\,{\rm d}m_0(\eta)=(t-t')\bar{q}(x,t),
\end{align}
and
\begin{align}\label{4.7}
\bar{F}(\bar{y}_*^{\prime}{+};x,t)-\bar{F}(\bar{y}_*^{\prime}{+};x,t')=(t-t')\int_{-\infty}^{\bar{y}_*^{\prime}{+}}-\widetilde{m}_0(\eta)\,{\rm d}m_0(\eta)=(t-t')\bar{q}(x,t').
\end{align}
By the arbitrariness of $t$ and $t'$, \eqref{4.5} follows from \eqref{4.6} and \eqref{4.7} by integrating over $(t_1,t_2)$.
\end{proof}

{\it Proof of Theorem \ref{Ex2}.} 

Since $-\widetilde{m}_0(\eta)\,{\rm d}m_0(\eta)=-{\rm d}\big(\frac{1}{2}m_0^2(\eta)-\frac{M}{2}m_0(\eta)\big)$ in the sense of BV derivative,
then from \eqref{2.12m2}--\eqref{2.12q2},
\begin{align}
\bar{q}(x,t)&=\int_{-\infty}^{\bar{y}_*(x,t){\pm}} -\widetilde{m}_0(\eta)\,{\rm d}m_0(\eta)=-\frac{1}{2}m_0^2(\bar{y}_*(x,t){\pm})+\frac{M}{2} m_0(y_*(x,t){\pm})\nonumber\\[1mm]
&=-\frac{1}{2}\bar{m}^2(x,t)+\frac{M}{2}\bar{m}(x,t),\nonumber
\end{align}
which, by \eqref{2.6}, implies that
\begin{align}\label{4.8}
\partial_x \bar{q}=\partial_x \big(-\frac{1}{2}\bar{m}^2+\frac{M}{2}\bar{m}\big)=-\widehat{m}\, \partial_x \bar{m},
\end{align}
in the sense of BV derivative.
From Lemma $\ref{4.3}$,  $\bar{\nu}_x=-\bar{m}(x,t)$ and $\bar{\nu}_t=\bar{q}(x,t)$ in the sense of distribution. 
Then for any $\varphi \in C_c^{\infty}(\mathbb{R}\times\mathbb{R}^+)$,
\begin{align*}
\iint \varphi_t \bar{m} \,{\rm d}x{\rm d}t+\iint \varphi \widehat{m} \,{\rm d}\bar{m}{\rm d}t
&=\iint \varphi_t \bar{m} \,{\rm d}x{\rm d}t-\iint \varphi\,{\rm d}\bar{q}{\rm d}t\nonumber\\[1mm]
&=\iint -\varphi_t \bar{\nu}_x+\varphi_x \bar{\nu}_t \,{\rm d}x{\rm d}t
=0.
\end{align*}

Now we prove the measure $\bar{\rho}$ weakly converges to $\rho_0$ as $t\rightarrow 0+$. 
It is sufficient to show that for $x$ at which $m_0(x)$ is continuous, $\bar{m}(x,t)\rightarrow m_0(x)$ as $t\rightarrow 0+$.
First, for $x\in {\rm spt}\{\rho_0\}$, by \eqref{2.9a}, 
\begin{align*}
x-\frac{t M}{2}\leq \bar{y}_*(x,t) \leq x+\frac{t M}{2},
\end{align*}
which implies $\bar{y}_*(x,t)\rightarrow x$ as $t \rightarrow 0+$, so that $\bar{m}(x,t)\rightarrow m_0(x)$.
If $x\in \mathbb{R}\backslash {\rm spt}\{\rho_0\}$, then there exists open interval $(a,b)\subset \mathbb{R}\backslash {\rm spt}\{\rho_0\}$ with $a,b \in {\rm spt}\{\rho_0\}$ and $x\in (a,b)$.
Since $\widetilde{m}_0(\eta)\in [-\frac{M}{2},\frac{M}{2}]$ and 
\begin{align*}
\bar{F}(y;x,t)=t\int_{-\infty}^{y-}-\widetilde{m}_0(\eta)-\frac{x-\eta}{t}\,{\rm d}m_0(\eta),
\end{align*}
then $\frac{x-a}{t}>\frac{M}{2}$ and $\frac{x-b}{t}<-\frac{M}{2}$ as $t$ sufficiently small,
which yields $\bar{y}_*(x,t)=a$ and $\bar{\nu}(x,t)=\bar{F}(a{+};x,t)$.
Thus $\bar{m}(x,t)=m_0(a{+})=m_0(x)$. 
The weak continuity of $\bar{\rho}$ at $t=0$ is proved.

\section{Proof of the Relaxation limit}
This section is devoted to the proof of Theorem \ref{thm3}.
First, by slow time scaling, we show that the minimum points of $F^{\tau}(y;x,t)$ converge to the minimum points of $\bar{F}(y;x,t)$. 
Since there exists only one minimum point for almost all $(x,t)$, then the mass $m^{\tau}(x,t)$ of pressureless Euler--Poisson equations converges to the mass $\bar{m}(x,t)$ of drift equations almost everywhere as the relaxation time $\tau\rightarrow 0$, which implies the weak convergence from $\rho^\tau$ to $\bar{\rho}$.
Similarly, the momentum $q(x,t)$ converges to $\bar{q}(x,t)$ almost everywhere.
By the Radon-Nikodym derivatives $u(x,t)=\frac{q_x}{m_x}$, the velocity $u(x,t)\rightarrow \bar{u}(x,t)$ with respect to measure $\bar{m}_x$ as $\tau \rightarrow 0$. 

\vspace{10pt}

By slow time scaling, $F(y;x,t)$ given by \eqref{potentialF} transforms into
\begin{align}\label{2.14}
F^{\tau}(y;x,t){:=}F(y;x,\frac{t}{\tau}){=}\int_{-\infty}^{y-}\eta{+}u_0(\eta)(\tau{-}\tau e^{-\frac{t}{\tau^2}}){+}\widetilde{m}_0(\eta)(\tau^2-\tau^2 e^{-\frac{t}{\tau^2}}{-}t){-}x\,{\rm d}m_0(\eta),
\end{align}
and the minimum point \eqref{y*eq} transforms into
\begin{align}
y_*^{\tau}(x,t):=y_*(x,\frac{t}{\tau}).
\end{align}
As $\tau \rightarrow 0+$, by \eqref{2.9}, it is obvious that
\begin{align}
\lim_{\tau \rightarrow 0+}F^{\tau}(y;x,t)=\bar{F}(y;x,t).
\end{align}
For $(x,t)\in {\rm spt}\{\rho\}$ with $\rho=m_x$, it follows from \eqref{2.1b} that, after slow time scaling,
\begin{align}\label{2.15}
x{-}U_0(\tau{-}\tau e^{-\frac{t}{\tau^2}}){+}\frac{M}{2}(\tau^2{-}\tau^2 e^{-\frac{t}{\tau^2}}{-}t)
\leq y_*^{\tau}(x,t) 
\leq x{+}U_0(\tau{-}\tau e^{-\frac{t}{\tau^2}}){-}\frac{M}{2}(\tau^2{-}\tau^2 e^{-\frac{t}{\tau^2}}{-}t),
\end{align}
since the left-hand side and the right-hand side of \eqref{2.15} converge to $x-\frac{t M}{2}$ and $x+\frac{t M}{2}$ respectively, 
then $\{y_*^{\tau}(x,t)\}$ is bounded as $\tau$ sufficiently small, 
so that there exists a convergent subsequence (still denoted as $y_*^{\tau}(x,t)$) with the limits denoted as
\begin{align}\label{2.16}
\hat{y}(x,t):=\lim_{\tau\rightarrow 0+}y_*^{\tau}(x,t).
\end{align}
It follows from $y_*(x,t)\in S(x,t)$ that $F^{\tau}(y_*^{\tau}(x,t){\pm};x,t)\leq F^{\tau}(y;x,t)$ holds for any $y\in\mathbb{R}$.
Let $\tau\rightarrow 0+$, by the left-continuity of $F$ with respect to $y$, we have 
\begin{align*}
\bar{F}(\hat{y}(x,t){\pm};x,t)\leq \bar{F}(y;x,t), \qquad {\rm for\ any} \ y\in\mathbb{R},
\end{align*}
which implies 
\begin{align}\label{2.17}
\hat{y}(x,t)\in \bar{S}(x,t) \cap {\rm spt}\{\rho_0\},
\end{align}
where $\bar{S}(x,t)$ is given in \eqref{2.10}.

From \eqref{2.12}, for fixed $t>0$ and any $x_1<x_2$,
\begin{align*}
\big(\bar{y}_*(x_1,t),\bar{y}^*(x_1,t)\big) \cap \big(\bar{y}_*(x_2,t),\bar{y}^*(x_2,t)\big)=\varnothing,
\end{align*}
which implies that there exists at most countable point $x$ satisfying $\bar{y}_*(x,t)<\bar{y}^*(x,t)$.
Thus for almost everywhere $x\in\mathbb{R}$, we have $\bar{y}_*(x,t)=\bar{y}^*(x,t)$.

In order to prove the measure $\rho^{\tau}(x,t)$ weakly converges to measure $\bar{\rho}(x,t)$ for fixed $t$ as $\tau \rightarrow 0+$, 
it is sufficient to show $m^{\tau}(x,t)\rightarrow \bar{m}(x,t)$ at $x$ which satisfies $\bar{y}_*(x,t)=\bar{y}^*(x,t)$.

For $(x,t)$ satisfying $\bar{y}_*(x,t)=\bar{y}^*(x,t)$, \eqref{2.17} yields $\hat{y}(x,t)=\bar{y}_*(x,t)$.
Without loss of generality, assume $\bar{\nu}(x,t)=\bar{F}(\bar{y}_*(x,t),x,t)$,
then it follows from \eqref{meq}, \eqref{qeq}, \eqref{2.12m}, and \eqref{2.12q} that, as $\tau\rightarrow 0+$,
\begin{align}
&m^{\tau}(x,t):=m(x,\frac{t}{\tau})=\int_{-\infty}^{y_*^{\tau}(x,t){\pm}}\,{\rm d}m_0(\eta)\longrightarrow \int_{-\infty}^{\hat{y}(x,t){-}}\,{\rm d}m_0(\eta)=\bar{m}(x,t),\label{2.18}
\end{align}
which implies the weak convergence from $\rho^{\tau}$ to $\bar{\rho}$.

It follows from \eqref{4.8} that, for $(x,t)\in {\rm spt}\{\bar{\rho}\}$ and in the sense of Radon-Nikodym derivative,
\begin{align}
\bar{u}(x,t):=-\widehat{m}(x,t)=\frac{\bar{q}_x}{\bar{m}_x}.
\end{align}
For $(x,t)$ satisfying $\bar{y}_*(x,t)=\bar{y}^*(x,t)$, similar to \eqref{2.18},
\begin{align}
&\frac{1}{\tau}q^{\tau}(x,t):=\frac{1}{\tau}q(x,\frac{t}{\tau})=\int_{-\infty}^{y_*^{\tau}(x,t){\pm}}u_0(\eta)\frac{1}{\tau} e^{-\frac{t}{\tau}}+\widetilde{m}_0(\eta)(e^{-\frac{t}{\tau}}-1)\,{\rm d}m_0(\eta)\nonumber\\[1mm]
&\longrightarrow \int_{-\infty}^{\hat{y}(x,t)-} -\widetilde{m}_0(\eta)\,{\rm d}m_0(\eta)=\bar{q}(x,t).\label{2.19}
\end{align}
For $(x,t)\in {\rm spt}\{\bar{\rho}\}$, we can choose $x_1<x<x_2$ with $x_1, x_2$ being continuous points of $\bar{m}$, 
then by $q_x=u m_x$ in \eqref{qeum},
\begin{align}
\lim_{\tau\rightarrow 0+}u^{\tau}(x,t)
&=\lim_{\tau\rightarrow 0+}\frac{1}{\tau}u(x,\frac{t}{\tau})
=\lim_{\tau\rightarrow 0+}\lim_{x_1,x_2\rightarrow x{\pm}}\frac{\frac{1}{\tau}q^{\tau}(x_2,t)-\frac{1}{\tau}q^{\tau}(x_1,t)}{m^{\tau}(x_2,t)-m^{\tau}(x_1,t)}\nonumber\\[1mm]
&=\lim_{x_1,x_2\rightarrow x{\pm}}\lim_{\tau\rightarrow 0+}\frac{\frac{1}{\tau}q^{\tau}(x_2,t)-\frac{1}{\tau}q^{\tau}(x_1,t)}{m^{\tau}(x_2,t)-m^{\tau}(x_1,t)}
=\lim_{x_1,x_2\rightarrow x{\pm}}\frac{\bar{q}(x_2,t)-\bar{q}(x_1,t)}{\bar{m}(x_2,t)-\bar{m}(x_1,t)}\nonumber\\[1mm]
&=\bar{u}(x,t).\label{2.20}
\end{align}

\section{Uniqueness of Entropy Solution to pressureless Euler--Poisson equations}
In this section, we call the entropy solution given by $\eqref{meq}$ and $\eqref{ueq}$ 
in Theorem \ref{ExisThm} to be the standard entropy solution and denote it by $(m^s,u^s)$. 
By Lemma $\ref{RNlem}$, 
$u^s=\frac{{\rm d}q^s}{{\rm d}m^s}$ in the sense of Radon--Nikodym derivatives,
where $q^s$ is defined by $\eqref{qeq}$. So we also call $(m^s,q^s)$ to be the standard entropy solution alternatively.

To prove the uniqueness for entropy solutions, we need to show that
any entropy solution $(m,u)$ must coincide with 
the standard entropy solution $(m^s,u^s)$.
The proof of the uniqueness for entropy solutions is divided into three parts:  
Firstly, for any fixed $t_1>0$, we prove that 
any entropy solution $(m,q)$ can be expressed by its value at $t=t_1$.
Secondly, by letting $t_1\rightarrow 0{+}$, 
we obtain the expressions of $(m,q)$ with respect to initial data in $\eqref{mx0t0}$--$\eqref{qx0t02}$,
then by proving $\xi(x,t)\in S(x,t)$, it follows from Lemma $\ref{sigmasol}$ that 
any entropy solution $(m,u)$ equals to the standard entropy solution $(m^s,u^s)$.

\begin{Lem}\label{sigmasol}
Let the standard entropy solution $(m^s$, $q^s)$ be given by $\eqref{meq}$ and $\eqref{qeq}$,
and for any $y_0 \in S(x,t)$, let $(m^{\sigma},q^{\sigma})$ be defined by 
\begin{eqnarray}
&&m^{\sigma}(x,t)=
\begin{cases}
\int_{{-}\infty}^{y_0{-}} {\rm d}m_0(\eta)
\qquad {\rm if} \ \nu(x,t)=F(y_0;x,t),\\[2mm]
\int_{{-}\infty}^{y_0{+}}{\rm d}m_0(\eta)\qquad {\rm otherwise},
 \end{cases}\label{meq2}\\[2mm]
&&\,\, q^{\sigma}(x,t)=
\begin{cases}
\int_{{-}\infty}^{y_0{-}}u_0(\eta) e^{-\frac{t}{\tau}}{+}\widetilde{m}_0(\eta)(\tau e^{-\frac{t}{\tau}}{-}\tau) \,{\rm d}m_0(\eta)
\quad {\rm if} \ \nu(x,t)=F(y_0;x,t),\\[2mm] 
\int_{{-}\infty}^{y_0{+}}u_0(\eta) e^{-\frac{t}{\tau}}{+}\widetilde{m}_0(\eta)(\tau e^{-\frac{t}{\tau}}{-}\tau) \,{\rm d}m_0(\eta)
\quad {\rm otherwise}.
\end{cases}\label{qeq2}
\end{eqnarray}
Then $m^s(x,t)=m^{\sigma}(x,t)$ and $q^s(x,t)=q^{\sigma}(x,t)$ almost everywhere, 
and $u^s(x,t)=\frac{{\rm d}q^s}{{\rm d}m^s}=\frac{{\rm d}q^{\sigma}}{{\rm d}m^{\sigma}}=u^{\sigma}(x,t)$ with respect to the measure $m_x^s=m_x^{\sigma}$.
\end{Lem}

\begin{proof}
Lemma $\ref{sigmasol}$ follows from the fact that:
For any fixed $t>0$, the points $(x,t)$ satisfying $y_*(x,t)<y^*(x,t)$ are at most countable; 
and for the case that $y_*(x,t)=y^*(x,t)$, 
it holds that $\rho_0=0$ on $(y_*(x,t),y_0)$ if $y_*(x,t)<y_0$.
\end{proof}

\subsection{Formula of entropy solution with respect to value in $t=t_1$.}
Assume that $(\rho,u)=(m_x,u)$ is an entropy solution in the sense of Definition $\ref{entropylem}$. 
In order to get the formula of $(m(x,t), u(x,t))$ with respect to the data $(m(\eta,t_1),u(\eta,t_1))$, we need to consider the following Cauchy problem in $(x,t)\in \mathbb{R}\times [t_1,T]$,
\begin{equation}\label{mmaineq2}
\begin{cases}
m_t+um_x=0,\\[1mm]
(m_xu)_t+(m_x u^2)_x=-\widetilde{m}m_x-\frac{u m_x}{\tau},\\[1mm]
(m,u)|_{t=t_1}=(m(\eta,t_1),u(\eta,t_1)).
\end{cases}
\end{equation}
The proof is divided into three steps.

\smallskip
\noindent
{\bf 1.} We first show that the solution $m(x,t)$ of the Cauchy problem $\eqref{mmaineq2}$ 
can be expressed by its value $m(\eta,t_1)$ as in $\eqref{umeq}$. 
Obviously, $m(x,t)$ satisfies the following linear transport equation
\begin{equation}\label{scalareq}
\begin{cases}
m_t+um_x=0,\\[1mm]
m|_{t=t_1}=m(\eta,t_1).
\end{cases}
\end{equation}

Similar to Definition $\ref{weakdef}$, by using the Lebesgue-Stieltjes integral,
the definition of weak solutions of the Cauchy problem $\eqref{scalareq}$ 
is given as follows ({\it cf.} \cite{huang2001well}):
\begin{Def}\label{sweaklem}
Let $m(x,t) \in [0,M]$ is increasing in $x$ and the measure $m_x$ is weakly continuous with respect to $t$. 
Then, $m(x,t)$ is called to be a weak solution of $\eqref{scalareq}$ if, 
for any $\psi \in C_c^{\infty}(\mathbb{R}^2_+)$, 
\begin{equation}\label{psimu}
\iint \psi_tm \,{\rm d}x {\rm d}t-\iint \psi u \,{\rm d}m {\rm d}t=0.
\end{equation}
\end{Def}

\smallskip
Denote $U:=||u(x,t)||_{L^{\infty}}$.
Let $u_{\epsilon}:=u*j_{\epsilon}$ for $u$ satisfying Oleinik condition $\eqref{ux2ux1}$, 
where $j_{\epsilon}$ for $\epsilon>0$ is a standard mollifier.
Then
\begin{equation}\label{uepsilon}
u_{\epsilon} \leq U,\quad \quad u_{\epsilon x} \leq \frac{1}{t}.
\end{equation}
Furthermore, let $x(t):=X_{t_1}^{\epsilon}(\xi ,t)$ solve the following ODE:
\begin{equation}\label{dxdt}
\frac{{\rm d}x(t)}{{\rm d}t}=u_{\epsilon},\quad x(t_1)=\xi\in\mathbb{R}.
\end{equation}

Then, we have the following lemma as presented in \cite{WD1997}.

\begin{Lem}\label{reflem}
The characteristics defined by $\eqref{dxdt}$ satisfy the following statements{\rm:}
\begin{itemize}
\item[(i)]  There exists a subsequence $\{\epsilon_i\}$ of $\epsilon>0$ such that, 
for any $\xi\in\mathbb{R}$,
\begin{equation}\label{XXt1}
\lim_{\epsilon_i \rightarrow 0}X_{t_1}^{\epsilon_i}(\xi,t)=X_{t_1}(\xi,t).
\end{equation}
 Furthermore, $X_{t_1}(\xi,t)$ is Lipschitz continuous with respect to $t$.

\vspace{2pt}
\item[(ii)] If $X_{t_1}(\xi_1,t_0)=X_{t_1}(\xi_2,t_0)$ holds for some $\xi_1<\xi_2$ and $t_0 \geq t_1>0$, then 
$$X_{t_1}(\xi_1,t)=X_{t_1}(\xi_2,t)\qquad {\rm for\ any}\ t\in [t_0,\infty).$$ 
Furthermore, $X_{t_1}(\xi,t)$ for $\xi \in \mathbb{R}$ cover the strip $\{t\geq t_1\}$.

\vspace{2pt}
\item[(iii)] If $0<t_2<t_1$, then $X_{t_2}(\xi,t)=X_{t_1}(X_{t_2}(\xi,t_1),t)$ holds for any $t\geq t_1$.

\vspace{2pt}
\item[(iv)]  Let $X(\xi,t)$ defined as $\eqref{XXt1}$ with $t_1=0$, and denote
\begin{equation}\label{mathcalU}
\mathcal{U}=\{\xi\,:\, \exists t>0\ s.t.\ X(\xi{-}0,t) \neq X(\xi{+}0,t)\}.
\end{equation}
Then, for any $\xi \in \mathbb{R} \backslash \mathcal{U}$ and almost all $t>0$, 
\begin{align}\label{6.1}
\frac{\partial X(\xi,t)}{\partial t}=u(X(\xi,t),t).
\end{align}

\vspace{2pt}
\item[(v)] Let $\xi_{t_1}(x,t)$ and $\eta_{t_1}(x,t)$ be given by
\begin{equation}\label{xietat1}
\xi_{t_1}(x,t):=\sup \{\xi:X_{t_1}(\xi,t)<x\}\qquad \eta_{t_1}(x,t):=\inf \{\xi:X_{t_1}(\xi,t)>x\}.
\end{equation}
Then the set
$\varGamma_{t_1}=\{(x,t):\xi_{t_1}(x,t) \neq \eta_{t_1}(x,t), t\geq t_1\}$ consists of at most countable Lipschitz continuous curves. Such curves are called to be $\delta$-shock waves of $u(x,t)$. Furthermore, 
\begin{align}
&\xi_{t_1}(x{-},t)=\xi_{t_1}(x,t),\quad \eta_{t_1}(x,t)=\xi_{t_1}(x{+},t)=\eta_{t_1}(x{+},t),\label{6.2a}\\[1mm]
&X(\xi(x,t){-},t)\leq x \leq X(\xi(x,t){+},t). \label{6.2b}
\end{align}
\end{itemize}
\end{Lem}

Similar to $F(y;x,t)$, we define
\begin{align}
H_{t_1}(y;x,t)=\int_{-\infty}^{y-} X_{t_1}(\eta,t)-x\,{\rm d}m(\eta,t_1),
\end{align}
which, by \eqref{xietat1}, yields
$$
\min_y H_{t_1}(y;x,t)=H(\xi_{t_1}(x,t);x,t).
$$
Define 
\begin{align}
w(x,t)=\int_{-\infty}^{\xi_{t_1}(x,t)-}\,{\rm d}m(\eta,t_1), \quad
q_2(x,t)=\int_{-\infty}^{\xi_{t_1}(x,t)-}u(X_{t_1}(\eta,t),t)\,{\rm d}m(\eta,t_1),
\end{align}
Similar to the Lemma 2.5 in Wang--Ding \cite{WD1997}, we have $q_{2x}=u w_x$ in the sense of Radon-Nikodym derivative.
On the other hand, it is easy to see that Lemma \ref{mqnulem} still holds here, then 
\begin{align*}
\frac{\partial}{\partial x}\min_y H_{t_1}(y;x,t)=-w(x,t), \quad 
\frac{\partial}{\partial t}\min_y H_{t_1}(y;x,t)=q_2(x,t),
\end{align*}
which implies $w(x,t)$ is a weak solution of \eqref{scalareq}.
Since the uniqueness of \eqref{scalareq} is proved in Theorem 3.1 of Huang-Wang \cite{huang2001well}, then $m(x,t)=w(x,t)$ almost everywhere in $t\geq t_1$, and we have the following proposition.

\begin{Pro}\label{scalarthm}
Suppose that $m(\eta,t_1) \in [0,M]$ is nondecreasing in $\eta$, $u(x,t)$ is bounded and satisfies the Oleinik condition \eqref{ux2ux1}. 
Then,  the Cauchy problem $\eqref{scalareq}$ admits a unique weak solution,  
which is given by
\begin{equation}\label{umeq}
m(x,t)=\int_{{-}\infty}^{\xi_{t_1}(x,t){-}} {\rm d}m(\eta,t_1),
\end{equation}
where $\xi_{t_1}(x,t)$ is given by $\eqref{xietat1}$.
\end{Pro}

\smallskip
\noindent
{\bf 2.} We now prove that the $q(x,t)$ of the Cauchy problem $\eqref{mmaineq2}$ can be expressed by 
\begin{equation}\label{uqeq}
q(x,t)=\int_{{-}\infty}^{\xi_{t_1}(x,t){-}}u(\eta,t_1) e^{-\frac{t-t_1}{\tau}}{+}\widetilde{m}(\eta,t_1)(\tau e^{-\frac{t-t_1}{\tau}}-\tau) \,{\rm d}m(\eta,t_1).
\end{equation}

\begin{Lem}\label{uRNlem}
 Let $m=m(x,t)$ and $q=q(x,t)$ be given by $\eqref{umeq}$ and $\eqref{uqeq}$, respectively.
Then, $q_x=u(x,t)m_x$ in the sense of Radon--Nikodym derivatives. 
Furthermore, for any $x_1, x_2\in\mathbb{R}$, $t\geq t_1$ and $s\in [t_1,t]$, 
\begin{align}\label{Xt1u}
&\ \int_{\xi_{t_1}(x_1,t){\pm}}^{\xi_{t_1}(x_2,t){\pm}} X_{t_1}(\eta,s)\,{\rm d}m(\eta,t_1)\nonumber\\[1mm]
&=\int_{\xi_{t_1}(x_1,t){\pm}}^{\xi_{t_1}(x_2,t){\pm}}\eta{+}u(\eta,t_1)(\tau{-}\tau e^{-\frac{s-t_1}{\tau}})
{+}\widetilde{m}(\eta,t_1)\big(\tau^2{-}\tau^2 e^{-\frac{s-t_1}{\tau}}{-}\tau(s{-}t_1)\big)\,{\rm d}m(\eta,t_1).
\end{align}
\end{Lem}

\begin{proof}
{\bf 1.} We first prove \eqref{Xt1u}.
For any $x_1, x_2\in\mathbb{R}$ and $t\geq t_1$, we denote
$$x_1(\zeta):=X_{t_1}(\xi_{t_1}(x_1,t),\zeta),\quad x_2(\zeta):=X_{t_1}(\xi_{t_1}(x_2,t),\zeta)
\qquad {\rm for}\ \zeta\in[t_1,t].$$
Then, $\xi_1:=\xi_{t_1}(x_1,\zeta)$ and $\xi_2:=\xi_{t_1}(x_2,\zeta)$ 
keep constant along curves $x_1(\zeta)$ and $x_2(\zeta)$ respectively.
Let $\phi_{1\epsilon}(x),\phi_{2\epsilon}(x) \in C^{\infty}(\mathbb{R})$ satisfy
$$\phi_{1\epsilon}(x)
=\begin{cases}
1\quad {\rm if}\ x \leq {-}\epsilon,\\[1mm] 
0\quad {\rm if}\ x \geq 0,
\end{cases}
\qquad
\phi_{2\epsilon}(x)=
\begin{cases}
1\quad {\rm if}\  x \leq 0,\\[1mm]
0\quad {\rm if}\  x \geq \epsilon,
\end{cases}
$$
and for any $\varphi(\zeta)\in C_0^{\infty}[t_1,t)$, we set
$$\psi_{\epsilon}=\phi_{\epsilon}(x,\zeta)\varphi(\zeta), \qquad 
\phi_{\epsilon}(x,\zeta)=\phi_{2\epsilon}(x-x_2(\zeta))-\phi_{1\epsilon}(x-x_1(\zeta)).$$
Then, by $\eqref{weakeq}_2$, we have 
$$\iint \psi_{\epsilon\zeta}u+\psi_{\epsilon x}u^2 \,{\rm d}m {\rm d}\zeta +\int \psi_{\epsilon}(\eta,t_1)u(\eta,t_1) \,{\rm d}m(\eta,t_1)=\iint (\widetilde{m}+\frac{u}{\tau})\psi_{\epsilon} \,{\rm d}m {\rm d}\zeta,$$
which, by a simple calculation, implies
\begin{align}\label{phiepeq}
&\ \iint \varphi_{\zeta} \phi_{\epsilon}u-(\widetilde{m}{+}\frac{u}{\tau})\varphi \phi_{\epsilon} \,{\rm d}m {\rm d}\zeta +\int \varphi(t_1)\phi_{\epsilon}(\eta,t_1)u(\eta,t_1) \,{\rm d}m(\eta,t_1)\nonumber\\[1mm]
& =-\int_{t_1}^t \int_{x_1(\zeta){-}\epsilon}^{x_1(\zeta)}\varphi u \phi_{1\epsilon}'(x_1'(\zeta)-u) \,{\rm d}m {\rm d}\zeta
+\int_{t_1}^t \int_{x_2(\zeta)}^{x_2(\zeta){+}\epsilon}\varphi u \phi_{2\epsilon}'(x_2'(\zeta)-u) \,{\rm d}m {\rm d}\zeta.
\end{align}
It follows from \eqref{6.1} that the right hand of $\eqref{phiepeq}$ tends to zero as $\epsilon\rightarrow 0$, 
by letting $\epsilon \rightarrow 0$ in $\eqref{phiepeq}$, we obtain
\begin{equation}
\label{phiueq}
\int_{t_1}^t \int_{x_1(\zeta){-}}^{x_2(\zeta){+}}\varphi_{\zeta}u-(\widetilde{m}{+}\frac{u}{\tau})\varphi \,{\rm d}m {\rm d}\zeta +\int_{\xi_1{-}}^{\xi_2{+}}\varphi(t_1)u(\eta,t_1) \,{\rm d}m(\eta,t_1)=0.
\end{equation}

For any $s\in [t_1,t)$ and $\delta\in (0,t-s)$, by taking $\varphi(\zeta)=\varphi_{\delta}(\zeta)$ defined by 
\begin{equation*}
\varphi_{\delta}(\zeta)=
\begin{cases}
1\quad {\rm if}\ t_1 \leq \zeta \leq s,\\[1mm]
0\quad {\rm if}\ \zeta \geq s+\delta.
\end{cases}
\end{equation*}
into $\eqref{phiueq}$ and letting $\delta \rightarrow 0$, we obtain
\begin{equation}\label{Xpeq}
\int_{\xi_1{-}}^{\xi_2{+}}\frac{1}{\tau}(X_{t_1}(\eta,s)-X_{t_1}(\eta,t_1))+u(X_{t_1}(\eta,s),s)-u(X_{t_1}(\eta,t_1),t_1)+(s-t_1)\widetilde{m}(\eta,t_1)\,{\rm d}m(\eta,t_1)=0,
\end{equation}
where we used the fact that,
$$\int_{x_1(s){-}}^{x_2(s){+}}\widetilde{m}(x,s) \,{\rm d}m(x,s)=\int_{\xi_1{-}}^{\xi_2{+}}\widetilde{m}(\eta,t_1)\,{\rm d}m(\eta,t_1).$$
Let 
\begin{align*}
h(s)=\int_{\xi_1-}^{\xi_2+} X_{t_1}(\eta,s)\,{\rm d}m(\eta,t_1),
\end{align*}
then \eqref{Xpeq} transforms into
\begin{align}
\frac{1}{\tau}(h(s)-h(t_1))+h'(s)-h'(t_1)+(s-t_1)\int_{\xi_1-}^{\xi_2+}\widetilde{m}(\eta,t_1)\,{\rm d}m(\eta,t_1)=0.
\end{align}
Solving this ODE implies \eqref{Xt1u} with the lower limit being $\xi_1{-}$ and upper limit being $\xi_2{+}$. Similarly, it can be checked that $\eqref{Xt1u}$ holds for other cases.

\smallskip
\noindent
{\bf 2.} Now we prove $q_x=u(x,t)m_x$ in the sense of Radon-Nikodym sense.
For any $(x,t)$, it follows from Lemma \ref{reflem} (ii), (v) that 
\begin{align}
X_{t_1}(\xi_{t_1}(x,t),t)=X_{t_1}(\xi_{t_1}(x{-},t),t)=X_{t_1}(\eta_{t_1}(x,t),t)=X_{t_1}(\xi_{t_1}(x{+},t),t)=x.
\end{align}
Then for $x_1<x<x_2$, by \eqref{Xt1u} and \eqref{6.1},
\begin{align}
&\lim_{x_1,x_2\rightarrow x{\pm}}\frac{q(x_2,t)-q(x_1,t)}{m(x_2,t)-m(x_1,t)}\nonumber\\[1mm]
&=\lim_{x_1,x_2\rightarrow x{\pm}}\frac{\int_{\xi_{t_1}(x_1,t){\pm}}^{\xi_{t_1}(x_2,t){\pm}}u(\eta,t_1) e^{-\frac{t-t_1}{\tau}}{+}\widetilde{m}(\eta,t_1)(\tau e^{-\frac{t-t_1}{\tau}}-\tau)\,{\rm d}m(\eta,t_1)}
{\int_{\xi_{t_1}(x_1,t){\pm}}^{\xi_{t_1}(x_2,t){\pm}}\,{\rm d}m(\eta,t_1)}\nonumber\\[1mm]
&=\lim_{x_1,x_2\rightarrow x{\pm}} \frac{\int_{\xi_{t_1}(x_1,t){\pm}}^{\xi_{t_1}(x_2,t){\pm}} X'_{t_1}(\eta,t)\,{\rm d}m(\eta,t_1)}{\int_{\xi_{t_1}(x_1,t){\pm}}^{\xi_{t_1}(x_2,t){\pm}}\,{\rm d}m(\eta,t_1)}=u(x,t).
\end{align}

\end{proof}

\subsection{Formula of entropy solution with respect to intial data.}

For any entropy solution $(m,u)=(m,\frac{{\rm d}q}{{\rm d}m})$, we have the formulas \eqref{umeq} and \eqref{uqeq} depending on the values of entropy solution in $t=t_1$.
Noticing that $t_1>0$ is arbitrary, so that by letting $t_1\rightarrow 0{+}$ we can obtain the formula of entropy solutions with respect to initial data.

\smallskip
\noindent
{\bf 1.}
For any point $(x_0,t_0)$, we choose a decreasing sequence $\{t_i: i=1,2 \cdots \}$ such that $t_i \rightarrow 0{+}$ as $i \rightarrow \infty$.
By Lemma \ref{reflem}(iii), 
$$
\Big|\frac{\xi_{t_i}(x_0,t_0)-\xi_{t_j}(x_0,t_0)}{t_i-t_j}\Big|\le U,
$$
then $\{\xi_{t_i}(x_0,t_0): i=1,2 \cdots\}$ is a Cauchy sequence and assume 
\begin{equation}\label{xix0t0}
\xi_{t_i}(x_0,t_0)\rightarrow \xi(x_0,t_0) \qquad\text{as $i \rightarrow \infty$}. 
\end{equation}
As $t_i\rightarrow 0+$,  $X_{t_i}(\xi_{t_i}(x_0,t_0),t)$ form a curve originated from $\xi(x_0,t_0)$ and terminated at $(x_0,t_0)$. We denote it as $X(\xi(x_0,t_0);x_0,t_0)$.
It follows from the boundedness of $u(x,t)$ and the weak continuity of $u(x,t) m_x$ that the limit $\lim_{t_i\rightarrow 0+}u(\xi_{t_i}(x_0,t_0),t_i)$ exists, and denote it as $c(\xi(x_0,t_0);x_0,t_0)$.

\begin{Lem}\label{xicon}
For any point $(x_0,t_0) \in \mathbb{R}\times \mathbb{R}^+$, 
\begin{itemize}
\item [(i)] If $[m_0(\xi(x_0,t_0))]=m_0(\xi(x_0,t_0){+}){-}m_0(\xi(x_0,t_0){-}){=}0$,
or $[m_0(\xi(x_0,t_0))]{>}0$ with $\xi(x_0,t_0)\in \mathbb{R} \backslash \mathcal{U}$ defined in \eqref{mathcalU},
then
\begin{align}
&m(x_0,t_0)=m_0(\xi(x_0,t_0)=\int_{{-}\infty}^{\xi(x_0,t_0){-}}\,{\rm d}m_0(\eta),\label{mx0t0}\\[1mm]
&q(x_0,t_0)=\int_{{-}\infty}^{\xi(x_0,t_0){-}}u_0(\eta)e^{-\frac{t}{\tau}}{+}\widetilde{m}_0(\eta)(\tau e^{-\frac{t}{\tau}}{-}\tau)\,{\rm d}m_0(\eta).\label{qx0t0}
\end{align}
\item[(ii)] If $\xi(x_0,t_0) \in \mathcal{U}$ with $[m_0(\xi(x_0,t_0))]{>}0$, then
\begin{align}
&m(x_0,t_0)=
\begin{cases}
\int_{{-}\infty}^{\xi(x_0,t_0){-}} \,{\rm d}m_0(\eta)
\quad {\rm if} \ c(\xi(x_0,t_0);x_0,t_0)\leq u_0(\xi(x_0,t_0)),\\[2mm]
\int_{{-}\infty}^{\xi(x_0,t_0){+}} \,{\rm d}m_0(\eta)
 \quad {\rm if} \ c(\xi(x_0,t_0);x_0,t_0)> u_0(\xi(x_0,t_0)).
\end{cases}\label{mx0t02}\\[1mm]
&q(x_0,t_0)=\nonumber\\[1mm]
&\begin{cases}
\int_{{-}\infty}^{\xi(x_0,t_0){-}}u_0(\eta)e^{-\frac{t}{\tau}}{+}\widetilde{m}_0(\eta)(\tau e^{-\frac{t}{\tau}}{-}\tau) \,{\rm d}m_0(\eta)
\quad {\rm if} \ c(\xi(x_0,t_0);x_0,t_0){\leq} u_0(\xi(x_0,t_0)),\\[2mm]
\int_{{-}\infty}^{\xi(x_0,t_0){+}}u_0(\eta)e^{-\frac{t}{\tau}}{+}\widetilde{m}_0(\eta)(\tau e^{-\frac{t}{\tau}}{-}\tau) \,{\rm d}m_0(\eta)
\quad {\rm if} \ c(\xi(x_0,t_0);x_0,t_0){>} u_0(\xi(x_0,t_0)).
\end{cases}\label{qx0t02}
\end{align}

\end{itemize}
\end{Lem}

\begin{proof}
{\bf 1.} First we prove the formula of $m(x_0,t_0)$. The proof is divided into two parts accordingly.

\smallskip
\noindent
{\bf (a).}
For any point $(x_0,t_0)$,  it follows from Lemma \ref{reflem} (ii) and (v) that $\xi_{t_i}(x,t)\equiv \xi_{t_i}(x_0,t_0)$ for $(x,t)\in X_{t_i}(\xi_{t_i}(x_0,t_0),t)$, 
which, by \eqref{umeq}, implies $m(x,t)\equiv m(x_0,t_0)=m(\xi_{t_i}(x_0,t_0),t_i)$ along $X_{t_i}(\xi_{t_i}(x_0,t_0),t)$. 

Since $m(x,t)\equiv m(x_0,t_0)$ along $X(\xi(x_0,t_0);x_0,t_0)$, if $[m_0(\xi(x_0,t_0))]=0$, by the weak continuity of $\rho=m_x$ at $t=0$, \eqref{mx0t0} holds.
If $[m_0(\xi(x_0,t_0))]>0$ with $\xi(x_0,t_0)\in \mathbb{R}\backslash \mathcal{U}$, by the left continuity of $m(x,t)$ in \eqref{umeq}, \eqref{mx0t0} also holds.

\smallskip
\noindent
{\bf (b).} 
Without loss of generality, assume that $\xi(x_0,t_0)=0$.
We denote $x_1(t)=X(0{-},t)$ and $x_2(t)=X(0{+},t)$ for $t\in [0,t_0].$
Then for any $\bar{x} \in [x_1(t_0),x_2(t_0)]$, $\xi(\bar{x},t_0)=0$ and $m(x,t)\equiv m(\bar{x},t_0)$ for $(x,t)\in X(0;\bar{x},t_0)$.

Let $c_1=\lim_{t \rightarrow 0{+}}x_1'(t)$ and $c_2=\lim_{t \rightarrow 0{+}}x_2'(t)$.
We define 
\begin{equation}\label{mc1c2}
m(c)=
\begin{cases}
\lim_{t \rightarrow 0{+}}m(x_1(t){-},t) &\quad {\rm if}\ c <c_1,\\[1mm]
m(x,t)& \quad {\rm if}\  c_1<c<c_2,\\[1mm]
\lim_{t \rightarrow 0{+}}m(x_2(t){+},t) &\quad {\rm if}\ c >c_2.
\end{cases} 
\end{equation}

By choosing smooth function $\varphi_{\epsilon}$ such that 
$\varphi_{\epsilon}=1$ in $(x_1(t){-}\epsilon,x_2(t){+}\epsilon)$ and $\varphi_{\epsilon}=0$
in $({-}\infty,x_1(t)-2\epsilon)\cup (x_2(t)+2\epsilon,\infty)$, we have 
\begin{align}\label{varphieq}
\int \varphi_{\epsilon} \,{\rm d}m(x,t)
& =\int_{x_1(t)-2\epsilon}^{x_1(t){-}0}\varphi_{\epsilon} \,{\rm d}m(x,t)+[m(x_1(t),t)]
\nonumber\\[1mm]
&\quad \,\, +\int_{c_1{+}0}^{c_2{-}0} \,{\rm d}m(c) +[m(x_2(t),t)]
+\int_{x_2(t){+}0}^{x_2(t)+2\epsilon} \varphi_{\epsilon} \,{\rm d}m(x,t),
\end{align}
where $[m(x,t)]=m(x{+},t)-m(x{-},t)$.

Since the first and fifth term of the right-hand side of $\eqref{varphieq}$ tend to zero as $\epsilon \rightarrow 0$ 
and $m_x$ is weakly continuous at $t=0$, by let $t \rightarrow 0{+}$,  
\begin{align}\label{m0diseq}
[m_0(0)]
&=\lim_{t \rightarrow 0{+}}[m(x_1(t),t)]+\int_{c_1{+}0}^{c_2{-}0} \,{\rm d}m(c)
+\lim_{t \rightarrow 0{+}}[m(x_2(t),t)]
=\int_{c_1{-}0}^{c_2{+}0} \,{\rm d}m(c).
\end{align}
Similarly, by the weak continuity of momentum $u m_x$ and energy $u^2m_x$ at $t=0$,
\begin{equation}\label{energyeq}
[m_0(0)]u_0(0)=\int_{c_1{-}0}^{c_2{+}0}c \,{\rm d}m(c),\qquad [m_0(0)]u_0(0)^2=\int_{c_1{-}0}^{c_2{+}0}c^2 \,{\rm d}m(c).
\end{equation}
Combining $\eqref{m0diseq}$--$\eqref{energyeq}$, we have 
\begin{equation}\label{cu0eq}
 \int_{c_1{-}0}^{c_2{+}0}(c-u_0(0))^2 \,{\rm d}m(c)=0.
\end{equation}
Since $m(c)$ is increasing in $c$, then $m(c)$ must be discontinuous at $c=u_0(0)$. 
From $\eqref{energyeq}$, $c_1 \leq u_0(0) \leq c_2$ so that 
$m(c)$ takes two values determined by $c=u_0(0)$, which yields \eqref{mx0t02}.

\smallskip
\noindent
{\bf 2.} For the formula of $q(x_0,t_0)$, notice that $m_x$ and $u m_x$ are weakly continuous at $t=0$, then as $t_1\rightarrow 0$, 
\begin{align}
&u(\eta,t_1)\, {\rm d}m(\eta,t_1)\rightharpoonup u_0(\eta)\,{\rm d}m_0(\eta), \label{6.3a}\\[1mm]
&\widetilde{m}(\eta,t_1)\, {\rm d}m(\eta,t_1)={\rm d}\big(\frac{1}{2}m^2(\eta,t_1){-}\frac{M}{2}m(\eta,t_1)\big) \rightharpoonup {\rm d}\big(\frac{1}{2}m_0^2(\eta){-}\frac{M}{2}m_0(\eta)\big)=\widetilde{m}_0(\eta)\,{\rm d}m_0(\eta).\label{6.3b}
\end{align}
Similar argument as $m(x_0,t_0)$ yields \eqref{qx0t0} and \eqref{qx0t02}.
\end{proof}

Since $t_1>0$ is arbitrary, then Lemma \ref{uRNlem} still holds and $q_x=u(x,t)m_x$ for any $t>0$ in the sense of Radon-Nikodym derivatives, so that $(m(x,t),q(x,t))$ is an entropy solution.

\subsection{Uniqueness of entropy solution.}
Since any entropy solution $(m(x,t),q(x,t))$ can be expressed by \eqref{mx0t0}--\eqref{qx0t02}, if we can show that $\xi(x_0,t_0)\in S(x,t)$ is a minimum point of $F(y;x,t)$, then by Lemma \ref{sigmasol}, the uniqueness of entropy solutions is proved.

First we define a generalized potential originated from $t_1$:
\begin{align}
F_{t_1}(y;x,t){=}\int_{-\infty}^{y-} \eta{+}u(\eta,t_1)(\tau-\tau e^{-\frac{t-t_1}{\tau}}){+}\widetilde{m}(\eta,t_1)(\tau^2{-}\tau^2 e^{-\frac{t-t_1}{\tau}}{-}\tau(t-t_1)){-}x \,{\rm d}m(\eta,t_1).
\end{align}

\begin{Lem}
For any $(x,t)$ with $t>t_1$, 
\begin{equation}\label{minFt1}
\min_y F_{t_1}(y;x,t)=F_{t_1}(\xi_{t_1}(x,t);x,t).
\end{equation}
By letting $t_1\rightarrow 0$, then 
\begin{align}\label{6.3}
\xi_{t_1}(x,t)\rightarrow \xi(x,t)\in S(x,t),
\end{align}
where $\xi(x,t)$, $S(x,t)$ is given in \eqref{xix0t0} and \eqref{nuSeq}, respectively. 
\end{Lem}
 
\begin{proof}
It follows from Lemma $\ref{uRNlem}$ that
\begin{align*}
&\ F_{t_1}(y;x,t)-F_{t_1}(\xi_{t_1}(x,t);x,t)\\[1mm]
&=\int^{y{-}}_{\xi_{t_1}(x,t){-}} \eta{+}u(\eta,t_1)(\tau-\tau e^{-\frac{t-t_1}{\tau}}){+}\widetilde{m}(\eta,t_1)(\tau^2{-}\tau^2 e^{-\frac{t-t_1}{\tau}}{-}\tau(t-t_1)){-}x \,{\rm d}m(\eta,t_1)\\[1mm]
&=\int_{\xi_{t_1}(x,t){-}}^{y{-}} X_{t_1}(\eta,t)-x \,{\rm d}m(\eta,t_1). 
\end{align*}
Then, $\eqref{minFt1}$ follows from \eqref{6.2b} and the facts that $X_{t_1}(\eta,t)$ is increasing in $\eta$. 

Letting $t_1\rightarrow 0$ in \eqref{minFt1}, it follows from \eqref{6.3a}--\eqref{6.3b} and the left continuity of $F(\cdot;x,t)$ that 
\begin{align*}
F_{t_1}(\xi_{t_1}(x,t);x,t)\rightarrow F(\xi(x,t);x,t) &\leq F(y;x,t), \\[1mm]
{\rm or}\ F_{t_1}(\xi_{t_1}(x,t);x,t)\rightarrow F(\xi(x,t){+};x,t) &\leq F(y;x,t),
\end{align*}
which implies \eqref{6.3}.

\end{proof}

\smallskip
Up to now, we have completed the proof of Theorem $\ref{UniThm}$.



\medskip
\begin{thebibliography}{99}

\bibitem{ABR2000}
G. Alì, D. Bini, and S. Rionero,
{\it Global existence and relaxation limit for smooth solutions to the Euler--Poisson model for semiconductors},
SIAM J. Math. Anal. {\bf 32(3)} (2000), 572-587.


\bibitem{B1994}
F. Bouchut,
{\it On zero pressure gas dynamics},
Advances in kinetic theory and computing: selected papers
(1994), 171-190.

\bibitem{BJ1999}
F. Bouchut and F. James,
{\it Duality solutions for pressureless gases, monotone scalar conservation laws, and uniqueness},
Commun. Partial Differ. Equ.
{\bf 24(11-12)} (1999), 2173-2189.


\bibitem{BG2013}
Y. Brenier, W. Gangbo, G. Savar{\'e}, and M. Westdickenberg,
{\it Sticky particle dynamics with interactions},
J. Math. Pures Appl. 
{\bf 99(5)} (2013), 577-617.

\bibitem{BG1998}
Y. Brenier and E. Grenier,
{\it Sticky particles and scalar conservation laws},
SIAM J. Numer. Anal.
{\bf 35(6)} (1998), 2317-2328.

\bibitem{CHT2025}
G. W. Cao, F. M. Huang, and G. R. Tang,
{\it Formula for Entropy solutions for the $1$--D Pressureless Euler--Poisson System: 
Well-posedness of Entropy Solutions and the Asymptotic Behaviors},
to appear.


\bibitem{carter2000lebesgue}
M. Carter and B. Van Brunt,
{\it The Lebesgue-Stieltjes integral},
Springer New York, 2000.

\bibitem{CSW2015}
F. Cavalletti, M. Sedjro, and M. Westdickenberg,
{\it A simple proof of global existence for the 1D pressureless gas dynamics equations},
SIAM J. Math. Anal.
{\bf 47(1)} (2015), 66-79.


\bibitem{E1996}
W. N. E, Y. G. Rykov, and Y. G. Sinai,
{\it Generalized variational principles, global weak solutions and behavior with random initial data for systems of conservation laws arising in adhesion particle dynamics},
Comm. Math. Phys.
{\bf 177} (1996), 349-380.

\bibitem{HHW2014}
S. Y. Ha, F. M. Huang, and Y. Wang,
{\it A global unique solvability of entropic weak solution to the one-dimensional pressureless Euler system with a flocking dissipation},
J. Differential Equations
{\bf 257(5)} (2014), 1333-1371.

\bibitem{HZ2000}
L. Hsiao and K. J. Zhang,
{\it The relaxation of the hydrodynamic model for semiconductors to the drift–diffusion equations},
J. Differential Equations
{\bf 165(2)} (2000), 315-354.

\bibitem{huang2001well}
F. M. Huang and Z. Wang,
{\it Well posedness for pressureless flow},
Comm. Math. Phys.
{\bf 222} (2001), 117-146.

\bibitem{H2020}
R. Hynd,
{\it A trajectory map for the pressureless Euler equations},
Trans. Amer. Math. Soc.
{\bf 373(10)} (2000), 6777-6815.

\bibitem{Jin2016}
C. Y. Jin,
{\it Existence and uniqueness of entropy solution to pressureless Euler system with a flocking dissipation},
Acta Math. Sci.
{\bf 36(5)} (2016), 1262-1284.


\bibitem{JA2009}
A. Jüngel,
{\it Transport equations for semiconductors},
Springer Vol. 773 (2009).

\bibitem{LM1999}
 C. Lattanzio and P. Marcati,
 {\it The relaxation to the drift-diffusion system for the 3-$ D $ isentropic Euler-Poisson model for semiconductors},
 Discrete Contin. Dyn. Syst. {\bf 5(2)} (1999), 449-455.

\bibitem{LT2023}
T. M. Leslie and C. Tan,
{\it Sticky particle Cucker–Smale dynamics and the entropic selection principle for the 1D Euler-alignment system},
Comm. Partial Differential Equations
{\bf 48(5)} (2023), 753-791.

\bibitem{LPZ2021}
Y. Li, Y. J. Peng, and L. Zhao,
{\it Convergence rates in zero-relaxation limits for Euler-Maxwell and Euler-Poisson systems},
J. Math. Pures Appl. 
{\bf 154} (2021), 185-211.

\bibitem{Liu1987}
T. P. Liu,
{\it Hyperbolic conservation laws with relaxation},
Comm. Math. Phys.
{\bf 108(1)} (1987), 153-175.

\bibitem{MN1995}
P. Marcati and R. Natalini,
{\it Weak solutions to a hydrodynamic model for semiconductors and relaxation to the drift-diffusion equation},
Arch. Rational Mech. Anal. {\bf 129} (1995), 129-145.

\bibitem{MR2000}
P. Marcati and B. Rubino,
{\it Hyperbolic to parabolic relaxation theory for quasilinear first order systems},
J. Differential Equations {\bf 162(2)} (2000), 359-399.


\bibitem{Na1997}
R. Natalini,
{\it Recent results on hyperbolic relaxation problems},
Analysis of systems of conservation laws, Aachen, (1997).

\bibitem{NS2009}
L. Natile and G. Savar{\'e},
{\it A Wasserstein approach to the one-dimensional sticky particle system},
SIAM J. Math. Anal.
{\bf 41(4)} (2009), 1340-1365.

\bibitem{NT2008}
T. Nguyen and A. Tudorascu,
{\it Pressureless Euler/Euler–Poisson systems via adhesion dynamics and scalar conservation laws},
SIAM J. Math. Anal.
{\bf 40(2)} (2008), 754-775.

\bibitem{NT2015}
T. Nguyen and A. Tudorascu,
{\it One-dimensional pressureless gas systems with or without viscosity},
Comm. Partial Differential Equations
{\bf 40(9)} (2015), 1619-1665.

\bibitem{vol1967spaces}
A. I. Volpert, 
{\it The spaces $BV$ and quasilinear equations},
Math. USSR-Sb.
{\bf 2} (1967), 255-302.

\bibitem{Wang2001}
D. H. Wang,
{\it Global solutions and relaxation limits of Euler-Poisson equations},
Z. Angew. Math. Phys. {\bf 52} (2001), 620-630.

\bibitem{WD1997}
Z. Wang and X. Q. Ding,
{\it Uniqueness of generalized solution for the Cauchy problem of transportation equations},
Acta Math. Sci.
{\bf 17(3)} (1997), 341-352.


\bibitem{GBW}
G. B. Whitham, 
{\it Linear and Nonlinear Waves},
Wiley, New York, 1974.

\bibitem{Z1970}
Y. B. Zel'Dovich,
{\it Gravitational instability: An approximate theory for large density perturbations},
Astronomy and Astrophysics
{\bf 5} (1970), 84-89.

\bibitem{ZQY2024}
R. Zhao, A. Qu, and H. Yuan,
{\it Probability measures on path space for rectilinear damped pressureless Euler-Poisson equations},
J. Differential Equations
{\bf 387} (2024), 152-199.

\end{thebibliography}
\end{document}